\documentclass[reqno, 11pt]{amsart}
\usepackage[percent]{overpic}
\usepackage{calc,graphicx,amsfonts,amsthm,amscd,amsmath,amssymb,enumerate,dsfont,mathrsfs,paralist,bbm,tikz}
\usetikzlibrary{patterns,arrows}
\usepackage{subfig}
\usepackage{pdfsync}
\usepackage{stmaryrd}
\usepackage[numeric,initials,nobysame]{amsrefs}
\usepackage{marginnote}
\usepackage{hyperref}
\usepackage[bottom]{footmisc}

\usepackage{booktabs}
\usepackage{color}

\newcommand{\ie}{\hbox{\it{i.e.}, }}

\reversemarginpar
\newlength\fullwidth
\numberwithin{equation}{section}

\DeclareMathSymbol{\leqslant}{\mathalpha}{AMSa}{"36} 
\DeclareMathSymbol{\geqslant}{\mathalpha}{AMSa}{"3E} 
\DeclareMathSymbol{\eset}{\mathalpha}{AMSb}{"3F}     
\def\<{\langle}
\def\>{\rangle}

\renewcommand{\leq}{\;\leqslant\;}                   
\renewcommand{\geq}{\;\geqslant\;}                   
\renewcommand{\le}{\,\leqslant\,}                   
\renewcommand{\ge}{\,\geqslant\,}                   
\renewcommand{\b}{\beta}
\def\0{\textbf{0}}

\newcommand{\eps}{\epsilon}

\def\1{\ifmmode {1\hskip -3pt \rm{I}} \else {\hbox {$1\hskip -3pt \rm{I}$}}\fi}

\newcommand{\var}{\operatorname{Var}}

\newcommand{\tc}{\thinspace |\thinspace}
\newcommand{\id}{\mathbbm{1}}

\newcommand{\trel}{T_{\rm rel}}

\renewcommand{\b}{\beta}
\renewcommand{\l}{\lambda}
\renewcommand{\L}{\Lambda}

\renewcommand{\l}{\lambda}
\renewcommand{\a}{\alpha}
\renewcommand{\d}{\delta}
\renewcommand{\t}{\tau}

\newcommand{\g}{\gamma}
\newcommand{\G}{\Gamma}

\newtheorem{theorem}{Theorem}[section]
\newtheorem*{theorem*}{Theorem}
\newtheorem{lemma}[theorem]{Lemma}
\newtheorem{proposition}[theorem]{Proposition}
\newtheorem{corollary}[theorem]{Corollary}

\newtheorem{claim}[theorem]{Claim}
\newtheorem{maintheorem}{Theorem}

\newtheorem{conjecture}[maintheorem]{Conjecture}

\newtheorem*{question*}{Question}

\newtheorem{assumption}[theorem]{Assumption}

\theoremstyle{definition}
\newtheorem{definition}[theorem]{Definition}
\newtheorem*{remark*}{Remark}
\newtheorem*{idefinition*}{Definition}
\newtheorem{remark}[theorem]{Remark}

\newcommand{\cC}{\ensuremath{\mathcal C}}
\newcommand{\cD}{\ensuremath{\mathcal D}}
\newcommand{\cE}{\ensuremath{\mathcal E}}

\newcommand{\cL}{\ensuremath{\mathcal L}}
\newcommand{\cM}{\ensuremath{\mathcal M}}

\newcommand{\cP}{\ensuremath{\mathcal P}}

\newcommand{\cR}{\ensuremath{\mathcal R}}
\newcommand{\cS}{\ensuremath{\mathcal S}}

\newcommand{\cU}{\ensuremath{\mathcal U}}

\newcommand{\cX}{\ensuremath{\mathcal X}}

\newcommand{\bbE}{{\ensuremath{\mathbb E}} }

\newcommand{\bbH}{{\ensuremath{\mathbb H}} }

\newcommand{\bbL}{{\ensuremath{\mathbb L}} }

\newcommand{\bbN}{{\ensuremath{\mathbb N}} }

\newcommand{\bbP}{{\ensuremath{\mathbb P}} }
\newcommand{\bbQ}{{\ensuremath{\mathbb Q}} }
\newcommand{\bbR}{{\ensuremath{\mathbb R}} }

\newcommand{\bbZ}{{\ensuremath{\mathbb Z}} }
\newcommand{\Z}{{\ensuremath{\mathbb Z}} }

\newcommand{\Ber}{{\rm Ber}}

\let\a=\alpha \let\b=\beta   \let\d=\delta  
 \let\g=\gamma \let\h=\eta      \let\l=\lambda
      \let\o=\omega      
  \let\s=\sigma \let\t=\tau   
  
   \let\G=\Gamma  \let\L=\Lambda 
\let\O=\Omega

\renewcommand{\to}{\rightarrow}
\newcommand{\ds}{\displaystyle}
\newcommand{\ol}{\overline}

\let\epsilon\varepsilon
\let\eps\varepsilon

\def\Ex{\mathbb{E}}

\def\H{\mathbb{H}}

\def\Qb{\mathbb{Q}}

\def\Z{\mathbb{Z}}

\def\S{\mathcal{S}}

\def\U{\mathcal{U}}

\def\stab{\mathcal{S}}

\def\<{\langle}
\def\>{\rangle}

\def\0{\mathbf{0}}

\begin{document}

\title[]{Universality results for kinetically constrained spin models\\in two dimensions} 
\author[Fabio Martinelli]{Fabio Martinelli}
\address{Dipartimento di Matematica e Fisica, Universit\`a Roma Tre, Largo S.L. Murialdo 00146, Roma, Italy}\email{martin@mat.uniroma3.it}

\author[Robert Morris]{Robert Morris}
\address{IMPA, Estrada Dona Castorina 110, Jardim Bot\^{a}«Ïico, Rio de Janeiro, 22460-320, Brazil}\email{rob@impa.br}

\author[Cristina Toninelli]{Cristina Toninelli}
\address{Laboratoire de Probabilit\'es, Mod\'elisation et Statistique,
  CNRS-UMR 7599 Universit\'es Paris VI-VII 4, Place Jussieu F-75252 Paris Cedex 05 France}\email{cristina.toninelli@upmc.fr}

\thanks{This work has been supported by the ERC Starting Grant 680275 ``MALIG'', ANR-15-CE40-0020-01
and by the PRIN 20155PAWZB ``Large Scale Random Structures''. RM is also partially supported by CNPq (Proc.~303275/2013-8) and by FAPERJ (Proc.~201.598/2014).}

  
\begin{abstract}
Kinetically constrained models (KCM) are reversible interacting particle systems on $\mathbb Z^d$ with continuous time Markov dynamics of Glauber type, which represent a natural stochastic (and non-monotone) counterpart of the family of cellular automata known as $\mathcal U$-bootstrap percolation. KCM also display some of the peculiar features of the so-called ``glassy dynamics'', and as such they are extensively used in the physics literature to model the liquid-glass transition, a major and longstanding open problem in condensed matter physics.

We consider two-dimensional KCM with update rule $\cU$, and focus on proving universality results for the mean infection time of the origin, in the same spirit as those recently established in the setting of $\cU$-bootstrap percolation. 
We first identify what we believe are the correct universality classes, which turn out to be different from those of  $\cU$-bootstrap percolation. We then prove universal upper bounds on the mean infection time within each class, which we conjecture to be sharp up to logarithmic corrections. In certain cases, including all supercritical models, and the well-known Duarte model, our conjecture has recently been confirmed in~\cite{MMT}. In fact, in these cases our upper bound is sharp up to a constant factor in the exponent.
For certain classes of update rules, it turns out that the infection time of the KCM diverges much faster {{than}} for the corresponding $\cU$-bootstrap process when the equilibrium density of infected sites goes to zero. This is due to the occurrence of energy barriers which determine the dominant behaviour for KCM, but which do not matter for the monotone bootstrap dynamics.
\end{abstract}

\maketitle
\tableofcontents

\section{Introduction}
Kinetically constrained models (KCM) are interacting particle systems on the integer lattice $\mathbb Z^d$, which were introduced in the physics literature in the 1980s in order to model the liquid-glass transition (see e.g.  \cites{Ritort,GarrahanSollichToninelli} for reviews), a major and still largely open problem in condensed matter physics. The main motivation for the ongoing (and extremely active) research on KCM is that, despite their simplicity, they feature some of the main signatures of a super-cooled liquid near the glass transition point. 

A generic KCM is a continuous time Markov process of Glauber type defined as follows. A configuration $\o$ is defined by assigning to 
each site $x\in\mathbb Z^d$ an occupation variable $\omega_x\in\{0,1\},$ corresponding to an empty or occupied site respectively. Each site waits an independent, mean one, exponential time and then, iff  a certain local constraint is satisfied by the current configuration $\o$, its occupation variable is updated to be occupied with rate $p$ and to empty with rate $q=1-p$. 
All the constraints that have been considered in the physics literature belong to the following general class \cite{CMRT}.

Fix an {\sl update family} $\,\mathcal U=\{X_1,\dots,X_m\}$, that is, a finite collection of finite subsets of $\mathbb Z^d\setminus \{ \mathbf{0} \}$. Then $\o$ satisfies the constraint at site $x$ if there exists $X\in \cU$ such that $\o_y=0$ for all $y\in X+x$.
Since each update set  belongs to $\mathbb Z^d\setminus \{ \mathbf{0} \}$,
the constraints never depend on the state of the to-be-updated site. As a consequence, the product Bernoulli($p$) measure $\mu$ is a reversible invariant measure, and the process started at $\mu$ is stationary. 
Despite this trivial equilibrium measure, however, KCM display an
extremely rich behaviour which is qualitatively different from that of
interacting particle systems with non-degenerate birth/death rates
(e.g. the stochastic Ising model). This behaviour
    includes the key dynamical features of real glassy materials:
    anomalously long mixing times~\cites{Aldous,CMRT,MT}, aging and dynamical heterogeneities~\cite{FMRT-cmp}, and ergodicity breaking transitions corresponding to percolation of blocked structures~\cite{GarrahanSollichToninelli}. Moreover, proving the above results rigorously turned out to be a surprisingly challenging task, in part due to the fact that several of the classical tools typically used to analyse reversible interacting particle systems (e.g., coupling, censoring, logarithmic Sobolev inequalities) fail for KCM.

KCM can be also viewed  as a natural non-monotone and stochastic counterpart of $\mathcal U$-bootstrap percolation, a {{well-studied}} class of discrete cellular automata, {{see~\cite{BBPS,BDMS,BSU}.}} For $\mathcal U$-bootstrap on $\mathbb Z^d$, given a configuration of ``infected" sites $A_t$ at time $t$, infected sites remain infected, and a site $v$ becomes infected at time $t + 1$ if the translate by $v$ of one of the sets in $\cU$ belongs to $A_t$. One then defines the final infection set $[A]_\cU:= \bigcup_{t=1}^\infty A_t$ and the \emph{critical probability} of the $\cU$-bootstrap process on $\Z^d$ to be
\begin{equation}\label{def:qc:bootstrap}
q_c\big( \Z^d, \cU \big) := \inf \Big\{ q \,:\, \bbP_q\big( [A]_\cU = \Z^d \big) =1 \Big\},
\end{equation}
where $\bbP_q$ denotes the product probability measure on $\Z^d$ with density $q$ of infected sites.  
The following key connection between $\mathcal U$-bootstrap percolation and KCM has been established by Cancrini, Martinelli, Roberto and Toninelli~\cite{CMRT}: KCM processes are ergodic with exponentially decaying time auto-correlations for $q>q_c\big( \Z^d, \cU \big)$, and they are not ergodic for $q<q_c\big( \Z^d, \cU \big)$. More precisely, the results of \cite{CMRT} prove that the \emph{relaxation time} $\trel(q;\cU)$ (see Definition \ref{def:PC}) 
and the \emph{mean infection time}\footnote{The mean infection time is very close to the \emph{persistence time} in the physics literature} $\mathbb E_{\mu}(\tau_0)$ (i.e. the mean over the stationary KCM process of the first time at which the origin becomes empty) are finite for $q>q_c\big( \Z^d, \cU \big)$ and infinite for $q< q_c\big( \Z^d, \cU \big)$.
Both from a physical and mathematical point of view, a key question is thus to determine the divergence of the time scales $\trel(q;\cU)$ and $\mathbb E_{\mu}(\tau_0)$ as $q\downarrow q_c(\Z^d,\cU)$. 
We will now briefly review the known results, which show that KCM exhibit a very
large variety of possible scalings depending on the details of the update family $\cU$.

We {{begin}} by discussing one of the most extensively studied KCM, which was introduced by J\"ackle and Eisinger \cite{JACKLE}: the so-called \emph{East model}. This model has update family $\mathcal{U}= \big\{ \{-\vec e_1,\},\dots,\{-\vec e_d\} \big\}$, so in the one-dimensional setting $d = 1$ a site can update iff it is the neighbour ``to the east" of an empty site.  It is not difficult to see that in any dimension $q_c(\Z^d,\cU)=0$.
For $d = 1$, it was first proved in~\cite{Aldous} that the relaxation time $\trel(q)$ is finite for any $q \in (0,1]$, and it was later shown (see~\cites{CFM,Aldous,CMRT}) that it diverges as 
$$\exp\bigg( \big( 1 + o(1) \big) \frac{\log(1/q)^2}{2\log2} \bigg)$$
as $q\downarrow 0$. A similar scaling was later proved in any dimension $d \ge 1$, see~\cite{CFM2}. 

Another well-studied KCM, introduced by Friedrickson and Andersen \cite{FH}, is the $k$-facilitated model (FA-kf), whose update family consists of the $k$-sets of nearest neighbours of the origin: a site can be updated iff it has at least $k$ empty nearest neighbours. In this case it was proved in~\cite{vanEnter,BPd} that $q_c(\Z^d,\cU) = 0$ for all $1 \le k \le d$, whereas $q_c(\Z^d,\cU) = 1$ for all $k > d$. Moreover, the relaxation time $\trel(q)$ diverges as $1/q^{\Theta(1)}$ when $k = 1$~\cite{CMRT,AS}, and as a $(k-1)$-times iterated exponential of $q^{-1/(d-k+1)}$ when $2 \le k \le d$~\cite{MT}. The above scalings also hold for the mean infection time $\mathbb E_{\mu}(\tau_0)$.

The above model-dependent results (which are, in fact, the only ones that have
been proved so far) include a large diversity of possible scalings
of the mean infection time, together with a strong sensitivity to the
details of the update family $\cU$. Therefore, a very natural ``universality'' question emerges: 

\begin{question*}
Is it possible to group all possible update families $\cU$ into distinct classes, in such a way that all members of the same
  class  induce the same divergence of the mean infection time as $q$
  approaches from above the critical value $q_c(\mathbb Z^d,\cU)$?
\end{question*}

Such a general question has not been addressed so far, even in the physics literature: physicists lack a general criterion to predict the different scalings. This fact is particularly unfortunate since, due to the anomalous and sharp divergence of times, numerical simulations often cannot give clear cut and reliable answers. Indeed, some of the rigorous results recalled above corrected some false conjectures that were based on numerical simulations. 


The universality question stated above has, however, being addressed and successfully solved for two-dimensional $\cU$-bootstrap percolation (see~\cite{BSU,BBPS,BDMS}, or~\cite{Robsurvey} for a recent review).
The update families $\cU$ were classified by Bollob\'as, Smith and Uzzell~\cite{BSU} into three universality classes: \emph{supercritical}, \emph{critical} and \emph{subcritical} (see Definition \ref{def:stable}), according to a simple geometric criterion. They also proved in~\cite{BSU} that $q_c\big( \Z^2, \cU \big) = 0$ if $\cU$ is supercritical or critical, and it was proved by Balister, Bollob\'as, Przykucki and Smith~\cite{BBPS} that $q_c\big( \Z^2, \cU \big) > 0$ if $\cU$ is subcritical. For critical update families $\cU$, the scaling (as $q \downarrow 0$) of the typical infection time of the origin starting from $\bbP_q$ was determined very precisely by Bollob\'as, Duminil-Copin, Morris and Smith~\cite{BDMS} (improving bounds obtained in~\cite{BSU}), and various universal properties of the dynamics were obtained.

In this paper we take {{an}} 
important step towards establishing a similar universality picture for
two-dimensional KCM with supercritical or critical update family
$\cU$. Using a geometric criterion, we propose a classification of the
two-dimensional update families into universality classes, which is
inspired by, but at the same time quite different from, that
established for bootstrap percolation. More precisely, we classify a
supercritical update family $\cU$ as being \emph{supercritical
  unrooted} or \emph{supercritical rooted} and a critical $\cU$ as
being \emph{$\a$-rooted} or \emph{$\b$-unrooted}, where $\a \in \bbN$
and $\a \leq \b \in \bbN \cup \{\infty\}$ are called the
\emph{difficulty} and the \emph{bilateral difficulty} of $\cU$
respectively (see Definitions~\ref{def:rooted}
and~\ref{def:alpha:rooted}). We then prove {{(see Sections
    3-7)}}
 the following two main universality results (see Theorems~\ref{mainthm:1} and~\ref{mainthm:2} in Section \ref{sec:results}) on the mean infection time $\bbE_\mu(\t_0)$.\\

\noindent
{\bf Supercritical KCM.}  
{\sl Let $\cU$ be a supercritical two-dimensional update family. Then, as $q \to 0$,
\begin{enumerate}
\item[$(a)$] if $\cU$ is unrooted
$$\bbE_\mu(\t_0) \leq q^{-O(1)};$$
\item[$(b)$] if $\cU$ is rooted, 
$$\bbE_\mu(\t_0) \leq \exp\Big( O\big( \log q^{-1} \big)^2 \Big).$$
\end{enumerate}}

\pagebreak
\noindent
{\bf Critical KCM.}  
{\sl Let $\cU$ be a critical two-dimensional update family with difficulty $\a$ and bilateral difficulty $\b$. Then, as $q \to 0$,
\begin{enumerate}[(a)]
\item if $\cU$ is $\a$-rooted
$$\bbE_\mu(\t_0) \leq \exp\Big( q^{-2\a} \big( \log q^{-1} \big)^{O(1)} \Big);$$
\item if $\cU$ is $\b$-unrooted
$$ \bbE_\mu(\t_0) \leq \exp\Big( q^{-\b} \big( \log q^{-1} \big)^{O(1)} \Big).$$
\end{enumerate}
}
\medskip

Even though the theorems above only establish universal \emph{upper bounds} on $\bbE_\mu(\t_0)$, we conjecture that our bounds provide the correct scaling up to logarithmic corrections. This has recently been proved for supercritical models in~\cite{MMT}. For critical update families, the bound $\mathbb E_{\mu}(\tau_0)=\Omega(T_{\mathcal U})$ (see~\cite{MT}*{Lemma 4.3}), where $T_\cU$ denotes the median infection time of the origin for the $\cU$-bootstrap process at density $q$, together with the results of~\cite{BDMS} on $T_{\mathcal U}$, provide a matching lower bound for all $\beta$-unrooted models with $\alpha = \beta$ (for example, the FA-2f model).
In particular, these recent advances combined with the above theorems prove two conjectures that we put forward in~\cite{Robsurvey}. Among the $\alpha$-rooted models, those which have been considered most extensively in the literature are the Duarte and modified Duarte model (see~\cite{Duarte,Mountford,BCMS-Duarte}), for which $\alpha=1$ and $\beta=\infty$. In~\cite{MMT}, using very different tools and ideas from those in this paper, a lower bound on $\bbE_\mu(\t_0)$ was recently obtained for both models that matches our upper bound, including the logarithmic corrections, yielding $\bbE_\mu(\t_0) = \exp\big( \Theta\big( q^{-2} ( \log 1/q)^4 \big) \big)$.

The above results imply that for all supercritical rooted KCM, and also for the Duarte-KCM, the mean infection time diverges much faster than the median infection time for the corresponding $\cU$-bootstrap process, which obeys $T_{\mathcal U}\sim 1/q^{\Theta(1)}$ for supercritical models~\cite{BSU}, and 
$T_{\mathcal U}\sim\exp\big( \Theta\big( q^{-1} ( \log 1/q)^2 \big) \big)$ for the Duarte model~\cite{Mountford}.  
This is a consequence of the fact that for these KCM the infection time is not well-approximated by the number of updates needed to infect the origin (as it is for bootstrap percolation), but is the result of a much more complex mechanism. In particular, the visits of the process to regions of the configuration space with an anomalous amount of infection (borrowing from physical jargon we may call them ``energy barriers") are heavily penalized and require a very long time to actually take place.

Providing an insight into the heuristics and/or the key steps of the proofs at this stage, before providing a clear definition of the geometrical quantities involved, would inevitably be rather vague. We therefore defer these explanations to Section~\ref{sec:roadmap}. We can, however, state two high-level ingredients. The first one consists in identifying, for each class of update families $\cU,$ an ``efficient" (and potentially optimal) dynamical strategy for the difficult (i.e., unlikely) task of infecting the origin. This is necessarily more complex {{than}} the growth of the corresponding $\cU$-bootstrap process, since an efficient strategy must necessarily feature both infection and healing in order to avoid crossing excessively high energy barriers. The second ingredient consists in using the above strategy as a \emph{guide},\footnote{In this respect our situation shares some similarities with other large deviations problems, where an imagined optimal dynamical strategy has the role of suggesting and motivating several, otherwise mysterious, analytic steps.} without actually implementing it, for the analytic technique introduced in~\cite{MT} by two of the authors of the present paper, which allows one to bound the relaxation time $\trel(q;U)$. In~\cite{MT} this technique was successfully applied to the FA-kf model, with the imagined mechanism for infecting the origin being a large droplet of infected sites moving as a random walk in {{a suitable (evolving)}} random environment of sparse infection. Here we have to go well beyond the method of~\cite{MT}, since the random walk picture does not apply to rooted models. Our main novelty is a new and more complex analytic approach to bound $\trel(q,\cU)$ which is inspired by the East dynamics (see Section~\ref{sec:roadmap} for more details).

\subsection{Notation}

We gather here (for the reader's convenience) some of the standard notation that we use throughout the paper. First, recall that we write $\mu$ for the Bernoulli product measure $\otimes_{x \in \bbZ^2}{\rm Ber}(p)$ on $\bbZ^2$, where $q = 1 - p$ will always be assumed to be sufficiently small (depending on the update family $\cU$).  

If $f$ and $g$ are positive real-valued functions of $q$, then we will write $f = O(g)$ if there exists a constant $C > 0$ (depending on $\cU$, but \emph{not} on $q$) such that $f(q) \le C g(q)$ for every sufficiently small $q > 0$. We will also write $f(q) = \O(g(q))$ if $g(q) = O(f(q))$ and $f(q) = \Theta(g(q))$ if both $f(q) = O(g(q))$ and $g(q) = O(f(q))$. 

All constants, including those implied by the notation $O(\cdot)$, $\O(\cdot)$ and $\Theta(\cdot)$, are quantities that may depend on the update family $\cU$ (and other quantities where explicitly stated) but not on the parameter $q$. If $c_1$ and $c_2$ are constants, then $c_1 \gg c_2 \gg 1$ means that $c_2$ is sufficiently large, and $c_1$ is sufficiently large depending on $c_2$. Similarly, $1 \gg
c_1 \gg c_2 > 0$ means that $c_1$ is sufficiently small, and $c_2$ is sufficiently small depending on $c_1$. Finally, we will use the standard notation $[n] = \{1,\ldots,n\}$.




\section{Universality classes for KCM and main results}

In this section we will begin by recalling the main universality results for
bootstrap cellular automata. We will then define the KCM process associated
to a bootstrap update family, introduce its universality classes, and state our main results about its scaling near criticality. To finish, we will provide an outline of the heuristics behind our main theorems, and a sketch of their proofs.

\subsection{The bootstrap monotone cellular automata and its universality properties}\label{sec:boot}

Let us begin by defining a large class of two-dimensional monotone cellular automata, which were recently introduced by Bollob\'as, Smith and Uzzell~\cite{BSU}. 
 
 \pagebreak
 
\begin{definition}\label{def:Uboot}
Let $\cU = \{ X_1,\ldots,X_m \}$ be an arbitrary finite collection of finite subsets of $\Z^2 \setminus \{ \0 \}$. The \emph{$\cU$-bootstrap process} on $\Z^2$ is defined as follows: given a set $A \subset \Z^2$ of initially \emph{infected} sites, set $A_0 = A$, and define for each $t \geq 0$, 
\begin{equation}\label{eq:def:Uboot:At}
A_{t+1} = A_t \cup \big\{ x \in \Z^2 \,:\, X + x \subset A_t \text{ for some } X \in \cU \big\}.
\end{equation}
We write $[A]_\cU = \bigcup_{t \ge 0} A_t$ for the \emph{closure} of $A$ under the $\cU$-bootstrap process.
\end{definition}

Thus, a vertex {{$x$}} becomes infected at time $t + 1$ if the translate
by {{$x$}} of one of the sets in $\cU$ (which we refer to as the
\emph{update family}) is already entirely infected at time~$t$, and
infected vertices remain infected forever. For example, if we take
$\cU$ to be the family of $2$-subsets of the set of nearest neighbours of the origin, we obtain the classical $2$-neighbour bootstrap process, which was first introduced in 1979 by Chalupa, Leath and Reich~\cite{CLR}. 
One of the key
insights of Bollob\'as, Smith and Uzzell~\cite{BSU} was that, at
least in two dimensions, the typical global behaviour of the
$\cU$-bootstrap process acting on random initial sets should be determined by the
action of the process on discrete half-planes. 

For each unit vector $u \in S^1$, let 
$\H_u := \{x \in \Z^2 : \< x,u \> < 0 \}$
denote the discrete half-plane whose boundary is perpendicular to $u$. 

\begin{definition}\label{def:stable}
The set of \emph{stable directions} is
$$\stab = \stab(\U) = \big\{ u \in S^1 \,:\, [\H_u]_\U = \H_u \big\}.$$
The update family $\cU$ is:
\begin{itemize}
\item \emph{supercritical} if there exists an open semicircle in $S^1$ that is disjoint from $\cS$, \vspace{0.1cm}
\item \emph{critical} if there exists a semicircle in $S^1$ that has finite intersection with $\cS$, and if every open semicircle in $S^1$ has non-empty intersection with $\cS$, \vspace{0.1cm}
\item \emph{subcritical} if every semicircle in $S^1$ has infinite intersection with $\cS$. 
\end{itemize}
\end{definition}

The first step towards justifying this trichotomy is given by the following theorem, which was proved in~\cites{BSU,BBPS}. Recall from~\eqref{def:qc:bootstrap} the definition of $q_c\big( \Z^2, \cU \big)$, the critical probability of the $\cU$-bootstrap process on $\Z^2$.


\begin{theorem}
If $\cU$ is a supercritical or critical two-dimensional update family, then $q_c\big( \Z^2, \cU
\big)=0$, whereas if $\cU$ is subcritical then $q_c\big( \Z^2, \cU \big) > 0$.
\end{theorem}

For supercritical and critical update families, the main question is therefore to determine the scaling as
$q\to 0$ of the typical time it takes to infect the origin.  

\begin{definition}
The \emph{typical infection time} at density $q$ of an update family $\cU$ is defined to be
\[
T_{\cU} \, = \, T_{q,\,\cU} \, := \, \inf\bigg\{ t \ge 0 \,:\, \bbP_q\big( \0 \in A_t \big) \geq \frac 12 \bigg\},
\]  
where (we recall) $\bbP_q$ indicates that every site is included in $A$ with probability $q$, independently from all other sites, and $A_t$ was defined in~\eqref{eq:def:Uboot:At}. We will write $T_{\cU}$, omitting the suffix $q$ from the notation, whenever there is no risk of confusion.
\end{definition}

In order to state the main result of~\cite{BDMS} we need some additional definitions.
Let ${{\bbQ_1}} \subset S^1$ denote the set of rational directions on the circle, and for each $u \in {{\bbQ_1}}$, let $\ell_u^+$ be the (infinite) subset of the line $\ell_u := \{x \in \Z^2 : \< x,u \> = 0 \}$ consisting of the origin and the sites to the right of the origin as one looks in the direction of $u$. Similarly, let $\ell_u^- := (\ell_u\setminus\ell^+_u) \cup \{ \0 \}$ consist of the origin and the sites to the left of the origin. 
Given a two-dimensional bootstrap percolation update family $\cU$, let $\alpha_\cU^+(u)$ be the minimum (possibly infinite) cardinality of a set $Z \subset \Z^2$ such that $[\bbH_u \cup Z]_\cU$ contains infinitely many sites of $\ell_u^+$, and define $\alpha_\cU^-(u)$ similarly (using $\ell_u^-$ in place of $\ell_u^+$). 

\begin{definition}\label{def:alpha}
Given $u \in {{\bbQ_1}}$, the \emph{difficulty} of $u$ (with respect to $\cU$) is\footnote{In order to slightly simplify the notation, and since the update family $\cU$ will always be clear from the context, we will not emphasize the dependence of the difficulty on $\cU$.}
\[
\alpha(u) :=
\begin{cases}
\min\big\{ \alpha_\cU^+(u), \alpha_\cU^-(u) \big\} &\text{if } \alpha_\cU^+(u) < \infty \text{ and } \alpha_\cU^-(u)<\infty \\
\hfill \infty \hfill & \text{otherwise.}
\end{cases}
\]
Let $\cC$ denote the collection of open semicircles of $S^1$. The
\emph{difficulty} of $\cU$ is given by 
\begin{equation}\label{eq:alphaU}
\alpha \, := \, \min_{C \in \cC} \, \max_{u \in C} \, \alpha(u),
\end{equation}
and the \emph{bilateral difficulty} by 
\begin{equation}\label{eq:betaU}
\beta \, := \, \min_{C \in \cC} \, \max_{u \in C} \, \max\big\{ \alpha(u),\a(-u) \big\}.
\end{equation}
A critical update family $\cU$ is \emph{balanced} if there exists a closed semicircle $C$ such that $\alpha(u) \leq \alpha$ for all $u\in C$. It is said to be \emph{unbalanced} otherwise.
\end{definition}

\begin{remark}
If $u \in S^1$ is not a stable direction then $[\H_u]_\U = \bbZ^2$ (see~\cite {BSU}*{Lemma~3.1}), and therefore $\alpha(u) = 0$. Moreover, it was proved in~\cite {BSU}*{Lemma~5.2} (see also~\cite{BDMS}*{Lemma~2.7}) that if $u \in \S(\cU)$ then $\alpha(u) < \infty$ if and only $u$ is an isolated point of $\cS(\cU)$. It follows that $\alpha = 0$ for every supercritical update family, and that $\alpha$ is finite for every critical update family. Observe also that $\alpha \leq \beta \leq \infty$, and that $\beta$ can be infinite even for a supercritical update family (for example, one can embed the one-dimensional East model in two dimensions). A well-studied critical model with $\b$ infinite (and $\a = 1$) is the Duarte model (see~\cite{Duarte,Mountford,BCMS-Duarte}), which has update family 
\begin{equation}\label{def:Duarte}
\cD = \big\{ \{ (-1,0), (0,1) \}, \{(-1,0), (0,-1) \},\{ (0,1), (0,-1) \} \big\}.
\end{equation}
\end{remark}

Roughly speaking, Definition~\ref{def:alpha} says that a direction $u$ has finite difficulty if there exists a finite set of sites that, together with the half-plane $\bbH_u$, infect the entire line $\ell_u$. Moreover, the difficulty of $u$ is at least $k$ if it is necessary (in order to infect $\ell_u$) to find at least $k$ infected sites that are `close' to one
another. If the open semicircle $C$ with $u$ as midpoint contains no direction
of difficulty greater than $k$, then it is possible for a ``critical
droplet" of infected sites to grow in the direction of $u$ without ever finding more than $k$ infected sites close together. 
As a consequence, if the bilateral difficulty is not greater than $k$, {{then}} there exists
a direction $u$ (the midpoint of the optimal semicircle in \eqref{eq:betaU})
such that a suitable critical droplet is able to grow in \emph{both
directions $u$ and $-u$}, without ever finding
more than $k$ infected sites close together.

We are now in a position to state the main results on the scaling of the typical infection time for supercritical and critical update families. The following bounds were proved in~\cite{BDMS} (for critical families) and in~\cite{BSU} (for supercritical families). 

\begin{theorem}
\label{thm:tripartition}
Let $\cU$ be a two-dimensional update family. Then, as $q \to 0$,
\begin{enumerate}[(a)]
\item if $\cU$ is supercritical then 
$$T_\cU \, = \, q^{-\Theta(1)};$$
\item if $\cU$ is critical and balanced with difficulty $\a$, then
$$T_\cU \, = \, \exp\bigg( \frac{\Theta(1)}{q^{\a}} \bigg);$$
\item if $\cU$ is critical and unbalanced with difficulty $\a$, then
$$T_\cU \, = \, \exp\bigg( \frac{\Theta\big( \log (1/q) \big)^2}{q^{\a}} \bigg).$$
\end{enumerate}
\end{theorem}

\begin{remark}
Note that in the above result the bilateral difficulty $\b$ plays no role. This is because in bootstrap percolation a droplet of empty sites only needs to grow in one direction (as opposed to moving back and forth). For KCM, on the other hand, we will see that the ability to move in two opposite directions will play a crucial role.  
\end{remark}

\subsection{General finite range KCM}

In this section we define a 
class of two-dimensional interacting particle systems known as \emph{kinetically constrained models}. As will be clear from what follows, KCM are intimately connected with bootstrap cellular automata.

We will work on the probability space $(\O,\mu)$, where $\O=\{0,1\}^{\bbZ^2}$ and $\mu$ is the product
Bernoulli($p$) measure, and we will be interested in the asymptotic regime $q\downarrow 0,$ where $q=1-p$. Given $\o\in \O$ and $x\in \bbZ^2$, we will say that $x$ is ``empty" (or ``infected") if $\o_x=0$. We will say that $f \colon \O \mapsto \bbR$ is a \emph{local function} if it depends on only finitely many of the variables $\o_x$.

Given a two-dimensional update family $\cU = \{X_1,\dots,X_m\}$, the corresponding KCM is the Markov process on $\O$ associated to the Markov generator 
\begin{equation}
  \label{eq:generator}
(\cL f)(\o)= \sum_{x\in \bbZ^2}c_x(\o)\big( \mu_x(f) - f \big)(\o),
\end{equation}
where $f \colon \O \mapsto \bbR$ is a local function, $\mu_x(f)$ denotes the average of $f$ w.r.t.~the variable $\o_x$, and $c_x$ is the indicator function of the event that there exists an update rule $X\in \cU$ such that $\o_y = 0$ for every $y \in X + x$.

Informally, this process can be described as follows. Each vertex $x\in \bbZ^2$, with rate one and independently across $\bbZ^2$, is resampled from $\big( \{0,1\},{\rm Ber}(p) \big)$ iff {{one of the}} update rules of the $\cU$-bootstrap process at $x$ {{is satisfied}} by the current configuration of the empty sites. In what follows, we will sometimes call such an update a \emph{legal update} or \emph{legal spin flip}. It follows (see~\cite{CMRT}) that $\cL$ {{is}} the generator of a reversible Markov process on $\O$, with reversible measure $\mu$. 

We now define the two main quantities we will use to
    characterize the dynamics of the KCM process. The first
of these is the relaxation time $\trel(q,\cU)$. 

\begin{definition}
\label{def:PC}
We say that $C>0$ is a Poincar\'e constant for a given KCM if, for all local functions $f$, we have
\begin{equation}
  \label{eq:gap}
\var(f) \leq C \, \cD(f),
\end{equation}
where $\cD(f)=\sum_x \mu\bigl(c_x \var_x(f)\bigr)$ is the KCM Dirichlet form of $f$ associated to $\cL$. 
If there exists a finite Poincar\'e constant we then define
\[
\trel(q,\cU):=\inf\big\{ C > 0 \,:\, C \text{ is a Poincar\'e constant for the KCM} \big\}.
\]
Otherwise we say that the relaxation time is infinite.
\end{definition}
A finite relaxation time implies that the reversible measure $\mu$ is mixing for the
semigroup $P_t=e^{t\cL}$ with exponentially decaying time
auto-correlations \cite{Liggett}. More precisely, in that case $\trel(q,\cU)^{-1}$ coincides with the best positive constant $\l$ such that,
\begin{equation}\label{eq:relax:exp:decay}
\var\left(e^{t\cL} f\right)\leq e^{-2\l t}\var(f) \qquad \forall \, f \in L^2(\mu).
\end{equation}
One of the main results of~\cite{CMRT} states that for any $q > 0$ {{we have}} $\trel(q,\cU) < \infty$ for every two-dimensional update family $\cU$ such that $q_c\big( \Z^2, \cU \big) = 0$.
  
The second (random) quantity is the hitting time   
\[
\t_0 = \inf\big\{ t \ge 0 \,:\, \o_0(t) = 0 \big\}.
\] 
In the physics literature the hitting time $\t_0$ is usually referred
to as the \emph{persistence time}, while in the bootstrap
percolation framework it would be more conveniently dubbed the
\emph{infection time}. For our purposes, the most important connection between the mean
infection time $\bbE_\mu(\t_0)$ for the stationary KCM process (\ie with $\mu$ as initial
distribution) and $\trel(q,\cU)$ is as follows (see \cite{Praga}*{Theorem 4.7}): 
\begin{equation}
  \label{eq:mean-infection}
\bbE_\mu(\t_0) \leq \frac{\trel(q,\cU)}{q} \qquad \forall \; q\in (0,1).
\end{equation}
The proof is quite simple. By definition, $\t_0$ is the hitting time of $A = \big\{ \o : \ \o_0 = 0 \big\}$, and it is a standard result (see, e.g.,~\cite{DaiPra}*{Theorem 2}) that $\bbP_\mu(\t_0>t)\le e^{-t\l_A}$, where
\[
\l_A=\inf\big\{ \cD(f) \,:\, \mu(f^2) = 1 \text{ and } f(\o) = 0 \text{ for every } \o \in A \big\}.
\]
Observe that $\var(f) \ge \mu(A) = q$ for any function $f$ satisfying $\mu(f^2)=1$ that is identically zero on $A$.  
This implies that $\l_A \ge q/\trel(q,\cU)$, and so~\eqref{eq:mean-infection} follows.

\begin{remark}
If the initial distribution $\nu$ of the KCM process is different from the invariant measure~$\mu$, then it is only known that $\bbE_\nu(\t_0)$ is finite in a couple of specific cases (the $d$-dimensional East process~\cites{CMST,CFM3}, and the 1-dimensional FA-1f process~\cite{BCMRT}), even under the assumption that $\nu$ is a product Bernoulli($p'$) measure with $p' \neq p$.
\end{remark}

A matching lower bound on $\bbE_\mu(\t_0)$ in terms of $\trel(q,\cU)$ is not known. However, in~\cite{MT}*{Lemma 4.3} it was proved that 
\begin{equation}
  \label{eq:lowbound}
\bbE_\mu(\t_0)=\O(T_{\cU}).
\end{equation}

\subsection{Universality results}\label{sec:results}

We are now ready to define precisely the universality classes for KCM with a supercritical or critical update family. We will also restate (in a more precise form) our main results and conjectures on the scaling of $\bbE_\mu(\t_0)$ as $q \to 0$. We begin with the (much easier) supercritical case.

\begin{definition}\label{def:rooted}
A supercritical two-dimensional update family $\cU$ is said to be
\emph{supercritical rooted} if there exist two non-opposite stable directions in
$S^1$. Otherwise it is called \emph{supercritical unrooted}.
\end{definition} 

Our first main result, already stated in the Introduction, provides an upper bound on $\bbE_\mu(\t_0)$ for every supercritical two-dimensional update family that is (by the results of~\cite{MMT}) sharp up to the implicit constant factor in the exponent. Recall that if $\cU$ is supercritical then $T_{\cU} = q^{-\Theta(1)}$, by Theorem~\ref{thm:tripartition}. 

\begin{maintheorem}[Supercritical KCM]\label{mainthm:1}  
Let $\cU$ be a supercritical two-dimensional update family. Then, as $q \to 0$,
\begin{enumerate}
\item[$(a)$] if $\cU$ is unrooted
$$\bbE_\mu(\t_0) \leq q^{-O(1)} = \exp\Big( O\big( \log T_{\cU} \big) \Big),$$
\item[$(b)$] if $\cU$ is rooted, 
$$\bbE_\mu(\t_0) \leq \exp\Big( O\big( \log q^{-1} \big)^2 \Big) = \exp\Big( O\big( \log T_{\cU} \big)^2 \Big).$$
\end{enumerate}
\end{maintheorem}

We next turn to our bounds for critical update families, the proofs of which will require us to overcome a number of significant technical challenges, in addition to those encountered in the supercritical case. In this setting the distinction between critical unrooted and critical rooted is more subtle, and both the difficulty $\a$ and the bilateral difficulty $\b$ (see Definition~\ref{def:alpha}) play an important role. Recall that for a critical update family the difficulty is finite, but that the bilateral difficulty may be infinite.

\begin{definition}\label{def:alpha:rooted}
A critical update family $\cU$ with difficulty $\a$ and bilateral difficulty $\b$ is said to be $\alpha$-\emph{rooted} if $\b\ge 2\a$. Otherwise it is said to be $\beta$-\emph{unrooted}.\footnote{We warn the attentive reader that when $\a<\b<2\a$ the model is here called $\b$-unrooted, while in~\cite{Robsurvey} it was called $\a$-rooted.}
\end{definition}

\pagebreak

The following theorem is the main contribution of this paper.

\begin{maintheorem}[Critical KCM]\label{mainthm:2}
Let $\cU$ be a critical two-dimensional update family with difficulty $\a$ and bilateral difficulty $\b$. Then, as $q \to 0$,
\begin{enumerate}[(a)]
\item if $\cU$ is $\a$-rooted
$$\bbE_\mu(\t_0) \le \exp\Big( O\Big( q^{-2\a} \big( \log q^{-1} \big)^4 \Big) \Big) = \exp\Big( \tilde O\big( \log T_\cU \big)^2 \Big);$$
\item if $\cU$ is $\b$-unrooted
$$\bbE_\mu(\t_0) \le \exp\Big( O\Big( q^{-\b} \big( \log q^{-1} \big)^3 \Big) \Big) = \exp\Big( \tilde O\big( \log T_\cU \big)^{\b/\a} \Big).$$
\end{enumerate}
\end{maintheorem}

\begin{remark}
\label{rem:Trel}
It will follow immediately from our proof that  the upper bounds  of Theorems~\ref{mainthm:1} and~\ref{mainthm:2}  hold also for the relaxation time $\trel(q,\cU)$.
Indeed, we first establish  upper bounds for the relaxation time, then derive the upper bounds on $\bbE_\mu(\t_0)$ via
\eqref{eq:mean-infection}.  
\end{remark}

It was recently proved in~\cite{MMT} that the upper bounds in Theorem~\ref{mainthm:1} are best possible up to the implicit constant factor in the exponent for all supercritical update families (note that this follows from~\eqref{eq:lowbound} for unrooted models). We conjecture that the bounds for critical models in Theorem~\ref{mainthm:2} are also best possible, though in a slightly weaker sense: up to a polylogarithmic factor in the exponent.

\begin{conjecture}\label{critical:conj}
Let $\cU$ be a critical two-dimensional update family with difficulty $\a$ and bilateral difficulty $\b$. Then, as $q \to 0$,
\begin{enumerate}[(a)]
\item if $\cU$ is $\a$-rooted
$$\bbE_\mu(\t_0) = \exp\Big( q^{-2\a} \big( \log q^{-1} \big)^{\Theta(1)} \Big);$$
\item if $\cU$ is $\b$-unrooted
$$\bbE_\mu(\t_0) = \exp\Big( q^{-\b} \big( \log q^{-1} \big)^{\Theta(1)} \Big).$$
\end{enumerate}
\end{conjecture}

Observe that for $\a$-unrooted update families $\cU$ (i.e., families with $\b = \a$), the lower bound in Conjecture~\ref{critical:conj} follows from Theorem~\ref{thm:tripartition} and~\eqref{eq:lowbound}; in particular Theorem~\ref{mainthm:2} confirms~\cite{Robsurvey}*{Conjecture~2.4}. If $\cU$ is moreover unbalanced, then the upper and lower bounds given by Theorems~\ref{mainthm:2} and~\ref{thm:tripartition} differ by only a single factor of $\log(1/q)$ (in the exponent), and we suspect that in this case the lower bound is correct, see {{Remark~\ref{rmk:losing:log}.}} 
 
\begin{conjecture}\label{unbalanced:conj}
Let $\cU$ be an $\a$-unrooted, unbalanced, critical two-dimensional update family with difficulty $\a$. Then, as $q \to 0$,
$$\bbE_\mu(\t_0) = \exp\Big( \Theta\Big( q^{-\a} \big( \log q^{-1} \big)^2 \Big) \Big).$$
\end{conjecture}    

We remark that an example of an update family satisfying the conditions of Conjecture~\ref{unbalanced:conj} is the so-called \emph{anisotropic model} (see, e.g.,~\cite{DC-Enter,DPEH}) whose update family consists of all subsets of size 3 of the set
$$\big\{ (-2,0), (-1,0), (1,0), (2,0), (0,1), (0,-1) \big\}.$$
Another model for which Conjecture~\ref{critical:conj} holds is the Duarte model, defined in~\eqref{def:Duarte}, for which a matching lower bound (this time, up to a \emph{constant} factor in the exponent) was recently proved in~\cite{MMT}, confirming (in a strong sense)~\cite{Robsurvey}*{Conjecture~2.5}. For all other critical models, however, the best known lower bound is that given by Theorem~\ref{thm:tripartition} and~\eqref{eq:lowbound}, and is therefore (we think) very far from the truth.


\subsection{Heuristics and roadmap}\label{sec:roadmap}

We conclude this section with a high-level description of the intuition behind the proofs
of Theorems~\ref{mainthm:1} and~\ref{mainthm:2}, together with a roadmap of the actual proof, which is carried out in Sections~\ref{sec:CPI}--\ref{sec:fullgen}.

The first key point to be stressed is that we never actually follow the dynamics of the KCM process itself; instead, we will prove the existence of a Poincar\'e constant with the correct scaling as $q\to 0$, and use the inequality~\eqref{eq:mean-infection} to deduce a bound on the mean infection time.
We emphasize that this approach only works for the stationary KCM, that is, the process starting from the stationary measure $\mu$.  The second point is that, given that the Dirichlet form of the KCM 
$$\cD(f) = \sum_{x \in \bbZ^2} \mu\big( c_x \var_x(f) \big)$$
is a sum of local variances ($\Leftrightarrow$ spin flips) computed with suitable infection nearby ($\Leftrightarrow$ the constraints $c_x$), all of our reasoning will be guided by the fact that we need to have some infection ($\Leftrightarrow$ empty sites) next to where we want to compute the variance. Therefore, much of our intuition, and all of the technical tools, have been developed with the aim of finding a way to \emph{effectively} move infection where we need it. 

A configuration sampled from $\mu$ will always have ``mesoscopic'' droplets (large patches of infected sites), though these will typically be very far from the origin. The general theory of $\cU$-bootstrap percolation developed in~\cite{BDMS,BSU} allows us to quantify very precisely the critical size of those droplets that (typically) allows infection to grow from them and invade the system. However -- and this is a fundamental difference between bootstrap percolation and KCM -- it is extremely unlikely for the stationary KCM to create around a given vertex and at a given time a very large cluster of infection. Thus, it is essential to envisage an \emph{infection/healing} mechanism that is able to \emph{move} infection over long distances without creating too large an excess\footnote{In physical terms an excess of infection is equivalent to an ``energy barrier''.} of it.


At the root of our approach lies the notion of a \emph{critical
  droplet}. A critical droplet is a certain finite set $D$ whose
geometry depends on the update family $\cU$, and whose characteristic
size may depend on $q$. For supercritical models we can take any
sufficiently large (\emph{not} depending on $q$) rectangle oriented
along the mid-point $u$ of a semicircle $C$ free of stable
directions. For critical models the droplet $D$ is a more complicated
object called a \emph{quasi-stable half-ring} (see Definition
\ref{def:half-ring} and Figure~\ref{fig:half-ring}) oriented along the
midpoint $u$ of an open semicircle with largest difficulty either $\a$
or $\b$. The long sides of $D$ will have length either
$\Theta\big( q^{-\a} \log(1/q) \big)$ or $\Theta\big( q^{-\b} \log(1/q) \big)$ for the 
$\a$-rooted {{and}} $\b$-unrooted cases respectively,
while the short sides will always have length $\Theta(1)$. The key feature of a critical droplet for supercritical models (see
Section~\ref{sec:supercritical:bootstrap}) is that, if it is empty, {{then}} it is
able to infect a suitable translate of itself in the
$u$-direction. For unrooted supercritical models
the semicircle $C$ can be chosen in such a way that both $C$
\emph{and} $-C$ are free of stable directions. As a consequence, the
empty critical droplet will be able to infect a suitable translate of itself in \emph{both}
directions $\pm u$. 

For critical models the situation changes drastically. An empty critical droplet will not be able to infect freely another critical droplet next to it in the $u$-direction because of the stable directions which are present in every open semicircle. However, it will be able to do so (in the $u$-direction if the model is $\a$-rooted, and in the $\pm u$-directions if $\b$-unrooted) provided that it receives some help from a finite number of extra empty sites (in ``clusters" of size $\a$ or $\b$) nearby. If the size of the critical droplet {{is}} chosen as above, then it is straightforward to show that such extra helping empty sites will be present with high probability (see Section~\ref{sec:critical:rooted}).

Having clarified what a critical droplet is, and under which
circumstances it is able to infect nearby sites, we next explain what
we mean by ``moving a critical droplet''. For simplicity we explain
the heuristics only for
the supercritical case. Imagine that we have a sequence
$D_0,D_1,\dots,D_n$ of contiguous, non-overlapping and identical critical droplets
such that $D_{i+1}=D_i+ d_i u$ for some suitable $d_i>0$. Suppose first that the model is unrooted and that $D_0$ is completely infected, and let us write $\o_i$ for the configuration of spins 
in $D_i$. Using the infection in $D_0$ it possible to first infect
$D_1$, then $D_2$ and
then, using reversibility, restore (\ie heal) the original configuration $\o_1$
in $D_1$. Using the infection in $D_2$ we can next infect $D_3$ and
then, using the infection in $D_3$, restore $\o_2$ in $D_2$ (see the
schematic diagram below, where $\emptyset$ stands for an infected droplet)
\begin{align*}
\emptyset\ \o_1\ \o_2\ \o_3\dots &\mapsto \emptyset\ \emptyset\ \o_2\
\o_3\dots  \mapsto\emptyset\ \emptyset\ \emptyset\
\o_3\dots  \\
& \mapsto \emptyset\ \o_1\ \emptyset\ \o_3\dots \mapsto \emptyset\ \o_1\ \emptyset\ \emptyset\dots \mapsto \emptyset\ \o_1\ \o_2\ \emptyset\dots 
\end{align*}
If we
continue in this way, we end up moving the
original infection in $D_0$ to the last droplet $D_n$ without having
ever created more than two extra infected critical droplets
simultaneously. We remark that the sequence described above is reminiscent of how infection moves in the one-dimensional $1$-neighbour KCM. 

For rooted supercritical models, on the other hand, we cannot simply restore the configuration $\o_2$ in $D_2$ using only the infection in $D_3$ (in the unrooted case this was possible because infection could propagate in both the $u$ and $-u$ directions). As a consequence, we need to follow a more complicated pattern:
\begin{align*}
\emptyset\ \o_1\ \o_2\ \o_3\dots &\mapsto \emptyset\ \emptyset\ \o_2\
\o_3\dots  \mapsto\emptyset\ \emptyset\ \emptyset\
\o_3\dots  \\
&\mapsto \emptyset\ \emptyset\ \emptyset\ \emptyset\dots \mapsto \emptyset\ \ \emptyset\ \o_2\ \emptyset\dots \mapsto \emptyset\ \o_1\ \o_2\ \emptyset\dots, 
\end{align*}
in which healing is always induced by infection present in the adjacent droplet in the $-u$
direction. This latter case is reminiscent of the one-dimensional East
model. In this case, a combinatorial result proved in~\cite{CDG} implies that in
order to move the infection to $D_n$ it is necessary to create
$\asymp \log n$ \emph{simultaneous} extra infected critical
droplets. This logarithmic energy barrier is the reason for the
different scaling of $\bbE_\mu(\t_0)$ in rooted and unrooted
supercritical models (see Theorem~\ref{mainthm:1}). 

Let us now give a somewhat more detailed outline of our approach. We begin by partitioning $\bbZ^2$ into `suitable' rectangular blocks $\{V_i\}_{i\in \bbZ^2}$ with shortest side orthogonal to the direction $u$ (see Section \ref{sec:setting}). For supercritical models these blocks have sides of constant length, while for critical
models they will have length $q^{-\kappa}$ for some constant $\kappa \gg \a$, and height equal to that of a critical droplet, so either $\Theta\big( q^{-\a} \log(1/q) \big)$ or $\Theta\big( q^{-\b} \log(1/q) \big)$, depending on the nature of the model. Then, given a configuration $\o \in \O$, we declare a block to be \emph{good} or \emph{super-good} according to the following rules:
\begin{itemize}
\item For supercritical models \emph{any} block is good, while for critical models good blocks are those which contain ``enough" empty sites to allow an adjacent empty critical droplet to advance in the $u$ (or $\pm u$) direction(s) (see Definition~\ref{def:goodsets}).\smallskip
\item In both cases, a block is said to be super-good if it is good and also contains an empty (i.e., completely infected) critical droplet.
\end{itemize}
Good blocks turn out to be very likely w.r.t.~$\mu$ (a triviality in the supercritical case), and it follows by standard percolation arguments that they form a rather dense infinite cluster. Super-good blocks, on the other hand, are quite rare, with density $\rho= q^\Theta(1)$ in the supercritical case, $\rho = \exp\big( - \Theta\big( q^{-\a} \log(1/q)^2 \big) \big)$ in the critical $\a$-rooted case, and $\rho = \exp\big( - \Theta\big( q^{-\b} \log(1/q)^2 \big) \big)$ for critical $\b$-unrooted models.

We will then prove the existence of a suitable Poincar\'e constant in three steps, each step being associated to a natural kinetically constrained \emph{block dynamics}\footnote{See, e.g.,~Chapter~15.5 of~\cite{Levin-2008} for a introduction to the technique of block dynamics in reversible Markov chains.} on a certain length
scale. In each block dynamics the configuration in each block is resampled with rate one (and independently of other resamplings) if a certain constraint is satisfied.

Our first block dynamics forces one of the blocks neighbouring $V_i$ to be at the beginning of an oriented ``thick'' path $\g$ of good blocks, with length $\approx 1/\rho$, whose last block is super-good. Using the fact that this constraint is very likely, it is possible to prove (see Section~2 in~\cite{MT}) that the relaxation time of this process is $O(1)$, and moreover (see Proposition~\ref{lem:MT:prop34}) that the Poincar\'e inequality
\begin{equation}\label{eq:8}
\var(f) \leq 4 \sum_i \mu\big( \id_{\G_i} \var_i(f) \big)
\end{equation}
holds, where $\id_{\G_i}$ is the indicator of the event that a good path exists for $V_i$. 
Though this starting point is similar to the method we develop in \cite{MT},
for the next two steps of the proof we introduce here a completely different set of tools and ideas in order to avoid the direct use of {\emph canonical paths} 
(which could instead be used in \cite{MT} for the special case of the FA-2f model). Indeed for a general model (and especially for rooted models), using canonical paths and evaluating their congestion constants would result in a very heavy and complicated machinery. 
The next idea is to convert the \emph{long-range} constrained
Poincar\'e inequality \eqref{eq:8} into a \emph{short-range} one of the form 
\begin{equation}\label{eq:10}
  \var(f) \leq C_1(q)\sum_i \mu\big( \id_{SG_i} \var_i(f) \big),
\end{equation}
in which $\id_{SG_i}$  is the indicator of the event that a suitable
collection of blocks \emph{near} $V_i$ are good and one of them is
super-good. Which collections of blocks are ``suitable", and which one
should be super-good, depends on whether the model is rooted or
unrooted; we refer the reader to Theorem~\ref{thm:CPI} for the
details. The main content of Theorem~\ref{thm:CPI}, which we present in a slightly more general setting for later convenience, is that $C_1(q)$ can be taken equal to the best Poincar\'e constant (\ie the relaxation time) of a one-dimensional generalised $1$-neighbour or East process at the effective density $\rho$. Section~\ref{sec:CPI} is entirely dedicated to the task of formalising and proving the above claim. 

The final step of the proof is to convert the Poincar\'e inequality~\eqref{eq:10} into the true Poincar\'e inequality for our KCM
\[
\var(f) \leq C_2(q)\sum_x \mu(c_x \var_x(f)),
\]
with a Poincar\'e constant $C_2(q)$ which scales with $q$ as required
by Theorems~\ref{mainthm:1} and~\ref{mainthm:2}. In turn, {{this}} requires
{{us to prove}} that a full resampling of a block in the presence of nearby
super-good and good blocks can be simulated (or reproduced) by a sequence of legal single-site updates of the \emph{original} KCM, with a
global cost in the Poincar\'e constant compatible with
Theorems~\ref{mainthm:1} and~\ref{mainthm:2}. It is here that the
results of~\cite{BDMS,BSU} on the behaviour of the
    $\cU$-bootstrap process come into play. While for supercritical
models the task described above is relatively simple
(see Section~\ref{sec:supercritical}), for critical models the problem
is significantly more {{complicated}} and a suitable generalised
East process {{again plays}} a key role. A full sketch of the proof can be found in Section~\ref{sec:core}, see in particular the proof of Proposition~\ref{prop:I1:I2}, and Remark~\ref{rem:whyEast}.

\section{Constrained Poincar\'e inequalities}\label{sec:CPI}

The aim of this section is to prove a constrained Poincar\'e inequality for a product measure on $S^{\bbZ^2}$, where $S$ is a finite set. This general inequality will play an instrumental role in the proof of our main theorems, giving us precise control of the infection time for both supercritical and critical KCM. 

In order to state our general constrained Poincar\'e inequality, we will need some notation. Let $(S,\hat\mu)$ be a finite positive probability space, and set $\O = \big( S^{\bbZ^2},\mu \big)$, where $\mu = \otimes_{i\in \bbZ^2}\hat\mu$. A generic element $\O$ will be denoted by $\o=\{\o_i\}_{i\in \bbZ^2}$. For any local function $f$ we will write $\var(f)$ for its variance w.r.t.~$\mu$ and $\var_i(f)$ for the variance w.r.t.~to the variable $\o_i\in S$ conditioned on all the other variables $\{\o_j\}_{j\neq i}$. For any $i\in \bbZ^2$ we set 
$$\bbL^+(i) = i + \big\{ \vec e_1, \vec e_2 - \vec e_1 \big\} \qquad \textup{and} \qquad \bbL^-(i) = i - \big\{ \vec e_1, \vec e_2 - \vec e_1 \big\}.$$
Finally, let $G_2 \subseteq G_1\subseteq S$ be two events, and set $p_1 := \hat\mu(G_1)$ and $p_2 := \hat\mu(G_2)$. The main result of this section is the following theorem.

\begin{theorem}\label{thm:CPI}
For any $t\in (0,1)$ there exist $\vec T(t),T(t)$ satisfying $\vec T(t) \leq \exp\big( O\big( \log \tfrac{1}{t} \big)^2 \big)$ and $T(t) \leq t^{-O(1)}$ as $t \to 0$, such that the following oriented and unoriented constrained Poincar\'e inequalities hold.\vskip6pt
\noindent {\bf (A)} \hskip 12pt Suppose that $G_1=S$ and $G_2 \subseteq S$. Then, for all local functions $f$:
\begin{align}\label{eq:8East}
& \var(f) \leq \vec T(p_2) \sum_{i\in \bbZ^2}\mu\left(\id_{\{\o_{i+\vec e_1} \in G_2\}}\var_i(f)\right)\\
& \var(f) \leq T(p_2) \sum_{i\in \bbZ^2}\mu\left(\id_{\{\{\o_{i+\vec e_1}\in G_2\}\cup\{\o_{i-\vec e_1} \in G_2\}\}}\var_i(f)\right).\label{eq:8FA}
\end{align}
\vskip 6pt
\noindent {\bf (B)} \hskip 12pt Suppose that $G_2 \subseteq G_1 \subseteq S$. Then there exists $\d > 0$ such that, for all $p_1,p_2$ satisfying $\max\big\{ p_2, (1-p_1) (\log p_2)^2 \big\} \leq \d$,
and all local functions $f$:
\begin{align}\label{eq:9East}
& \var(f) \leq \vec T(p_2) \bigg( \sum_{i\in \bbZ^2}\mu\left(\id_{\{\o_{i+\vec e_2} \in G_2\}}\id_{\{\o_{j}\in G_1\, \forall j\in \bbL^+(i)\}} \var_i(f)\right)\nonumber\\
& \hspace{3.5cm} + \sum_{i\in \bbZ^2}\mu\left(\id_{\{\o_{i+\vec e_1}\in G_2\}}\id_{\{\o_{i-\vec e_1}\in G_1\}} \var_i\big( f \tc G_1 \big) \right)\bigg),\\
& \var(f) \leq T(p_2) \bigg( \sum_{\varepsilon=\pm 1}\sum_{i\in \bbZ^2}\mu\left(\id_{\{\o_{i+\varepsilon\vec e_2} \in G_2\}}\id_{\{\o_{j}\in G_1\, \forall j\in \bbL^{\varepsilon}(i)\}} \var_i(f)\right) \nonumber\\
& \hspace{3.5cm} + \sum_{\varepsilon=\pm 1}\sum_{i\in \bbZ^2}\mu\left(\id_{\{\o_{i+\varepsilon\vec e_1} \in G_2\}}\id_{\{\o_{i-\varepsilon \vec e_1}\in G_1\}} \var_i\big( f \tc G_1 \big) \right) \bigg). \label{eq:9FA}
\end{align}
\end{theorem}

\begin{remark}
When proving Theorem \ref{mainthm:1} the starting
point will be~\eqref{eq:8East} or~\eqref{eq:8FA}, depending on whether the model is
rooted or unrooted. Similarly, for critical models we will start the
proof of Theorem~\ref{mainthm:2} from~\eqref{eq:9East} or~\eqref{eq:9FA}
depending on whether the model is $\a$-rooted or $\b$-unrooted. This
choice is dictated by the $\cU$-bootstrap process according to the following
rule: we will require $V_i \subset [A]_\cU$ to hold for \emph{any} set $A$ of empty sites such that the indicator function in front of $\var_i(f)$ is equal to one. We refer the reader to Sections~\ref{sec:supercritical} and~\ref{sec:critic-ub}, and in particular to the proof of Lemma~\ref{lem:var0}, for more details.
\end{remark}

An important role in the proof of the theorem is played by the one-dimensional East and $1$-neighbour processes (see, e.g.,~\cite{CMRT}), and a certain generalization of these processes. For the reader's convenience, we begin by recalling these generalized models.
  
\subsection{The generalised East and $1$-neighbour models}\label{sec:East-FA} 

The standard versions of these two models are ergodic interacting particle systems on $\{0,1\}^n$ with kinetic constraints, which will mean that jumps in the dynamics are facilitated by certain configurations of vertices in state $0$. They are both reversible w.r.t.~the product measure $\pi= \Ber(\a_1) \otimes \cdots \otimes \Ber(\a_n)$, where $\Ber(\a)$ is the $\a$-Bernoulli measure and $\a_1,\ldots,\a_n \in (0,1)$. 

In the first process, known as the \emph{non-homogeneous East model} (see~\cites{JACKLE,East-review} and references therein), the state $\o_x$ of each point $x \in [n]$ is resampled at rate one  (independently across $[n]$) from the distribution $\Ber(\a_x)$, provided that $c_x(\o) = 1$, where 
$$c_x(\o) = \id_{\{\o_{x+1}=0\}} \qquad \textup{and} \qquad \o_{n+1} := 0.$$ 
In the second model, known as the \emph{non-homogeneous $1$-neighbour model} (and also as the FA-1f model~\cites{FH}), the resampling occurs in the same way, except in this case 
$$c_x(\o)= \max\big\{ \id_{\{\o_{x-1} = 0\}}, \, \id_{\{\o_{x+1} = 0\}} \big\} \qquad \textup{where} \qquad \o_{0} := 1 \qquad \textup{and} \qquad \o_{n+1} := 0.$$ 
It is known \cites{Aldous,CMRT,CFM} that the corresponding relaxation
times $T_{\text{\tiny East}}(n,\bar \a)$ and  $T_{\text{\tiny
    FA}}(n,\bar\a)$ (where $\bar\a = (\a_1,\ldots,\a_n)$) are finite
\emph{uniformly} in $n$ and that they satisfy the following scaling as
$q := \min\big\{ 1 - \a_x : x \in [n] \big\}$ tends to zero:
\begin{align}\label{eq:scaling}
T_{\text{\tiny East}}\big( n,\bar \a \big)= q^{-O(\min\{ \log n, \,
                                            \log(1/q) \})} \qquad
                                            \text{and}\qquad T_{\text{\tiny FA}}\big( n,\bar \a
\big) = q^{-O(1)}.
\end{align}
The proof of \eqref{eq:scaling} is deferred to the Appendix. In the proof of Theorem~\ref{thm:CPI} we will need to work in the following more general setting. 

Consider a finite product probability space of the form $\O=\otimes_{x\in
  [n]}(S_x,\nu_x)$, where $S_x$ is either a finite set or an interval
of $\bbR$, and $\nu_x$ is a positive
probability measure on $S_x$. Given $\{\o_x\}_{x \in [n]}\in \O$, we
will refer to $\o_x$ as the \emph{the state of the vertex
  $x$}. Moreover, for each $x\in [n]$, let us fix a constraining event
$S^g_x \subseteq S_x$ with $q_x := \nu_x(S^g_x) > 0$. We consider the following generalisations of the East and FA-1f processes on the space $\O$.

\begin{definition}
\label{def:gen:East} In the \emph{generalised East chain}, the state $\o_x$ of each vertex $x \in [n]$ is resampled at rate one (independently across $[n]$) from the distribution $\nu_x$, provided that $\vec c_x(\o) = 1$, where 
$$\vec c_x(\o) = \id_{\{\o_{x+1} \in S^g_{x+1}\}}$$
if $x \in \{1,\ldots,n-1\}$, and ${{\vec c_n}}(\o) \equiv 1$. 

In the \emph{generalised FA-1f chain}, the resampling occurs in the same way, except in this case $c_1(\o)= \id_{\{ \o_2 \,\in\, S^g_2 \}}$,  
$$c_x(\o)= \max\big\{ \id_{\{\o_{x-1} \, \in \, S^g_{x-1}\}}, \, \id_{\{\o_{x+1} \, \in \, S^g_{x+1}\}} \big\}$$
if $x \in \{2,\ldots,n-1\}$, and $c_n(\o) \equiv 1$. 

In both cases, set $q := \min_x q_x = \min_x \nu_x(S^g_x)$, and set $\a_x := 1 - q_x$ for each $x \in [n]$. 
\end{definition}

Note that the projection variables $\eta_x = \id_{\{S^g_x\}}$ evolve as a standard East or FA-1f chain, and it is therefore natural to ask whether the relaxation times of these generalised constrained chains can be bounded from above in terms of the relaxation times $T_{\text{\tiny East}}(n,\bar \a)$ and $T_{\text{\tiny FA}}(n,\bar \a)$ respectively. The answer is affirmative, and it is the content of the following proposition (cf.~\cite{CFM2}*{Proposition 3.4}), which provides us with Poincar\'e inequalities for the generalised East and FA-1f chains. 


\begin{proposition}\label{lem:gen-Poincare}
Let $f:\O\mapsto \bbR$. For the generalised East chain, we have
\begin{equation}\label{eq:Poincare:genEast}
\var(f) \leq \frac{1}{q} \cdot T_{\text{\tiny East}}(n,\bar \a) \cdot \sum_{x = 1}^n \nu\big( \vec c_x\var_{x}(f) \big),
\end{equation}
and for the generalised FA-1f chain, we have
\begin{equation}\label{eq:Poincare:genFA1f}
\var(f) \leq \frac{1}{q} \cdot T_{\text{\tiny FA}}(n,\bar \a) \cdot \sum_{x = 1}^n \nu\big( c_x \var_{x}(f) \big),
\end{equation}
where $\var_x(\cdot)$ denotes the conditional variance w.r.t.~$\nu_x$,
given all the other variables. 
\end{proposition}

The proof of this proposition, which is similar to that of~\cite{CFM}*{Proposition 3.4}, is deferred to {{the}} Appendix. 

\subsection{Proof of Theorem~\ref{thm:CPI}}


We begin with the proof of part (A), which is a relatively straightforward consequence of Proposition~\ref{lem:gen-Poincare} and~\eqref{eq:scaling}. The proof of part (B) is significantly more difficult, and we will require a technical
result from~\cite{MT} (see Proposition~\ref{lem:MT:prop34}, below) and a careful application of Proposition~\ref{lem:gen-Poincare} (and of convexity) after conditioning on various events.   

\subsubsection{Proof of part (A)}

Recall that in this setting $G_1 = S$ and $G_2\subset S$, where
$(S,\hat\mu)$ is an arbitrary finite positive probability space. 
Let $f$ be a local function and let $M > 0$ be sufficiently large so that $f$ does not depend on the variables at vertices $(m,n)$ with $|m| \ge M$. For each $n \in \bbZ$, let $\mu_{n}$ denote the product measure $\otimes_{m\in \bbZ}\ \hat \mu$ on $S^{\bbZ \times \{n\}}$, and note that  $\mu=\otimes_{n \in \bbZ}\ \mu_n$. By construction, $\var_{\mu_n}(f)$ coincides with the same conditional variance computed w.r.t.~$\mu_{n}^M := \otimes_{m \in \bbZ\cap [-M,M]}\ \hat \mu.$ 

We apply Proposition~\ref{lem:gen-Poincare} to the homogeneous product measure $\mu_n^{M}$ with the event $G_2$ as event $S_x^g$ for all $x \in \{-M,\ldots,M\}$. Note that $q_x = \hat\mu(G_2) = p_2$ for every $x$, and that $\var_{(M,n)}(f) = \var_{(-M,n)}(f) = 0$. It follows, using~\eqref{eq:scaling}, that
$$\var_{\mu_n}(f) \leq \, \vec T(p_2) \ds\sum_{m\in \bbZ}\mu_n\left(\id_{\{\o_{(m+1,n)} \in G_2\}} \var_{(m,n)}(f)\right),$$
where $\vec T(p_2) = \exp\Big( O\big( \log \tfrac{1}{p_2} \big)^2 \Big)$, and 
$$\var_{\mu_n}(f) \leq \, T(p_2) \ds\sum_{m \in \bbZ} \mu_n\left(\id_{\{\o_{(m+1,n)} \in G_2\} \cup \{\o_{(m-1,n)} \in G_2\}} \var_{(m,n)}(f)\right),$$
where $T(p_2) = p_2^{-O(1)}$. 
Using the standard inequality $\var_\mu(f)\le \sum_{n \in \bbZ} \mu\big( \var_{\mu_n}(f) \big)$, the Poincar\'e inequalities~\eqref{eq:8East} and~\eqref{eq:8FA} follow.

\subsubsection{Proof of part (B)}

We next turn to the significantly more challenging task of proving the constrained Poincar\'e inequalities~\eqref{eq:9East} and~\eqref{eq:9FA}. As noted above, in addition to Proposition~\ref{lem:gen-Poincare} we will require a technical result from~\cite{MT}, stated below as Proposition~\ref{lem:MT:prop34}. In order to state this result we need some additional notation. 

Recall that an \emph{oriented path of length $n$} in $\bbZ^2$ is a
sequence $\g = (i^{(1)},\dots, i^{(n)})$ of $n$ vertices of $\bbZ^2$
with the property that $i^{(k+1)} - i^{(k)} \in \{ \vec e_1, \vec e_2
\}$ for each $k \in [n-1]$. We will say that $\g$ starts at $i^{(1)}$,
ends at $i^{(n)}$, and that $i \in \g$ if $i = i^{(k)}$ for some $k
\in [n]$. Moreover, given $\o\in \O$, we will say that $\g$ is
\begin{itemize}
\item $\o$-\emph{good} if $\o_i \in G_1$ for all $i \in \bigcup_{j \in \g} \big\{ j, j + \vec e_1, j - \vec e_1 \big\}$, and 
\item $\o$-\emph{super-good} if it is good and there exists $i \in \g$ such that $\o_i\in G_2$,
\end{itemize}  
where $G_2 \subseteq G_1 \subseteq S$ are the events in the statement of Theorem~\ref{thm:CPI}. 

In what follows it will be convenient to order the oriented paths of
length $n$ starting from a given point according to the alphabetical order of the associated strings of $n$ unit vectors from the finite alphabet
$\cX=\{\vec e_1,\vec e_2\}$. Next, for each $i\in \bbZ^2$ we define
the key event $\G_i\subset \O$, as follows:
\begin{enumerate}[(i)]
\item there exists an oriented $\o$-good path $\g$, of length $L = \big\lfloor 1/p_2^2 \big\rfloor$ 
starting at $i + \vec e_2$;
\item the smallest such path (in the above order) is $\o$-super-good;
\item $\o_{i+\vec e_1} \in G_1$.
\end{enumerate}
In what follows, and if no confusion arises, we will abbreviate
$\o$-good and $\o$-super-good to good and super-good
respectively. The following upper bound on $\var(f)$ is
very similar to~\cite{MT}*{Proposition 3.4}, and we therefore defer the
proof to the Appendix. 

\begin{proposition}\label{lem:MT:prop34}
There exists $\d > 0$ such that, if $\max\big\{ p_2, (1-p_1) (\log p_2)^2 \big\} \le \d$, then 
\begin{equation}\label{eq:MT:prop34}
\var(f) \leq 4 \sum_{i \in \bbZ^2} \mu\big( \id_{\G_i} \var_i(f) \big)
\end{equation}
for every local function $f$. 
\end{proposition}

We would like to use Proposition~\ref{lem:gen-Poincare} to bound the right-hand side of~\eqref{eq:MT:prop34}. However, Proposition~\ref{lem:gen-Poincare} provides us with an upper bound on the variance of a function, 
whereas the quantity $\mu\big( \id_{\G_i} \var_i(f) \big)$ is more like the average of a local variance. We will therefore need to use convexity to bound from above the average of a local variance by a full variance. In order to reduce as much as possible the potential loss of such an operation, we first perform a series of conditionings on the measure $\mu$ and use convexity only on the final conditional measure.  

Roughly speaking, on the event $\G_i$ we first reveal, for each $j \ne i$ within distance $2/p_2^2$ of the origin, whether or not the event $\{\o_j \in G_1\}$ holds. Given this information, we know which paths of length $L$ and starting at $i + \vec e_2$ are good and we define $\g^*$ as the smallest one in the order defined above. Next, we reveal the \emph{last} $j^* \in \g^*$ such that $\{\o_{j^*} \in G_2\}$. Note that in doing so we do not need to observe whether or not the event $\{\o_j \in G_2 \}$ holds for any earlier $j$ (i.e., before $j^*$ in $\g^*$). Finally, defining $\g \subset \g^*$ to be the part of $\g^*$ before $j^*$, we reveal $\o_j$ for all $j \in \bbZ^2$, except for $j = i$ and $j \in \g$. 

\begin{figure}[ht]
  \centering
\begin{tikzpicture}[>=latex,scale=1]
[x=0.5cm, y=0.5cm]
\begin{scope}
\draw [help lines] (0,0) grid [step=0.5] (4,4);
\draw [line width=3pt ] (0,0.5)--(1,0.5)--(1,1)--(2,1)--(2,2.5)--(3,2.5)--(3,3.5);
\draw [line width=1pt ] (3,3.5)--(3.5,3.5)--(3.5,4);
\node at (1.7,1.5) {$\gamma$};
\node at (3.8,4) {$\gamma^*$};
\node at (3.7,3.3) {$\xi$};
\node at (0,0) {$\bullet$};
\node at (0,-0.5) {$0$};
\node at (3.5,3.5) {$\bullet$};
\end{scope}
\end{tikzpicture}
  \caption{The minimal good path $\g^*$, the position of the first
    super-good vertex $\xi$ encountered while traveling backward along
    $\g^*$, and the subpath $\g \subset \g^*$ (thick black) connecting $\vec e_2$ to a neighbour of $\xi$. }
\label{fig:1}
\end{figure} 

At the end of this process we are left with a (conditional) probability measure $\nu$ on $S^{\g \cup\{i\}}$. We will then apply convexity and Proposition~\ref{lem:gen-Poincare} to this measure. We now detail the above procedure.

\begin{proof}[Proof of part (B) of Theorem~\ref{thm:CPI}] 
Let $\d > 0$ be given by Proposition~\ref{lem:MT:prop34}, and assume that the events $G_2 \subseteq G_1 \subseteq S$ satisfy $\max\big\{ p_2, (1-p_1) (\log p_2)^2 \big\} \le \d$. By Proposition~\ref{lem:MT:prop34}, we have
\begin{equation}\label{eq:MT:prop34:repeat}
\var(f) \leq 4 \sum_{i \in \bbZ^2} \mu\big( \id_{\G_i} \var_i(f) \big)
 \end{equation}
for every local function $f$. We will bound each term of the sum in~\eqref{eq:MT:prop34:repeat}. Using translation invariance, it will suffice to consider the term $i = (0,0)$. 

For each $\o \in \G_{(0,0)}$, let $\g^* = \g^*(\omega)$ denote the smallest $\o$-good oriented path of length $L$ starting from $\vec e_2$, and note that $\g^*$ is $\o$-super-good, since $\o \in \G_{(0,0)}$. Let $\xi = \xi(\omega) \in \g^*$ be the first super-good vertex encountered while travelling along $\g^*$ backwards, \ie from its last point to its starting point $\vec e_2$. Finally, let $\g$ be the portion of $\g^*$ starting at $\vec e_2$ and ending at the vertex preceding $\xi$ in $\g^*$. 

We next perform the series of conditionings on the measure $\mu$ that were described informally above. Let $\L$ be the box of side-length $4/p_2^2$ centred at the origin. We first condition on the event $\G_{(0,0)}$ and on the $\s$-algebra generated by the events 
$$\big\{ \{\o_j \in G_1\} : j \in \L \setminus \{(0,0)\} \big\}.$$
Note that, since we are conditioning on the event $\G_{(0,0)}$, these events determine $\g^*$. Next we condition on the position of $\xi$ on $\g^*$; this determines the path $\g = (i^{(1)},\dots, i^{(n)})$. Finally we condition on all of the variables $\o_j$ with $j \not\in \g \cup \{(0,0)\}$. 
Let $\nu$ be the resulting conditional measure and observe that $(S^{\g\cup\{(0,0)\}},\nu)$ is a product probability space of the form $\otimes_{j \in \g \cup \{(0,0)\}}(S_j,\nu_j),$ with $(S_{(0,0)},\nu_{(0,0)}) = (S, \hat \mu)$ and $(S_j,\nu_j) = \big( G_1, \hat\mu(\cdot \tc G_1) \big)$ for each $j \in \g$. Notice that 
\begin{equation}\label{eq:CPI:convexity}
\mu\big( \id_{\G_{(0,0)}} \var_{(0,0)}(f) \big) = \mu\Big( \id_{\G_{(0,0)}} \, \nu\big( \var_{\nu_{(0,0)}}(f) \big) \Big) \le \, \mu\big( \id_{\G_{(0,0)}} \var_\nu(f) \big),
\end{equation}
because $\nu\big( \var_{\nu_{(0,0)}}(f) \big) \le \var_\nu(f)$, by convexity.

We can now bound $\var_\nu(f)$ from above by applying Proposition~\ref{lem:gen-Poincare} to the measure $\nu = \otimes_{j \in \g\cup \{(0,0)\}} (S_j,\nu_j),$ with the super-good event $G_2$ as the constraining event $S_j^g$. Observe that $\nu\big( S_{(0,0)}^g \big) = \hat \mu(G_2) = p_2$ and $\nu\big( S_j^g \big) =  \hat \mu\big( G_2 \, | \, G_1 \big) = p_2 / p_1$ for each $j \in \g$. The first Poincar\'e inequality~\eqref{eq:Poincare:genEast} in Proposition~\ref{lem:gen-Poincare} therefore gives 
\begin{equation}\label{eq:7}
\mu\big( \id_{\G_{(0,0)}} \var_\nu(f) \big) \leq \, \vec T(p_2) \cdot \mu\bigg( \id_{\G_{(0,0)}} \sum_{i \in \g \cup \{(0,0)\}} \nu\Big( \id_{\{\o_{m(i)}\in G_2\}} \var_{\nu_{i}}(f) \Big) \bigg), 
\end{equation}
where $m(i)$ is the next point on the path $\g^*$ after $i$ (\ie $m(i)$ is either $m(i) = i + \vec e_1$ or $m(i) = i + \vec e_2$) and 
\[
\vec T(p_2) \leq \frac{1}{p_2} \sup \big\{ T_{\text{\tiny East}}(n,\bar \a) : n \leq L \big\} \leq \, p_2^{-O( \log(1/p_2) )}, 
\]
by~\eqref{eq:scaling}. Recall that in Definition \ref{def:gen:East} the constraint for the last point is identically equal to one (this is in order to guarantee irreducibility of the chain), and observe that this condition holds in the above setting because, by construction, $\o_\xi\in G_2$.

Finally, we claim that~\eqref{eq:7} implies that
\begin{multline}\label{eq:CPI:extended:sum}
\mu\big( \id_{\G_{(0,0)}} \var_\nu(f) \big) \leq \, \vec T(p_2) \, \sum_{i \in \L} \bigg( \mu\Big( \id_{\{\o_{i + \vec e_1} \in G_2\}} \id_{\{\o_{i-\vec e_1} \in G_1\}} \var_i(f \tc G_1) \Big) \\
+ \mu\Big(\id_{\{\o_{i+\vec e_2} \in G_2\}}\id_{\{\o_{j}\in G_1\, \forall j\in \bbL^+(i)\}} \big( \var_i(f) + \var_i(f\tc G_1) \big) \Big) \bigg).
\end{multline}
Indeed, note that $\var_{\nu_{(0,0)}}(f) = \var_{(0,0)}(f)$ and that $\var_{\nu_{i}}(f) = \var_i(f \tc G_1)$ for each $i \in \g$, and recall that, by construction, $\o_{i + \vec e_1},\o_{i - \vec e_1} \in G_1$ for every $i \in \g$. Therefore, for each $i \in \g$, if $m(i) = i + \vec e_1$ then $\o_{i-\vec e_1} \in G_1$, and if $m(i) = i + \vec e_2$ then $\o_{j}\in G_1$ for each $j \in \bbL^+(i) = i + \big\{ \vec e_1, \vec e_2 - \vec e_1 \big\}$. Moreover, the event $\G_{(0,0)}$ implies that $\o_j \in G_1$ for each $j \in \bbL^+((0,0))$. Therefore every term of the right-hand side of~\eqref{eq:7} is included in the right-hand side of~\eqref{eq:CPI:extended:sum}, and hence~\eqref{eq:7} implies~\eqref{eq:CPI:extended:sum}, as claimed.

Now, combining~\eqref{eq:CPI:extended:sum} with~\eqref{eq:MT:prop34:repeat} and~\eqref{eq:CPI:convexity}, and noting that $\var_{i}(f) \geq p_1 \var_i(f \tc G_1)$ and that $|\L| \le p_2^{-O(1)}$, we obtain
\begin{multline*}
\var(f) \leq \, p_1^{-1} p_2^{-O(1)} \, \vec T(p_2) \sum_{i \in \bbZ^2} \bigg( \mu\Big( \id_{\{\o_{i + \vec e_1} \in G_2\}} \id_{\{\o_{i-\vec e_1} \in G_1\}} \var_i\big(f \tc G_1 \big) \Big) \\
+ \mu\Big(\id_{\{\o_{i+\vec e_2} \in G_2\}}\id_{\{\o_{j}\in G_1\, \forall j\in \bbL^+(i)\}} \var_i(f) \Big) \bigg),
\end{multline*}
which implies the oriented Poincar\'e inequality~\eqref{eq:9East}, as required.
 
The proof of the unoriented inequality~\eqref{eq:9FA} is almost the same, except we will use the second
Poincar\'e inequality~\eqref{eq:Poincare:genFA1f} in Proposition~\ref{lem:gen-Poincare}, instead of~\eqref{eq:Poincare:genEast}. To spell out the details, we obtain 
\begin{equation}\label{eq:CPI:FA1f:app}
\mu\big( \id_{\G_{(0,0)}} \var_\nu(f) \big) \leq \, T(p_2) \cdot \mu\bigg( \id_{\G_{(0,0)}} \sum_{i \in \g \cup \{(0,0)\}} \nu\Big( c_i \var_{\nu_{i}}(f) \Big) \bigg), 
\end{equation}
where $c_i$ is the indicator of the event that $G_2$ holds for at least one of the neighbours of $i$ on the path $\g^*$, and 
\[
T(p_2) \leq \frac{1}{p_2} \sup_{n \leqslant L} T_{\text{\tiny FA}}(n,\bar \a) = p_2^{-O(1)}, 
\]
by~\eqref{eq:scaling}. Note that the constraint for the last point is again identically equal to one since $\o_\xi\in G_2$. It follows (cf.~\eqref{eq:CPI:extended:sum}) that
\begin{multline}\label{eq:CPI:extended:sum:again}
\mu\big( \id_{\G_{(0,0)}} \var_\nu(f) \big) \leq \, T(p_2) \, \sum_{i \in \L} \sum_{\eps = \pm 1} \bigg( \mu\Big( \id_{\{\o_{i + \eps\vec e_1} \in G_2\}} \id_{\{\o_{i - \eps\vec e_1} \in G_1\}} \var_i\big( f \tc G_1 \big) \Big) \\
+ \mu\Big(\id_{\{\o_{i + \eps\vec e_2} \in G_2\}}\id_{\{\o_{j}\in G_1\, \forall j\in \bbL^\eps(i)\}} \big( \var_i(f) + \var_i\big( f \tc G_1 \big) \big) \Big) \bigg),
\end{multline}
since $\o_{i + \vec e_1},\o_{i - \vec e_1} \in G_1$ for every $i \in \g$, and the event $\G_{(0,0)}$ implies that $\o_j \in G_1$ for each $j \in \bbL^+((0,0))$. In particular, note that if $i  \in \g$ and $i + \vec e_2 \in \g$, then $\o_j \in G_1$ for each $j \in \bbL^+(i) = \bbL^-(i + \vec e_2) = i + \big\{ \vec e_1, \vec e_2 - \vec e_1 \big\}$. Therefore, as before, every term of the right-hand side of~\eqref{eq:CPI:FA1f:app} is included in the right-hand side of~\eqref{eq:CPI:extended:sum:again}.

Finally, combining~\eqref{eq:CPI:extended:sum:again} with~\eqref{eq:MT:prop34:repeat} and~\eqref{eq:CPI:convexity}, and since $\var_{i}(f) \geq p_1 \var_i(f \tc G_1)$ and $|\L| \le p_2^{-O(1)}$, we obtain
\begin{multline*}
\var(f) \leq \, p_1^{-1} p_2^{-O(1)} \, T(p_2) \sum_{i \in \bbZ^2} \sum_{\eps = \pm 1} \bigg( \mu\Big( \id_{\{\o_{i + \eps \vec e_1} \in G_2\}} \id_{\{\o_{i - \eps \vec e_1} \in G_1\}} \var_i\big( f \tc G_1 \big) \Big) \\
+ \mu\Big(\id_{\{\o_{i + \eps\vec e_2} \in G_2\}}\id_{\{\o_{j}\in G_1\, \forall j\in \bbL^\eps(i)\}} \var_i(f) \Big) \bigg),
\end{multline*}
which gives the unoriented Poincar\'e inequality~\eqref{eq:9FA}, as claimed, and hence completes the proof of Theorem~\ref{thm:CPI}.
\end{proof}

\section{Renormalization and spreading of infection}
\label{sec:strategy} 

In this section we shall define the setting to which we will apply Theorem \ref{thm:CPI} in order to bound from above the relaxation time, and hence the mean infection time, of supercritical and critical KCM. We will begin with a very brief informal description, before giving (in Section \ref{sec:setting}) the precise definition. We will then, in Sections~\ref{sec:supercritical:bootstrap} and~\ref{sec:critic:spread}, state two results from the theory of bootstrap percolation that will play an instrumental role in the proofs of Theorems~\ref{mainthm:1} and~\ref{mainthm:2}.  

Our basic strategy is to partition the lattice $\bbZ^2$ into disjoint rectangular ``blocks'' $\{V_i\}_{i\in \bbZ^2}$, whose size is  adapted to the bootstrap update family $\cU$. To each block $V_i$ we associate a block random variable $\o_i$, which is just the collection of i.i.d.~$0/1$ Bernoulli($p$) variables $\{\o_x\}_{x\in V_i}$ attached to each vertex of the block. In order to avoid confusion we will always use the letters $i,j,\dots$ for the labels of quantities associated to blocks, and the letters $x,y,\dots$ for the labels of the quantities associated to vertices of $\bbZ^2$. We will apply Theorem~\ref{thm:CPI} to the block variables $\{\o_i\}_{i\in \bbZ^2}$.       

\subsection{A concrete general setting} \label{sec:setting}

Let $v$ and $v^\perp$ be orthogonal rational directions in the first and second quadrant of $\bbR^2$ respectively. Let $\vec v$ be the vector joining the origin to the first site of $\bbZ^2$ in direction $v$, and similarly for $\vec v^\perp$. Let $n_1 \ge n_{2}$ be (sufficiently large) even integers, and set 
\begin{equation}\label{eq:basicrect}
R \, := \, \big\{ x \in \bbR^2 \,:\, x = \a n_1\vec v +\b n_2 \vec v^\perp,\, \a, \b \in [0,1) \big\}.
\end{equation}

The finite probability space $(S,\hat \mu)$ appearing in Section~\ref{sec:CPI} will always be of the form $S=\{0,1\}^V$, where $V = R \cap \bbZ^2$, and $\hat \mu$ is the Bernoulli$(p)$ product measure. Observe that the probability space $(S^{\bbZ^2},\mu)$ is isomorphic to $\O = \{0,1\}^{\bbZ^2}$ equipped with the Bernoulli$(p)$ product measure which, with a slight abuse of notation, we will continue to denote by $\mu$. For our purposes, a convenient isomorphism between the two probability spaces is given by a kind of tilted ``brick-wall'' partition of $\bbZ^2$ into disjoint copies of the basic block $V$ (see Figure \ref{fig:brickwall}). To be precise, for each $i = (i_1,i_2) \in \bbZ^2$, set $V_i := R_i \cap \bbZ^2$, where $R_i := R + (i_1 + i_2 / 2) n_1 \vec v + i_2 n_2 \vec v^\perp$. \medskip 

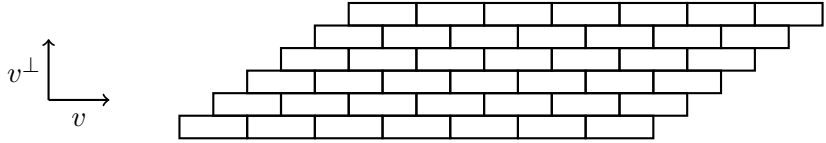
\begin{figure}[ht]
\centering
\begin{tikzpicture}[scale=0.3]
\begin{scope}[shift={(10,0)},scale=0.5]
\foreach \j in {0,1,2,3,4,5}{
\foreach \i in {0,...,6}{%
\draw [thick,opacity=1]
 (6*\i+6*\j/2,2*\j) rectangle (6*\i+6*\j/2+6,2*\j+2);}
}
\end{scope}
\begin{scope}[shift={(-3,-10)},scale=0.9]
\node at (9.5,12) {$v$};
\node at (6.8,14.5) {$v^\perp$};
\draw [->,thick] (8,13) -- (11,13); 
\draw [->,thick] (8,13)--(8,16);
\end{scope}
\end{tikzpicture}
\caption{The partition into blocks $V_i,\ i\in \bbZ^2$
} 
\label{fig:brickwall}
\end{figure}

In this partition the ``northern" and ``southern" neighbouring blocks of $V_i$ (\ie the blocks corresponding to $(i_1,i_2\pm 1)$) are shifted in the direction $\vec v$ by $\pm \, n_1 / 2$ w.r.t.~$V_i$. With this notation, and given $\o\in S^{\bbZ^2}$, it is then convenient to think of the variable $\o_i \in S$ as being the collection $\{ \o_x \}_{x \in V_i} \in \{0,1\}^{V_i}$. The local variance term $\var_i(f)$ (\ie the variance of $f$ w.r.t.~the variable $\o_i$ given all the other variables $\{\o_j\}_{j \neq  i}$), which appears in the various constrained Poincar\'e inequalities in the statement of Theorem~\ref{thm:CPI}, is then equal to the variance $\var_{V_i}(f)$ w.r.t.~the i.i.d.~Bernoulli($p$) variables $\{ \o_x \}_{x \in V_i}$, given all of the other variables $\{ \o_y \}_{y \in \bbZ^2 \setminus V_i}$. 

From now on, $\o$ will always denote an element of $\{0,1\}^{\bbZ^2}$ and, given $\L \subset \bbR^2$, we will write $\o_\L$ for the collection of i.i.d. random variables $\{ \o_x \}_{x \in \L \cap \bbZ^2}$, and $\mu_\L$ for their joint product Bernoulli($p$) law. We will say that $\L$ is \emph{empty} (or \emph{empty in $\o$}) if $\o$ is identically equal to $0$ on $\L \cap \bbZ^2$, and similarly that $\L$ is \emph{filled} (or \emph{completely occupied}) if $\o$ is identically equal to $1$ on $\L \cap \bbZ^2$.

We now turn to the definitions of the good and super-good events $G_2 \subset G_1 \subseteq S$. The good event $G_1$ will depend on the update family $\cU$, and will (roughly speaking) approximate the event that the block $V_i$ can be ``crossed" in the $\cU$-bootstrap process with the help of a constant-width strip connecting the top and bottom of $V_i$. For supercritical models this event is trivial, and therefore $G_1$ is the entire space $S$; for critical models, on the other hand, $G_1$ will require the presence of empty vertices inside $V$ obeying certain model-dependent geometric constraints (see Definition~\ref{def:goodsets}, below). The super-good event $G_2$ for supercritical models will simply require that $V$ is empty. For critical models it will require that $G_1$ holds, and additionally that there exists an empty subset $\cR$ of $V$, called a \emph{quasi-stable half-ring} (see Definitions~\ref{def:half-ring} and~\ref{def:goodsets}, and Figure~\ref{fig:half-ring}) of (large) constant width, and height equal to that of $V$. 
We emphasize that the parameters $n_1,n_2$ will be chosen (depending on the model) so that the probabilities  $p_1$ and $p_2$ of the events $G_1$ and $G_2$ (respectively) satisfy the key condition
$$\lim_{q \to 0} \max{\Big\{ p_2,\big( 1 - p_1 \big) \big( \log p_2 \big)^2 \Big\}}= 0$$
that appears in part (B) of Theorem~\ref{thm:CPI}.

\subsection{Spreading of infection: the supercritical case.}\label{sec:supercritical:bootstrap}

We are now almost ready to state the property of $\U$-bootstrap percolation (proved by Bollob\'as, Smith and Uzzell~\cite{BSU}) that we will need when $\U$ is supercritical, i.e., when there exists an open semicircle $C\subset S^1$ that is free of stable directions. If $\U$ is rooted, then we may choose $-v$ (in the construction of the rectangle $R$ and of the partition $\{V_i\}_{i\in \bbZ^2}$ described in Section~\ref{sec:setting}) to be the midpoint of any such semicircle; if $\U$ is unrooted, on the other hand, then $C$ can be chosen in such a way that $-C$ also has no stable directions, and we can choose $v$ to be the midpoint of any such semicircle. 



Recall that $[ V_i ]_\cU$ denotes the closure of $V_i = R_i \cap \bbZ^2$ under the
$\cU$-bootstrap process. The following result, proved in~\cite{BSU}, states that a large enough rectangle can infect the rectangle to its ``left" (i.e., in direction $-v$) under the $\cU$-bootstrap process, and if $\U$ is unrooted then it can also infect the rectangle to its ``right" (i.e., in direction $v$).

\begin{proposition}\label{prop:bootsc}  
Let $\U$ be a supercritical two-dimensional update family. If $n_1$ and $n_2$ are sufficiently large, then the following hold:
\begin{itemize}
\item[$(i)$] If $\U$ is unrooted, then $V_{(-1,0)} \cup V_{(1,0)} \subset [ V_{(0,0)} ]_\cU$.\smallskip
\item[$(ii)$] If $\U$ is rooted, then $V_{(-1,0)} \subset [ V_{(0,0)} ]_\cU$.
\end{itemize}
\end{proposition}
\begin{remark}
\label{rem:left-right}By definition, in the rooted case the semicircle $-C$ contains some stable
directions. Thus, $V_{(1,0)} \not\subset [ V_{(0,0)} ]_\cU$.
\end{remark}
The proof of Proposition~\ref{prop:bootsc} in~\cite{BSU} is non-trivial, and required some important innovations, most notably the notion of ``quasi-stable directions" (see Definition~\ref{def:quasi-stable}, below). We will therefore give here only a brief sketch, explaining how one can read the claimed inclusions out of the results of~\cite{BSU}

\begin{proof}[Sketch proof of Proposition~\ref{prop:bootsc}]
Both parts of the proposition are essentially immediate consequences of the following claim: if $R$ is a sufficiently large rectangle with two sides parallel to $w \in S^1$, and the semicircle centred at $w$ is entirely unstable, then $[R]_\U$ contains \emph{every} element of $\bbZ^2$ that can be reached from $R$ by travelling in direction $w$. This claim follows from~\cite{BSU}*{Lemma~5.5}, since in this setting all of the quasi-stable directions in $\cS_U'$ (see~\cite{BSU}*{Section~5.3}) are unstable (since they are contained in the semicircle centred at $w$), and if $u$ is unstable then the empty set is a $u$-block (see~\cite {BSU}*{Definition~5.1}). We refer the reader to~\cite{BSU}*{Sections~5 and~7} for more details. 
\end{proof}

\subsection{Spreading of infection: the critical case.}\label{sec:critic:spread}

We next turn to the more complicated task of precisely defining the good and super-good events for  critical update families. In this subsection we will lay the groundwork for the precise definitions of these events (which we defer until Section~\ref{sec:critic-ub}, see Definition~\ref{def:goodsets}) by recalling some definitions from~\cite{BSU,BDMS}, and introducing the key new objects needed for the proof of Theorem~\ref{mainthm:2}, which we call ``quasi-stable half-rings" (see Definition~\ref{def:half-ring} and Figure~\ref{fig:half-ring}, below). Throughout this subsection, we will assume that $\cU$ is a critical update family with difficulty $\a \in [1,\infty)$ and bilateral difficulty $\b \in [\a,\infty]$ (see Definition~\ref{def:alpha}). Recall that we say that $\U$ is $\a$-rooted if $\beta \ge 2\a$, and that $\U$ is $\b$-unrooted otherwise.

We begin by noting an important property of the set of stable directions $\cS(\cU)$.

\begin{lemma}\label{lem:stableset}
If $\beta < \infty$ then $\S(\U)$ consists of a finite number of isolated, rational directions. Moreover, if $\U$ is $\b$-unrooted and $\alpha(u^*) = \max\big\{ \a(u) : u \in \cS(\cU) \big\}$, then $\alpha(u) \le \b$ for every $u \in \S(\U) \setminus \{u^*,-u^*\}$.  
\end{lemma}

\begin{proof}
By~\cite {BSU}*{Theorem 1.10}, $\S(\U)$ is a finite union of rational closed intervals of $S^1$, and by~\cite {BSU}*{Lemma~5.2} (see also~\cite{BDMS}*{Lemma~2.7}), if $u \in \S(\cU)$ is a rational direction, then $\alpha(u) < \infty$ if and only {{if}} $u$ is an isolated point of $\cS(\cU)$. Thus, if one of the intervals in $\S(\U)$ is not an isolated point, then there exist two non-opposite stable directions in $S^1$, each with infinite difficulty, and so $\b = \infty$.

Now, suppose that $\U$ is $\b$-unrooted, and that $u \in \cS(\cU)$ satisfies $\alpha(u) > \beta$ and $u \not\in \{u^*,-u^*\}$. Then $u$ and $u^*$ are non-opposite stable directions in $S^1$, each
with difficulty strictly greater than $\beta$, which contradicts the definition of $\b$. 
\end{proof} 

In particular, if $\U$ is $\b$-unrooted then Lemma~\ref{lem:stableset} guarantees the existence of an open semicircle $C$ such that $(C \cup -C) \cap \cS(\cU)$ consists of finitely many directions, each with difficulty at most $\b$. The next lemma provides a corresponding property for $\alpha$-rooted models.

\begin{lemma}\label{cor:Ccritic}
If $\U$ is $\alpha$-rooted, then there exists an open semicircle $C$ such that $C \cap \cS(\cU)$ consists of finitely many directions, each with difficulty at most $\a$. 
\end{lemma}

\begin{proof}
By Definition~\ref{def:alpha}, there exists an open semicircle $C$ such that each $u \in C$ has difficulty at most $\a$. Since $\U$ is critical (and hence $\a$ is finite), it follows from~\cite{BSU}*{Lemma~5.2} (cf. the proof of Lemma~\ref{lem:stableset}) that each $u \in C$ is either unstable, or an isolated element of $\cS(\cU)$, and hence $C \cap \cS(\cU)$ is finite, as claimed. 
\end{proof}

Let us fix (for the rest of the subsection) an open semicircle $C$, containing finitely many stable directions, and such that the following holds:
\begin{itemize}
\item if $\U$ is $\a$-rooted then $\a(v) \le \a$ for each $v \in C$;\smallskip
\item if $\U$ is $\b$-unrooted then $\a(v) \le \b$ for each $v \in C \cup - C$. 
\end{itemize}
Let us also choose $C$ such that its mid-point $u$ belongs to $\bbQ_1$, and denote by $\pm u^\perp$ the boundary points of $C$. When drawing pictures we will always think of $C$ as the semicircle $(-\pi/2, \pi/2)$, though we emphasize that we do not assume that $u$ is parallel to one of the axes of $\Z^2$. We remark that the values of $\a(u^\perp)$ and $\a(-u^\perp)$ will not be important: we will only need to use the fact that they are both finite.

We are now ready to define one of the key notions from~\cite{BSU}, the set of quasi-stable directions. These are directions that are not (necessarily) stable, but which nevertheless it is useful
to treat as if they were. For any $v \in S^1$, let us write $\hat v$ for the direction in $S^1$ that is symmetric to $v$ w.r.t.~the mid-point $u$ of $C$.

\begin{definition}[Quasi-stable directions]\label{def:quasi-stable}
We say that a direction $v\in \Qb_1$ is quasi-stable if either $v$ or $\hat v$ is a member of the set
$$\{u \} \cup \cS(\cU) \cup \bigg( \bigcup_{X \in \U} \bigcup_{x \in X} \big\{ v \in S^1 : \< v,x \> = 0 \big\} \bigg).$$
\end{definition}

Observe that there are only finitely many quasi-stable directions in $C$ (and, if $\beta < \infty$, only finitely many in $S^1$). The key property of the family of quasi-stable directions is given by the following lemma, which allows us to empty the sites near the corners of ``quasi-stable half-rings" (see Definition~\ref{def:half-ring}, below). {{Recall that we write $\ell_v$ for the discrete line $\{x \in \Z^2 : \< x,v \> = 0 \}$.}}

\begin{lemma}[\cite{BSU}*{Lemma~5.3}]\label{lem:quasi}
For every pair $v,v'$ of consecutive quasi-stable directions there exists an update rule $X$ such that
$X \subset \big(\H_v\cup\ell_v\big)\cap\big(\H_{v'}\cup\ell_{v'}\big).$
\end{lemma}

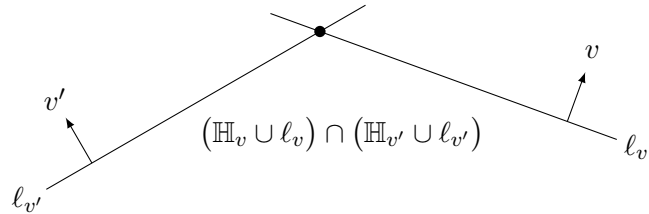
\begin{figure}[ht]
  \centering
  \begin{tikzpicture}[>=latex, scale=0.7]
    \node[circle,fill,inner sep=0pt,minimum size=0.15cm] at (0,0) {};
    \draw (30:1) -- (-150:6) ++(-150:0.4) node {$\ell_{v'}$};
    \draw (160:1) -- (-20:6) ++(-20:0.4) node {$\ell_{v}$};
    \draw[->] (-150:5) -- ++(120:1);
    \path (-150:5) ++(120:1.4) node {$v'$};
    \draw[->] (-20:5) -- ++(70:1);
    \path (-20:5) ++(70:1.4) node {$v$};
    \node at (-78:2) {$\big(\H_v\cup\ell_v\big)\cap\big(\H_{v'}\cup\ell_{v'}\big)$};
  \end{tikzpicture}
\end{figure}

\begin{proof}
The statement was proved in~\cite{BSU} (see also~\cite{BDMS}*{Lemma~3.5}) for the family $\cS(\cU) \cup \big( \bigcup_{X \in \U} \bigcup_{x \in X} \big\{ v \in S^1 : \< v,x \> = 0 \big\} \big)$ of quasi-stable directions, and it therefore holds for any superset of this family. 
\end{proof}

In order to define quasi-stable half-rings, we first need to introduce some additional notation:


\begin{definition}\label{def:strips}
Let $v \in \Qb_1$ with $\a(v)\le \a$. A \emph{$v$-strip} $S$ is any closed parallelogram in $\bbR^2$ with long sides perpendicular to $v$ and short sides perpendicular to $u^\perp$. 
\begin{itemize}
\item The $+$-boundary and $-$-boundary of $S$, denoted $\partial_+ (S)$ and $\partial_- (S)$ respectively, are the sides of $S$ with outer normal $v$ and $-v$.
\item The external boundary $\partial^{\rm ext} (S)$ is defined as that translate of $\partial_+ (S)$ in the $v$-direction which captures for the first time a new lattice point not already present in $S$. 
\item Given $\l > 0$, we define $\partial_{\l}^{\rm ext}(S)$ as the portion of $\partial^{\rm ext}(S)$ at distance $\l$ from its endpoints (see Figure \ref{fig:geom}). 
\end{itemize}
\end{definition}

\begin{figure}[ht]
  \centering
  \begin{tikzpicture}[>=latex, scale=0.5]
\begin{scope}
\draw [ultra thick] (-3,1.5)--(3,-1.5);
\draw [thin,->] (-3,1.5)--(-4,2);
\draw [thin,->] (3,-1.5)--(4,-2);
 \draw [dashed] (-3,1.5)--(-6,3);
 \draw [dashed] (3,-1.5)--(6,-3);
 \draw [->,thick] (0,0)--(-0.5,-1);
\node at (-0.8,-0.5) {$v$};
 \draw [->,thick] (-1.2,2)--(-1.2,3.2);
\node at (-0.3,3.1) {$u^\perp$};
\end{scope}
\begin{scope}[shift={(0.15,0.3)}]
\draw [fill=lightgray] (-4,2)--(4,-2)--(9,-2)--(1,2)--cycle;
\node at (1.3,-2) {$\partial^{\rm ext}_{\l}{(}S{)}$};
\node at (0.9,0.3) {$\partial_{+}(S)$};
\end{scope}
  \end{tikzpicture}
 \caption{A $v$-strip $S$, the $+$-boundary of $S$, the 
   external boundary (solid segment), and its subset $\partial_{\l}^{\rm
  ext}(S)$ (thick solid segment)}
\label{fig:geom}
\end{figure}
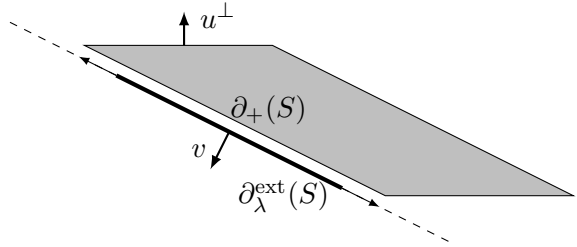
If $v$ is a stable direction, then a $v$-strip needs some ``help" from other infected sites in order to infect its external boundary (in the $\U$-bootstrap process). Our next ingredient (also first proved in~\cite{BSU}) provides us with a set that suffices for this purpose. 

Let $v$ be a quasi-stable direction with difficulty $\alpha(v) \le \a$, and let $Z_v \subset \bbZ^2$ be a set of cardinality $\a$ such that $[\H_v \cup Z_v]_\U \cap \ell_v$ is infinite. (In the language of~\cite{BDMS}, $Z_v$ is called a \emph{voracious set}.) The following lemma (see~\cite{BSU}*{Lemma~5.5} and~\cite{BDMS}*{Lemma~3.4}) states that if $S$ is a sufficiently large $v$-strip, then a bounded number of translates of $Z_v$, together with $S \cap \bbZ^2$, are sufficient to infect $\partial_{\l}^{\rm ext}(S{)}$ for some $\lambda = O(1)$. 

\begin{lemma}\label{lem:strip}
There exist $\l_v>0$, $T_v = \{a_1,\dots,a_r\} \subset {{\ell_v}}$ and $b\in {{\ell_v}}$ such that the following holds. If $S$ is a sufficiently large $v$-strip such that $\partial^{\rm ext} {(}S{)} \cap \bbZ^2 \subset \ell_v$, then 
\begin{equation}\label{eq:fillpartialS}
\partial_{\l_v}^{\rm
  ext}{(}S{)}\cap \bbZ^2 \subset \big[ (S\cap \bbZ^2) \cup (Z_v+a_1+k_1b)\cup\dots\cup(Z_v+a_r+k_rb) \big]_\U
\end{equation}
for every $k_1,\dots, k_r \in \Z$ such that $a_i + k_i b \in \partial_{\l_v}^{\rm ext}{(}S{)}$ for every $i \in [r]$.
\end{lemma}

Let us fix, for each quasi-stable direction $v \in C$, a constant $\l_v > 0$, a set $T_v = \{a_1,\dots,a_r\} \subset {{\ell_v}}$ and a site $b \in {{\ell_v}}$ given by Lemma~\ref{lem:strip}. If $S$ is a sufficiently large $v$-strip such that $\partial^{\rm ext} {(}S{)} \cap \bbZ^2 \subset \ell_v + x$ for some $x \in \bbZ^2$, 
then we will refer to any set of the form 
\begin{equation}\label{def:helping:set}
\big( (Z_v+a_1+k_1b) \cup \dots \cup (Z_v+a_r+k_rb) \big) + x,
\end{equation}
with $a_i + k_i b + x \in \partial_{\l_v}^{\rm ext}{(}S{)}$ for every $i \in [r]$, as a \emph{helping set} for $S$.

We are finally ready to define the key objects we will use to control the movement of empty sites in a critical KCM, the \emph{quasi-stable half-rings}. These are non{-}self-intersecting polygons, obtained by patching together suitable $v$-strips corresponding to quasi-stable directions (see Figure \ref{fig:half-ring}). Recall from Definition~\ref{def:quasi-stable} that, by construction, the set of quasi-stable directions in $C$ is symmetric w.r.t.~the midpoint $u$ of $C$.

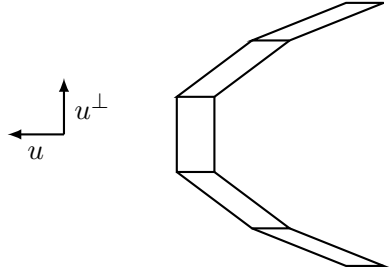
\begin{figure}[ht]
  \centering
  \begin{tikzpicture}[>=latex, scale=0.25]
\begin{scope}[rotate around={180:(10,10)}]
\draw [thick] (0,0)--(5,2)--(7,2)--(2,0)--(0,0);  
\draw [thick] (5,2)--(9,5)--(11,5)--(7,2);
\draw [thick] (9,5)--(9,9)--(11,9)--(11,5);
\draw [thick] (9,9)--(5,12)--(7,12)--(11,9);
\draw [thick] (5,12)--(0,14)--(2,14)--(7,12);
\end{scope}
\begin{scope}[shift={(-5,0)}]
\node at (6.5,12) {$u$};
\node at (9.5,14.5) {$u^\perp$};
\draw [->,thick] (8,13) -- (5,13); 
\draw [->,thick] (8,13)--(8,16);
\end{scope}
  \end{tikzpicture}
   \caption{A quasi-stable half-ring.}\label{fig:half-ring}
\end{figure}
 
\begin{definition}[Quasi-stable half-rings]\label{def:half-ring}
Let $(v_1,\dots, v_m)$ be the quasi-stable directions in $C$, ordered in such a way that $v_i$ and $v_{i+1}$ are consecutive directions for any $i\in [m-1]$, and $v_{i-1}$ comes before $v_i$ in clockwise order. Let $S_{v_i}$ be a $v_i$-strip with length $\ell_i$ and width $w_i$. We say that $\cR := \bigcup_{i=1}^m S_{v_i}$ is a \emph{quasi-stable half-ring} of width $w$ and length $\ell$ if the following holds: 
\begin{enumerate}[(i)]
\item $w_i = w$ and $\ell_i = \ell$ for each $i \in [m]$;
\item $S_{v_i}\cap S_{v_j} = \emptyset$, unless $v_i$ and $v_j$ are consecutive directions, in which case the two strips share exactly one of their short sides and no other point.
\end{enumerate}
\end{definition}

We can finally formulate the ``spreading of infection'' result that we will need later. Given a quasi-stable half-ring $\cR$, we will write $\cR^*$ for the quasi-stable half-ring $\cR + s u$, where $s > 0$ is minimal such that $\big( \cR^* \setminus \cR \big) \cap \bbZ^2$ is non-empty. Also, for any set $U \subset \bbZ^2$, let us write $[A]^U_\cU$ for the closure of $A$ under the $\cU$-bootstrap process restricted to $U$.

\begin{proposition}\label{prop:critic:spread}
There exists a constant $\lambda = \lambda(\cU) > 0$ such that following holds. Let $\cR$ be a quasi-stable half-ring of width $w$ and length $\ell$, 
where $w, \ell \geq \l$. 
Let $U$ be the set of points of $\bbZ^2$ within distance $\lambda$ of $\cR \cup \cR^*$, and let $Z_i$ be a helping set for $S_{v_i}$ for each $i \in [m]$. Then
$$\cR^* \cap \bbZ^2 \subset \big[ \big( \cR \cap \bbZ^2 \big) \cup Z_1 \cup \cdots \cup Z_m \big]^U_\cU.$$ 
\end{proposition}

\begin{proof}
This is a straightforward consequence of Lemmas~\ref{lem:quasi} and~\ref{lem:strip}. To see this, note first that, by Lemma~\ref{lem:strip}, the closure of $\big( \cR \cap \bbZ^2 \big) \cup Z_1 \cup \cdots \cup Z_m$ under the $\cU$-bootstrap process contains all points of $\cR^* \cap \bbZ^2$ except possibly those that lie within distance $O(1)$ of a corner of $\cR$. Moreover, the path of infection described in the proof of Lemma~\ref{lem:strip} in~\cite{BSU,BDMS} only uses sites within distance $O(1)$ of the $v$-strip $S$. Thus, if $\lambda$ is chosen large enough, we have $\partial_{\l/4}^{\rm ext}(S_{v_i}) \cap \bbZ^2 \subset \big[ \big( \cR \cap \bbZ^2 \big) \cup Z_i \big]^U_\cU$ for each $i \in [m]$. 

Now, by Lemma~\ref{lem:quasi}, it follows that the set $\big[ \big( \cR \cap \bbZ^2 \big) \cup Z_i \cup Z_{i+1} \big]^U_\cU$ contains the remaining sites of $\partial^{\rm ext}(S_{v_i}) \cap \bbZ^2$ and $\partial^{\rm ext}(S_{v_{i+1}}) \cap \bbZ^2$ that lie within distance $\lambda/4$ of the intersection of $S_{v_i}$ and $S_{v_{i+1}}$. Indeed, these sites can be infected one by one, working towards the corner, using sites in $\cR \cup \partial_{\l/4}^{\rm ext} (S_{v_i}) \cup \partial_{\l/4}^{\rm ext}(S_{v_{i+1}})$. Since this holds for each $i \in [m-1]$, it follows that the whole of $\cR^* \cap \bbZ^2$ is infected, as claimed.
\end{proof}


Given a quasi-stable half-ring $\cR$ of width $w$, we will write $\cR'$ for the quasi-stable half-ring $\cR + w u$, i.e., the minimal translate of $\cR$ in the $u$-direction such that $\cR \cap \bbZ^2$ and $\cR' \cap \bbZ^2$ are disjoint.

\begin{corollary}
\label{cor:critic:spread}
There exists a constant $\lambda = \lambda(\cU) > 0$ such that following holds. Let $\cR$ be a quasi-stable half-ring of width $w$ and length $\ell$, 
and suppose that $w \geq \lambda$ and $\ell \geq \lambda$. 
Let $U$ be the set of points of $\bbZ^2$ within distance $\lambda$ of $\cR\cup \cR'$, and let $A \subset U$ be such that for any quasi-stable direction $v$, and any $v$-strip $S_v$ such that $\partial^{\rm ext}(S_{v}) \cap \cR'$ has length at least $\ell$, there exists a helping set for $S_v$ in $A$. Then
$$\cR' \cap \bbZ^2  \subset \big[ \big( \cR \cap \bbZ^2 \big)\cup A \big]^U_\cU.$$ 
\end{corollary}

\begin{proof}
By construction, each $v_i$-strip of $\cR$ has a helping set in $\cR'$. Therefore, by Proposition~\ref{prop:critic:spread}, the $\cU$-bootstrap process restricted to $U$ is able to infect the quasi-stable half-ring $\cR^*$. We then repeat with $\cR$ replaced by $\cR^*$, and so on, until the entire quasi-stable half-ring $\cR'$ has been infected.     
\end{proof}

Observe that, under the additional assumption that each quasi-stable direction $v$ has a helping set contained in $\ell_v$, we may choose $A$ to be a subset of $\cR'$, but that in general we may (at some stage) need a helping set not contained in $\cR'$ in order to advance in the $u$-direction.


\begin{remark}\label{rem:helping-sets}
Later on, we will also need the above results in the
slightly different setting in which the first $v_1$-strip
entering in the definition of $\cR$ is longer than the others, while
all of the {other} $v_j$-strips, $j\neq 1$ have the same length.
In this case we will refer to $\cR$ as  an
\emph{elongated} quasi-stable half-ring. For
simplicity we preferred to state Proposition~\ref{prop:critic:spread}
in the slightly less general {setting above}, but exactly the same proof applies if $\cR$ is an elongated quasi-stable half-ring.   
\end{remark}

\section{Supercritical KCM: proof of Theorem~\ref{mainthm:1}}\label{sec:supercritical}

In this section we shall prove Theorem~\ref{mainthm:1}, which gives a sharp (up to a constant factor in the exponent) upper bound on the mean infection time for a supercritical KCM. We will first (in Section~\ref{sec:supercritic:unrooted}) give a detailed proof in the case that $\cU$ is unrooted, and then (in Section~\ref{sec:supercritic:rooted}) explain briefly how the proof can be modified to prove the claimed bound for rooted models.

\subsection{The unrooted case}\label{sec:supercritic:unrooted}

Let $\cU$ be a supercritical, unrooted, two-dimensional update family; we are required to show that there exists a constant $\lambda = \lambda(\cU)$ such that 
$$\bbE_\mu(\t_0) \leq q^{-\lambda}$$
for all sufficiently small $q > 0$. To do so, recall first from~\eqref{eq:mean-infection} that $\bbE_\mu(\t_0) \le \trel(q,\cU)/q$, and therefore, by Definition~\ref{def:PC}, it will suffice to prove that
\begin{equation}\label{eq:supercritical:aim}
\var(f) \leq q^{-\lambda} \sum_x \mu\big( c_x \var_x(f) \big)
\end{equation}
for some $\lambda = \lambda(\cU) > 0$ and all local functions $f$, where $c_x$ denotes the kinetic constraint for the KCM, i.e., $c_x$ is the indicator function of the event that there exists an update rule $X\in \cU$ such that $\o_y=0$ for each $y \in X + x$. We will deduce a bound of the form~\eqref{eq:supercritical:aim} from Theorem~\ref{thm:CPI} and Proposition~\ref{prop:bootsc}.  

Recall the construction and notation described in Sections~\ref{sec:setting} and~\ref{sec:supercritical:bootstrap}; in particular, recall the definitions 
of the blocks $V_i$, of the parameters $n_1$ and $n_2$ (which determine the side lengths of the basic rectangle $R$), and the choice of {$v$} as the midpoint of an open semicircle $C \subset S^1$ such that the set $C \cup -C$ contains no stable directions. As anticipated in Section \ref{sec:setting}, the choice of the good and super-good events $G_2 \subset G_1 \subseteq S$ entering in Theorem \ref{thm:CPI}, is, in this case, extremely simple. 

\begin{definition}
If $\U$ is a supercritical two-dimensional update family, then: 
\begin{enumerate}[(a)]
\item every block $V_i$ satisfies the \emph{good event} $G_1$ for $\U$ (\ie $G_1=S$);
\item a block $V_i$ satisfies the \emph{super-good event} $G_2$ for $\U$ if and only if it is empty. 
\end{enumerate}
\end{definition}

Let us fix the parameters $n_1$ and $n_2$ to be $O(1)$, but sufficiently large {{so that}} Proposition~\ref{prop:bootsc} holds. It follows that if $V_{(0,0)}$ is super-good, then the blocks $V_{(-1,0)}$ and $V_{(1,0)}$ (its nearest neighbours to the left and right respectively) lie in the closure under the $\cU$-bootstrap process of the empty sites in $V_{(0,0)}$. In particular, 
$$t^\pm = \min\big\{ t > 0 \,:\, A_t \supseteq V_{(\pm 1,0)} \big\},$$ 
are both finite, where $A_t$ is the set of sites infected after $t$ steps of the $\cU$-bootstrap process, starting from $A_0 = V_{(0,0)}$ (see Definition \ref{def:Uboot}). With foresight, define
\begin{equation}\label{def:super:Lambda}
\L:= \big( A_{t^-} \setminus V_{(0,0)} \big) + n_1 \vec v,
\end{equation}
and note that $\L\cap V_{\vec e_1} = \emptyset$ and $V_{(0,0)} \subset \L$.

\begin{proof}[Proof of part~$(a)$ of Theorem~\ref{mainthm:1}]
The first step is to apply Theorem~\ref{thm:CPI} to the probability space $(S^{\bbZ^2},\mu)$ described in Section~\ref{sec:setting}, in which each `block' variable $\o_i\in S$ is given by the {{collection $\{ \o_x \}_{x \in V_i} \in \{0,1\}^{V_i}$ of i.i.d.~Bernoulli($p$) variables.}} Recall that $p_1 = \hat\mu(G_1)$ and $p_2 = \hat\mu(G_2)$ are the probabilities of the good and super-good events, respectively, and note that, in our setting, $p_1 = 1$ and $p_2 \geq q^{O(n_1n_2)} = q^{O(1)}$. It follows, using~\eqref{eq:8FA}, that
\begin{equation}\label{eq:CPI:supercritical:application}
\var(f) \leq \frac{1}{q^{O(1)}} \sum_{i \in \bbZ^2} \mu\left( \id_{\{\text{either }V_{i+\vec e_1}\text{ or }V_{i-\vec e_1}\text{ is empty}\}}\var_{V_i}(f) \right)
\end{equation}
for all local functions $f$, where $\var_{V_i}(f)$ denotes the variance with respect to the variables $\{ \o_x \}_{x \in V_i}$, given all the other variables $\{ \o_y \}_{y \in \bbZ^2 \setminus V_i}$. 

To deduce~\eqref{eq:supercritical:aim}, it will suffice (by symmetry) to prove an upper bound on the right-hand side of~\eqref{eq:CPI:supercritical:application} of the form
\begin{equation}\label{eq:supercritical:var:Vzero}
\mu\left(\id_{\{V_{\vec e_1}\text{ is empty}\}} \var_{V_{(0,0)}}(f) \right) \le \, \frac{1}{q^{O(1)}}\sum_{x\in \L \cup V_{\vec e_1}}\mu\big( c_x \var_x(f) \big)
\end{equation}
for the set $\L$ defined in~\eqref{def:super:Lambda}, since the elements of $\L \cup V_{\vec e_1}$ are all within distance $O(1)$ from the origin, and so we may then simply sum over all $i \in \bbZ^2$. 

 
To prove~\eqref{eq:supercritical:var:Vzero}, the first step is to observe that, by the convexity of the variance, and recalling that $\L\cap V_{\vec e_1}=\emptyset$ and $V_{(0,0)}\subset \L$, we have
\begin{equation}\label{eq:super:convexity}
\mu\Big( \id_{\{V_{\vec e_1}\text{ is empty}\}} \var_{V_{(0,0)}}(f) \Big) \le \, \mu\left(\id_{\{V_{\vec e_1}\text{ is empty}\}}\var_{\L}(f) \right).
\end{equation}
To conclude we appeal to the following result which, for later
purposes, we formulate in a slightly more general setting than is needed here. In what follows, for any $\o\in \O$ and $U \subset \bbZ^2$, we shall write $[\o]_\cU^{U}$ for the closure of the set $\big\{ x \in \bbZ^2 : \ \o_x=0 \big\}$ {under} the {$\cU$-}bootstrap process restricted to ${U}$.   

\begin{lemma}
 \label{lem:var0}
Let $A,B \subset \bbZ^2$ be disjoint sets, and let $\cE$ be an event
depending only on $\o_{B}$. Suppose that there exists a set $U \supset
A \cup B$ such that $B \subset [\o]_\cU^{U}$ for any $\o\in \{0,1\}^{U}$ for which $A$ is
empty and $\o_{B}\in \cE$. Then
\begin{equation}\label{eq:lem:var0}
\mu\Big(\id_{\{A\text{ is empty}\}}\var_{B}\big( f \tc \cE \big) \Big) \leq 
 |{U}| q^{-|{U}|} \frac{2}{pq}\sum_{x \in {U}}\mu\big( c_x\var_x(f) \big)
\end{equation}
for any local function $f$.
\end{lemma}

Before proving the lemma we conclude the proof of part (a) of Theorem \ref{mainthm:1}. 
We apply the lemma with $A = V_{\vec e_1}$, $B = \L$, $U = A \cup B$ and $\cE$ the trivial event, i.e., $\cE = \O_{B}$. Indeed, by construction (see~\eqref{def:super:Lambda}), $\L\subset \big[ V_{\vec e_1} \big]_{\cU}^{U}$.
Thus \eqref{eq:lem:var0} becomes
\begin{equation}\label{eq:confusing}
\mu\left( \id_{\{V_{\vec e_1}\text{ is empty}\}} \var_\L(f) \right) \le \, |U| q^{-|U|}\frac{{{2}}}{pq} \sum_{x \in U} \mu\big( c_x \var_x(f) \big).
\end{equation}
Since $|U| = O(1)$, and using~\eqref{eq:super:convexity}, we conclude that for all $i \in \bbZ^2$,
$$\mu\left(\id_{\{V_{ i+\vec e_1}\text{ is empty}\}}\var_{V_{i}}(f)\right) \le \, \frac{1}{q^{O(1)}} \sum_{x \in U_i} \mu\big( c_x \var_x(f) \big),$$
where $U_i$ is the analogue of $U$ for the block $V_i$. 

As noted above, summing over $i \in \bbZ^2$ and using~\eqref{eq:CPI:supercritical:application}, we obtain the Poincar\'e inequality~\eqref{eq:supercritical:aim} with constant $q^{-O(1)}$, and by~\eqref{eq:mean-infection} and Definition~\ref{def:PC} it follows that there exists a constant $\lambda = \lambda(\cU)$ such that 
$$\bbE_\mu(\t_0) \leq \frac{\trel(q,\cU)}{q} \leq q^{-\lambda},$$
for all sufficiently small $q > 0$, as required. Since the bootstrap infection time $T_\cU$ of a supercritical update family satisfies $T_\cU = q^{-\Theta(1)}$, it also follows that $\bbE_\mu(\t_0) \le T_\cU^{O(1)}$. 
\end{proof}

\begin{proof}[Proof of Lemma \ref{lem:var0}]
Observe first that, for any $\o \in \O$, 
\begin{gather}\label{eq:super:variance:basic:upper:bound}
\var_B\big( f \tc \cE \big)(\o_{\bbZ^2\setminus B}) \leq \frac{1}{\mu_B(\cE)}\sum_{\eta_B\in \cE} \mu_B( \eta_B ) \Big( f\big(\eta_B,\o_{\bbZ^2\setminus B} \big) - f\big(0, \o_{\bbZ^2\setminus B}\big) \Big)^2,
\end{gather}
since $\Ex\big[ (X - a)^2 \big]$ is minimized by taking $a = \Ex[X]$, where $\big( 0, \o_{\bbZ^2 \setminus B} \big)$ denotes the configuration that is equal to $\o_{\bbZ^2 \setminus B}$ outside $B$, and empty inside $B$. 

We will break each term on the right-hand side of~\eqref{eq:super:variance:basic:upper:bound} into the sum of single spin-flips using the $\cU$-bootstrap process as follows. Fix $\o \in \O$ such that $A$ is empty, and fix $\eta_B \in \cE$. Using the assumption of the lemma, we claim that there exists a path $\g \equiv (\o^{(0)},\dots,\o^{(m)})$ in $\O$ such that:  
\begin{enumerate}[(i)]
\item $\o^{(0)} = (\eta_B,\o_{\bbZ^2\setminus B})$ and $\o^{(m)} = (0,\o_{\bbZ^2\setminus B})$;
  \item the length $m$ of $\g$ satisfies $m \leq 2|U|$;
\item for each $k=1,\dots,m$, there exists a vertex $x^{(k)} \in U$ such that 
\begin{itemize}
\item the configuration $\o^{(k)}$ is obtained from $\o^{(k-1)}$ by flipping the value at $x^{(k)}$;
\item this flip is legal, i.e., $c_{x^{(k)}}\big( \o^{(k-1)} \big) = 1$.
\end{itemize}
\end{enumerate}
We construct $\g$ in two steps: first we empty all of $B$, and possibly some of $U \setminus B$; then we reconstruct $\o_{\bbZ^2\setminus B}$ without changing $\o_B$. To spell out the details, observe first that, since $B \subset \big[ \big( \eta_B,\o_{U \setminus B} \big) \big]_\cU^{U}$, there exists a sequence of legal flips in $U$ connecting $\big( \eta_B,\o_{\bbZ^2 \setminus B} \big)$ to a configuration with $A \cup B$ empty. By choosing a minimal such sequence, we may assume that all of the flips are from occupied to empty, and therefore that this first part of the path has length at most $|U|$. 

Now, to reconstruct $\o_{\bbZ^2\setminus B}$, we simply run the same sequence backwards, except without performing the steps inside $B$. Note that all of these flips are legal, since skipping the steps inside $B$ only creates additional empty sites, and that this second part of the path also has length at most $|U|$, as required.

It follows, using Cauchy--Schwarz, that 
\begin{multline*}
\Big( f\big(\eta_B,\o_{\bbZ^2\setminus B} \big) - f\big(0, \o_{\bbZ^2\setminus B}\big) \Big)^2 \leq m \sum_{k = 1}^{m} c_{x^{(k)}}\big( \o^{(k-1)} \big) \Big( f\big(\o^{(k)}\big) - f\big( \o^{(k-1)} \big) \Big)^2 \\
\leq 2 |U| \frac{1}{\mu_*}\frac{1}{pq}\sum_{x\in {U}} \sum_{\eta\in \{0,1\}^{U}}\mu_{U}(\eta)c_x(\eta, \o_{\bbZ^2\setminus {U}})\ pq \Big( f\big(\eta^{(x)}, \o_{\bbZ^2\setminus {U}}\big) - f\big( \eta, \o_{\bbZ^2\setminus U} \big) \Big)^2,
\end{multline*}
for any $\o$ in which $A$ is empty, and any $\eta_B \in \cE$, where $\mu_*=\min_{\eta\in \{0,1\}^{U}} \mu_{U}(\eta) = q^{|U|}$, and $\eta^{(x)}$ denotes the configuration obtained from $\eta$ by flipping the spin at $x$. Notice that the right-hand side does not depend on $\eta_B$, and that $pq \big( f\big(\eta^{(x)}, \o_{\bbZ^2\setminus {U}}\big) - f\big( \eta, \o_{\bbZ^2\setminus {U}} \big) \big)^2$ is the local variance $\var_x(f)$ computed for the configuration $\o\equiv (\eta,  \o_{\bbZ^2\setminus {U}})$. 

Hence, using~\eqref{eq:super:variance:basic:upper:bound}, we obtain
$$\id_{\{A \text{ is empty}\}} \var_B\big( f \tc \cE \big)(\o_{\bbZ^2\setminus B}) \leq \, \frac{2 |U|q^{-|{U}|}}{pq} \sum_{x\in U}\mu_{U}(c_x\var_x(f)\bigr)( \o_{\bbZ^2\setminus U})$$
for any $\o \in \O$, and inequality~\eqref{eq:lem:var0} follows by averaging using the measure $\mu$.
\end{proof}

\subsection{The rooted case}\label{sec:supercritic:rooted}

Let $\cU$ be a supercritical, rooted, two-dimensional update family, let $C\subset S^1$ be a semicircle with no stable directions 
and recall that, thanks to ~\eqref{eq:mean-infection}, it will suffice
to prove a Poincar\'e inequality (cf.~\eqref{eq:supercritical:aim})
with constant $q^{{-}O(\log(1/q))}$. To prove this we will follow almost exactly the same route of the unrooted case, with the same definition of the blocks $V_i$ and of the good and super-good events. We will therefore
only give a very brief sketch of the proof in this new setting. 

The {{main difference}} w.r.t.~the unrooted case is that now the opposite semicircle $-C$ will necessarily contain some stable directions. This forces us to use the oriented
Poincar\'e inequality~\eqref{eq:8East} from Theorem~\ref{thm:CPI}
instead of the unoriented one~\eqref{eq:8FA}, because in this case (see
Proposition~\ref{prop:bootsc} and Remark~\ref{rem:left-right}) a super-good block is able to infect the block to its left but not the block to its right, \ie $V_{(-1,0)}\subset [V_{(0,0)}]_\cU$ but $V_{(1,0)} \not\subset [ V_{(0,0)} ]_\cU$. 

\begin{proof}[Proof of part~$(b)$ of Theorem~\ref{mainthm:1}]
We again apply Theorem~\ref{thm:CPI} to the probability space $(S^{\bbZ^2},\mu)$ described in Section~\ref{sec:setting}, but we use~\eqref{eq:8East} instead of~\eqref{eq:8FA}. Recalling that $p_1 = 1$ and $p_2 = q^{O(1)}$, we obtain
\begin{equation}\label{eq:CPI:supercritical:application:rooted}
\var(f) \leq \frac{1}{q^{O(\log(1/q))}} \sum_{i \in \bbZ^2} \mu\left( \id_{\{V_{i+\vec e_1}\text{ is empty}\}}\var_{V_i}(f) \right)
\end{equation}
for all local functions $f$. As before, using translation invariance, we only examine the $i = (0,0)$ term in the above sum. We claim that 
\begin{equation}\label{eq:supercritical:var:rooted}
\mu\left(\id_{\{V_{\vec e_1}\text{ is empty}\}} \var_{V_{(0,0)}}(f) \right) \le \, \frac{1}{q^{O(1)}} \sum_{x \in {{U}}}\mu\big( c_x \var_x(f) \big)
\end{equation}
{for ${{U}} = V_{\vec e_1} \cup \L$,} where $\L$ is the set defined in~\eqref{def:super:Lambda}. However, the proof of~\eqref{eq:supercritical:var:rooted} is identical to that of~\eqref{eq:supercritical:var:Vzero}, since Proposition~\ref{prop:bootsc} implies that $V_{(0,0)}$ can be entirely infected by $V_{\vec e_1}$. We therefore obtain the Poincar\'e inequality
\begin{equation}\label{eq:supercritical:aim:rooted}
\var(f) \le  \frac{1}{q^{O(\log(1/q))}}\sum_x \mu\big( c_x \var_x(f) \big)
\end{equation}
for all local functions $f$. Thus $\trel(q,\cU) \leq q^{-O(\log(1/q))}$, and hence
$$\bbE_\mu(\t_0) \leq \frac{\trel(q,\cU)}{q} \,\le \, q^{-O(\log(1/q))} \, = \, T_\cU^{O(\log T_\cU)},$$
as required, because $T_\cU = q^{-\Theta(1)}$. 
\end{proof}

\section{Critical KCM: proof of Theorem \ref{mainthm:2} under a simplifying assumption}\label{sec:critic-ub}

In this section we shall prove Theorem~\ref{mainthm:2} under the following additional assumption (see below): every stable direction $v$ with finite difficulty has a voracious set that is a subset of the line $\ell_v$. By doing so, we avoid some technical complications (mostly related to the geometry of the quasi-stable half-ring) which might obscure the main ideas behind the proof. The changes necessary to treat the general case are spelled out in detail in Section~\ref{sec:fullgen}. 

\begin{assumption}\label{easy}
For any stable direction $u \in \cS$ with finite difficulty $\a(u)$, there exists a set $Z_u \subset \ell_u$ of cardinality $\a(u)$ such that $\big[ \H_u \cup Z_u \big]_{\cU} \cap \ell_u$ is infinite.  
\end{assumption}

As in Section~\ref{sec:supercritical}, our main task will be to establish a suitable upper bound on the relaxation time $\trel(\cU;q)$. In Section~\ref{sec:critical:rooted} we will first analyse the $\a$-rooted case and the starting point will be the constrained Poincar\'e inequality~\eqref{eq:9East}; the proof the $\b$-unrooted case (see Section \ref{sec:beta-unrooted})
will be essentially the same, the main difference being that~\eqref{eq:9East} will be replaced by~\eqref{eq:9FA}. 


\subsection{$\alpha$-rooted update families}\label{sec:critical:rooted}

Let $\cU$ be a critical, $\alpha$-rooted, two-dimensional update
family, and recall from Definition~\ref{def:alpha:rooted} that $\cU$
has difficulty $\a$, and bilateral difficulty at least $2\a$. The
properties of $\cU$ that we will need below have already been proved
in Section~\ref{sec:critic:spread}; they all follow from the fact
(see Lemma~\ref{cor:Ccritic}) that there exists an open semicircle $C$
such that $C \cap \cS(\cU)$ consists of finitely many directions, each
with difficulty at most $\a$. In particular, we will make crucial use
of Corollary~\ref{cor:critic:spread}.  

We will prove that, if Assumption~\ref{easy} holds, then there exists a constant $\lambda = \lambda(\cU)$ such that 
$$\bbE_\mu(\t_0) \leq \frac{\trel(q,\cU)}{q} \leq \exp\Big( \lambda \cdot q^{-2\a} \big( \log(1/q) \big)^4 \Big)$$
for all sufficiently small $q > 0$. Note that the first inequality follows from~\eqref{eq:mean-infection}, and so, by Definition~\ref{def:PC}, it will suffice to prove that
\begin{equation}\label{eq:rooted:aim}
\var(f) \le \exp\Big( \lambda \cdot q^{-2\a} \big( \log(1/q) \big)^4 \Big) \sum_{x \in \bbZ^2} \mu\big( c_x \var_x(f) \big)
\end{equation}
for some $\lambda = \lambda(\cU)$ and all local functions $f$. We will
deduce a bound of the form~\eqref{eq:rooted:aim} starting
from~\eqref{eq:9East} and using Corollary~\ref{cor:critic:spread}. 

\begin{remark}
We are not able to use the unoriented constrained Poincar\'e inequality~\eqref{eq:9FA} in place of the oriented inequality~\eqref{eq:9East} in the proof of~\eqref{eq:rooted:aim} because there exist $\a$-rooted models (the Duarte model~\cite{Duarte} is one such example) with $\b = \infty$ such that, for any choice of the side-lengths $n_1$ and $n_2$ of the blocks $V_i$, and of the good and super-good events $G_2 \subset G_1$ satisfying the condition $(1-p_1)(\log p_2)^2 = o(1)$, the $\cU$-bootstrap process is not guaranteed to be able to infect the block $V_i$ using only that facts that the block $V_{i-\vec e_1}$ is infected and that some nearby blocks $V_j$ are good. For update families with $2\a < \b < \infty$, it \emph{is} possible to apply~\eqref{eq:9FA} for certain choices of $(n_1,n_2,G_1,G_2)$, but the best Poincar\'e constant we are able to obtain in this way is roughly $\exp\big( q^{-\b} \big)$, which is much larger than the one we prove using~\eqref{eq:9East}. 
\end{remark}

\subsubsection{The geometric setting and the good and super-good events}


Recall the construction and notation described in Sections~\ref{sec:setting} and~\ref{sec:critic:spread}; in particular, recall that $V = R \cap \bbZ^2$, where $R$ is a rectangle in the rotated coordinates $(v,v^\perp)$, and $u = -v$ is the midpoint of an open semicircle $C \subset S^1$ in which every stable direction has difficulty at most~$\a$. As in Section~\ref{sec:strategy}, when drawing figures we will think of $u$ as pointing to the left. We will choose the parameters $n_1$ and $n_2$ (which determine the side-lengths of $R$) depending on $q$; to be precise, set
$$n_1 = \big\lfloor q^{-{{3}}\kappa} \big\rfloor \qquad \text{and} \qquad n_2 = \big\lfloor \kappa^4 q^{-\a} \log (1/q) \big\rfloor,$$
where $\kappa = \kappa(\cU)$ is a sufficiently large constant.


In order to define the good and super-good events $G_1$ and $G_2$, we need to define some structures which we call \emph{$\kappa$-stairs}, which will provide us with a way of transporting infection `vertically'. Let us call the set of points of $V$ lying on the same line parallel to $u$ (resp. $u^\perp$) a \emph{row} (resp. \emph{column}) of $V$, and let us order the rows from bottom to top and the columns from left to right. 
Let $a$ and $b$ be (respectively) the number of rows and columns of $V$, and observe that, since $v$ is a rational direction, we have $a = \Theta(n_2)$ and $b = \Theta(n_1)$, where the implicit constants depend only on the update family $\cU$. We will say that a set of vertices is an \emph{interval of $V$} if it is the intersection with {{$V$}} of a line segment in~$\bbR^2$. Recall that $\kappa = \kappa(\cU) > 0$ was fixed above. 

\begin{definition} 
\label{def:stair} 
We say that a collection $\cM = \big\{ M^{(1)}, \ldots, M^{(a)} \big\}$ of disjoint intervals of $V$ of size $2\kappa$ forms an \emph{upward $\kappa$-stair} with
steps $M^{(1)},M^{(2)},\ldots$ if: 
 \begin{enumerate}[(i)]
 \item for each $i \in [a],$ the $i^{th}$-step $M^{(i)}$ belongs to the $i^{th}$-row of $V$;
\item the $i^{th}$-step is ``to the left'' of the $j^{th}$-step if
  $i<j$. More precisely, let $\big( M_\ell^{(i)}, M_r^{(i)} \big)$ be the abscissa (in the $(v,v^\perp)$-frame) of the leftmost and rightmost elements (respectively) of the $i^{th}$-step. Then $M_r^{(i)} < M_\ell^{(j)}$  whenever $i < j$.
\end{enumerate}
\end{definition}

We refer the reader to Figure \ref{fig:strategy} for a picture of an upward $\kappa$-stair.

We are now ready to define the good and super-good events. Let us say that a quasi-stable half-ring $\cR$ \emph{fits in} the block $V_i$ if the top and bottom sides of $\cR$ are contained in the top and bottom sides of {{$R_i$}}, and note that this determines the length $\ell$ of $\cR$, which moreover satisfies $\ell \geq n_2 / m$ (see Definition~\ref{def:half-ring}). Let $(v_1,\dots, v_m)$ be the quasi-stable directions in $C$ (see Definition~\ref{def:quasi-stable}), and recall the definitions of a $v$-strip $S_{v}$ (see Definition~\ref{def:strips}) and of a helping set $Z$ for $S_{v}$ (see immediately after Lemma~\ref{lem:strip}). Note that Assumption~\ref{easy} implies that, for any $j \in [m]$ and $v_j$-strip $S_{v_j}$, we may choose the voracious set $Z_{v_j}$ so that the helping sets for $S_{v_j}$ are subsets of $\partial^{\rm ext}(S_{v_j})$.

\begin{definition}[Good and super-good events]
\label{def:goodsets} \
\begin{enumerate}
\item The block $V_i = R_i \cap \bbZ^2$ satisfies the \emph{good} event $G_1$ iff:  
  \begin{enumerate}[(a)]
  \item for each quasi-stable direction $v \in C$ and every $v$-strip $S$
    such that the length of the segment $\partial^{\rm ext}(S) \cap R_i$ is at least $n_2/(4m)$, there
    exists an empty helping set $Z \subset {{\partial^{\rm ext}(S) \cap V_i}}$ for~$S$; 
\item there exists an empty upward {{$\kappa$}}-stair within the leftmost quarter of $V_i$. 
  \end{enumerate}
\item The block $V_i$ satisfies the \emph{super-good} event $G_2$ iff it satisfies the good event $G_1$, and moreover there exists an empty quasi-stable half-ring $\cR$ of width $\kappa$, 
that fits in $V_i$ and is entirely contained in the rightmost quarter of $R_i$.
\end{enumerate}
\end{definition}

Next we prove that the hypothesis for the part (B) of Theorem \ref{thm:CPI} holds in the above setting if $\kappa$ is sufficiently large.

\begin{lemma}
\label{lem:p1p2}
Let $p_1:=\hat\mu(G_1)$ and $p_2:=\hat\mu(G_2)$. There exists a constant $\kappa_0(\cU) > 0$ such that, for any $\kappa > \kappa_0(\cU)$,
\[
\lim_{q \to 0} \big( 1 - p_1 \big) \big( \log(1/p_2) \big)^2 = 0
\]
\end{lemma}

\begin{proof}
First, let's bound the probability that there is no empty helping set $Z \subset V_i$ for a given $v$-strip $S$ (where $v$ is a quasi-stable direction in $C$) such that $\partial^{\rm ext}(S) \cap R_i$ has length at least $n_2/(4m)$. Observe {{that we}} can choose $\Omega(n_2)$ potential values for each $k_j$ in~\eqref{def:helping:set} such that the corresponding sets $Z_v + a_j + k_j b$ are pairwise disjoint subsets of $\partial^{\rm ext} (S) \cap V_i$ (using Assumption~\ref{easy}), and that each such translate of $Z_v$ is empty with probability $q^\a$. (Here the implicit constant depends only on $\cU$.) The probability that $S$ has no empty helping set is therefore at most
$$O\Big( \big( 1 - q^\a \big)^{\Omega(n_2)} \Big) \leq q^{\kappa^3}$$
if $\kappa$ is sufficiently large and $q \ll 1$. There are at most $O(n_1^2 n_2^2)$ choices for the quasi-stable direction $v \in C$ and for the intersection of the $v$-strip $S$ with $V_i$. {{Thus, by the union bound, part~(a) of the definition of $G_1$ holds with probability at least $1 - q^{\kappa^2}$. 

To bound the probability of part~(b), observe that an interval of $V$ of size $\Theta(n_1/n_2)$ contains an empty interval of $V$ of size $2\kappa$ with probability at least
$$1 - \big(1 - q^{2\kappa} \big)^{\Theta(n_1/n_2)} \geq 1 - \exp\big( - q^{-\kappa + 2\alpha} \big),$$
if $q \ll 1$, and therefore the probability that $V$ contains an empty upward $\kappa$-stair is (by the union bound) at least
\begin{equation}
  \label{eq:or1}
1 - O(n_2) \exp\big( - q^{-\kappa + 2\alpha} \big) \geq 1 - q^{\kappa^2}
\end{equation}
 if $\kappa$ is sufficiently large and $q\ll 1$.}} 
It follows that 
$$1 - p_1 \, = \, 1 - \hat\mu(G_1) \, \leq \, 2 \cdot q^{\kappa^2}.$$
Moreover, the probability that there exists an empty quasi-stable half-ring $\cR$ of width $\kappa$ that fits in $V_i$ is at least $q^{O(n_2)}$, so (by the FKG inequality) we have 
$$\log(1 / p_2) \, \leq \, O(n_2) \log(1/q) \leq O\Big( q^{-\a} \big( \log (1/q) \big)^2 \Big),$$
where the implicit constant is independent of $q$. It follows that, if $\kappa$ is sufficiently large, then $\big( 1 - p_1 \big) \big( \log(1/p_2) \big)^2 \to 0$ as $q \to 0$, as required.
\end{proof}

Let us fix, from now on, the constant $\kappa$ to be sufficiently large so that Lemma~\ref{lem:p1p2} applies. In particular, by Theorem~\ref{thm:CPI}, the constrained Poincar\'e inequality~\eqref{eq:9East} holds for any local function $f$, i.e.,
\begin{align}
  \label{eq:9East:bis}
& \var(f) \leq \vec T(p_2) \bigg( \sum_{i\in \bbZ^2}\mu\left(\id_{\{\o_{i+\vec e_2} \in G_2\}}\id_{\{\o_{j}\in G_1\, \forall j\in \bbL^+(i)\}} \var_{V_i}(f)\right)\nonumber\\
& \hspace{3.5cm} + \sum_{i\in \bbZ^2}\mu\left(\id_{\{\o_{i+\vec e_1}\in G_2\}}\id_{\{\o_{i-\vec e_1}\in G_1\}} \var_{V_i}\big( f \tc G_1 \big) \right)\bigg),  
\end{align}
with 
\[
\vec T(p_2) \, = \, e^{O(\log(p_2)^2)} \, = \, \exp\Big( O\big( q^{-2\a} \log(1/q)^4 \big) \Big). 
\]
As in the supercritical setting (see Section~\ref{sec:supercritical}), our strategy will be to bound each of the sums in the r.h.s.~of~\eqref{eq:9East:bis} from above in terms of the Dirichlet form $\cD(f)$ of our  KCM. To do so, it will suffice to bound from above, for a fixed (and arbitrary) local function~$f$, the following two generic terms:
\[
I_1(i):= \mu\left(\id_{\{\o_{i+\vec e_1}\in
    G_2\}}\id_{\{\o_{i-\vec e_1}\in G_1\}}
\var_{V_i}\big( f \tc G_1 \big) \right),
\] 
and
\[
I_2(i):=\mu\left(\id_{\{\o_{i+\vec e_2}\in
    G_2\}}\id_{\{\o_j\in G_1\, \forall j\in \bbL^+(i)\}} \var_{V_i}(f) \right),
\]
see Figure \ref{fig:setting}. Using translation invariance it suffices to consider only the case $i=(0,0)$, so let us set $I_1 \equiv I_1((0,0)$ and $I_2 \equiv I_2((0,0))$.

\begin{figure}[ht]
\centering
\begin{tikzpicture}[scale=0.3]
\node at (-15,1) {$I_1:$};
\node at (-15,-3) {$I_2:$};
\begin{scope}
\draw [thick, fill=lightgray] (0,0) rectangle (12,2);
\draw [thick,  xshift=12.3cm] (0,0) rectangle (12,2);
\draw [thick,  xshift=-12.3cm] (0,0) rectangle (12,2);
\node [scale=0.8] at (6,1) {$V:\text{G}$};
\node [scale=0.8] at (18.5,1) {$V_{(1,0)}:\text{ SG}$};
\node [scale=0.8] at (-6,1) {$V_{(-1,0)}:\text{ G}$};
\end{scope}
\begin{scope}[shift={(0,-6)}]
\draw [thick,fill=lightgray] (0,0) rectangle (12,2);
\draw [thick,  xshift=12.3cm] (0,0) rectangle (12,2);
\draw [thick,xshift=-6.2cm,yshift=2.3cm] (0,0) rectangle
(12,2);
\node [scale=0.8] at (12.5-12.3,3.3) {$V_{(-1,1)}:\text{ G}$};
\draw [thick,xshift=6cm,yshift=2.3cm] (0,0) rectangle
(12,2);
\node [scale=0.8] at (6,1) {$V$};
\node [scale=0.8] at (18.5,1) {$V_{(1,0)}:\text{ G}$};
\node [scale=0.8] at (12.5,3.3) {$V_{(0,1)}:\text{ SG}$};
\end{scope}
\end{tikzpicture}
\caption{In $I_1$ the block $V\equiv V_{(0,0)}$ is conditioned to be good (G), while the blocks $V_{(-1,0)}$ and $V_{(1,0)}$ are good and super-good (SG) respectively. Recall that $\bbL^+((0,0)) = \big\{ (1,0), (-1,1) \big\}$, so in $I_2$ the 
  blocks $V_{(1,0)}$ and $V_{(-1,1)}$ are good, the block $V_{(0,1)}$
  is super-good, and $V$ is unconditioned.}
\label{fig:setting}
\end{figure}

Define $W_1= V_{(0,0)} \cup V_{(-1,0)} \cup V_{(1,0)}$ and $W_2 = V_{(0,0)} \cup V_{(-1,0)} \cup V_{(1,0)} \cup V_{(-1,1)} \cup V_{(0,1)}$. We will prove the following upper bounds on $I_1$ and $I_2$.

\begin{proposition}
\label{prop:I1:I2}
For each $j \in \{1,2\}$, {{there exists a $O(1)$-neighbourhood $\hat W_j$ of $W_j$ such that}}
\begin{align*}
  I_j \, \leq \exp\Big( O\big( q^{-\a} \log(1/q)^{{3}} \big) \Big) \sum_{x \in \hat W_j} \mu\big(c_x
  \var_x(f)\big).
\end{align*}
\end{proposition}

Observe that, combining Proposition~\ref{prop:I1:I2} with~\eqref{eq:9East:bis}, and noting that $|\hat W_j| = q^{-O(1)}$, we immediately obtain a final Poincar\'e inequality of the form~\eqref{eq:rooted:aim}, i.e.,
\[
\var(f) \leq \exp\Big( O\big( q^{-2\a} \log(1/q)^4 \big) \Big) \sum_{x \in \bbZ^2} \mu\big( c_x \var_x(f) \big),
\]
as required. It will therefore suffice to prove Proposition~\ref{prop:I1:I2}.


\subsubsection{{The core of the proof of Proposition~\ref{prop:I1:I2}}}\label{sec:core}

Before giving the full technical details of the proof of the proposition, we first explain the high-level idea we wish to exploit. Fix~$j \in \{1,2\}$, set $W := W_j$, and fix $\o\in \O$ such that the restriction of $\o$ to $W$ satisfies the requirement of the good and super-good environment of the blocks (see~Figure~\ref{fig:setting}). The key idea is to cover the block $V = V_{(0,0)}$ with a collection of pairwise disjoint ``fibers'' $\hat F_1,\ldots, \hat F_{N+1}$, each of which is a quasi-stable half-ring, for some $N \leq |V|$ depending on~$\o$. For each fiber $\hat F_i$, the set $F_i := \hat F_i \cap \bbZ^2$ is a subset of $W$ of cardinality $O(n_2)$ with the following key properties (which we will define precisely later): 
\begin{enumerate}[(a)]
\item the fiber $F_{N+1}$ is empty;
\item in each fiber $F_i$ a certain ``helping'' event $H_i$, depending only on the restriction of $\o$ to $F_i$, and implied by our assumption on the goodness\footnote{It is worth emphasizing here that $H_i$ only requires the blocks to be good, rather than super-good, and therefore holds with high probability.} of the blocks in $W$, occurs; 
\item {{the helping event $H_i$ has the following property: the $\cU$-bootstrap process restricted to a $O(1)$-neighbourhood of the set  $F_i \cup F_{i+1}$ is able to infect $F_i$ for any $\o$ such that $F_{i+1}$ is empty and $H_{i}$ occurs.}}  
\end{enumerate}
To be concrete, let us consider the term $I_1$. In this case we will take $F_{N+1}$ to be $\cR \cap \bbZ^2$, where $\cR$ is the rightmost empty quasi-stable half-ring of width $\kappa$ that fits in $V_{(1,0)}$, which exists by our assumption that $V_{(1,0)}$ is a super-good block. The other fibers $F_1,\ldots, F_{N}$ will be suitable disjoint translates of $F_{N+1}$ in the $u$-direction, satisfying $V \subset \L := \bigcup_{i=1}^N F_i$. The helping event $H_i$ will require the presence in $F_i$ of suitable helping sets for each quasi-stable direction{; we remark that the key requirement that $H_i$ depends only on the restriction of $\o$ to $F_i$ is a consequence of Assumption~\ref{easy}. Finally, the third condition (c) above will follow from Corollary \ref{cor:critic:spread}. A similar construction will be used for the term $I_2$, but the fibers will be slightly more complicated, see Figure~\ref{fig:strategy}.

 
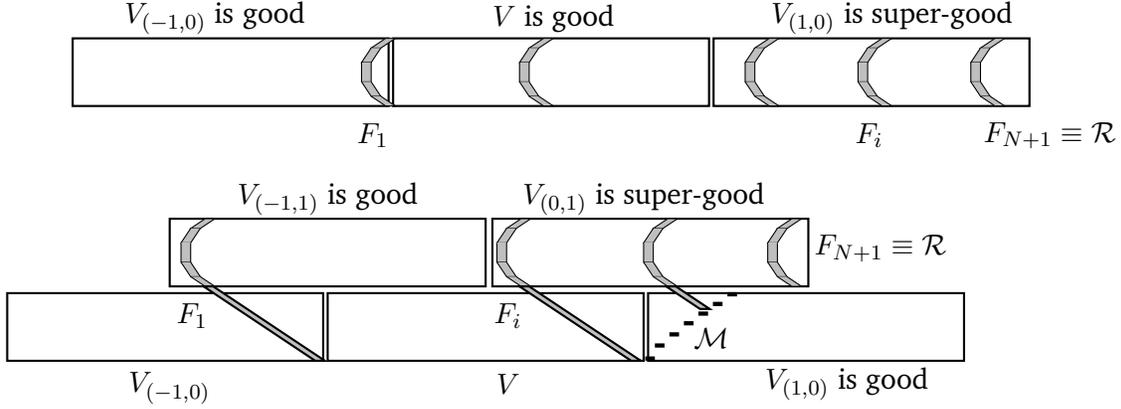
\begin{figure}[ht]
\centering
\begin{tikzpicture}[scale=0.3]
\begin{scope}[shift={(0,0)}]
\begin{scope}
\draw [thick,xshift=-7.1cm,yshift=3.3cm] (0,0) rectangle
(14,3);
\draw [thick,xshift=7.1cm,yshift=3.3cm] (0,0) rectangle
(14,3);
\draw [thick,xshift=-21.3cm,yshift=3.3cm] (0,0) rectangle
(14,3);
\end{scope}
\begin{scope}
\foreach \j in {0,0.5,1,2,2.7}{ 
\begin{scope}[shift={(20-\j*10,6.3)},rotate around={180:(0,0)}, scale=3/14]
\draw  [fill=lightgray] (0,0)--(3,2)--(5,2)--(2,0)--(0,0);  
\draw  [fill=lightgray] (3,2)--(5,5)--(7,5)--(5,2);
\draw [fill=lightgray] (5,5)--(5,9)--(7,9)--(7,5);
\draw [fill=lightgray] (5,9)--(3,12)--(5,12)--(7,9);
\draw [fill=lightgray] (3,12)--(0,14)--(2,14)--(5,12);
\end{scope}
}
\end{scope}

\begin{scope}
    \node at (0,7.2) {$V$ is good};
  \node at (15,7.2) {$V_{(1,0)}$ is super-good};
\node at (-15,7.2) {$V_{(-1,0)}$ is good};
  \node at (22,2) {$F_{N+1}\equiv \cR$};
\node at (14,2) {$F_i$};
\node at (-8,2) {$F_1$};
\end{scope}
\end{scope}


\begin{scope}[shift={(-10,-8)}]
  \node at (24.5,5) {$F_{N+1}\equiv \cR$};
\node at (8,2) {$F_i$};
\node at (-6,2) {$F_1$};
\node at (17,1) {$\mathcal M$};
\end{scope}

\begin{scope}[shift={(-10,-8)}]
  \begin{scope}
\draw [thick] (0,0) rectangle (14,3);
\draw [thick,  xshift=14.2cm] (0,0) rectangle (14,3);
\draw [thick, xshift=-14.2cm] (0,0) rectangle (14,3);
\draw [thick,xshift=-7cm,yshift=3.3cm] (0,0) rectangle (14,3);
\draw [thick,xshift=7.3cm,yshift=3.3cm] (0,0) rectangle (14,3);
\end{scope}

\begin{scope}[shift={(21,6.3)},rotate around={180:(0,0)}, scale=3/14]
\draw  [fill=lightgray] (0,0)--(3,2)--(5,2)--(2,0)--(0,0);  
\draw  [fill=lightgray] (3,2)--(5,5)--(7,5)--(5,2);
\draw [fill=lightgray] (5,5)--(5,9)--(7,9)--(7,5);
\draw [fill=lightgray] (5,9)--(3,12)--(5,12)--(7,9);
\draw [fill=lightgray] (3,12)--(0,14)--(2,14)--(5,12);
\end{scope}
\begin{scope}[shift={(15.5,6.3)},rotate around={180:(0,0)}, scale=3/14]
\draw  [fill=lightgray] (0,0)--(3,2)--(5,2)--(2,0)--(0,0);  
\draw  [fill=lightgray] (3,2)--(5,5)--(7,5)--(5,2);
\draw [fill=lightgray] (5,5)--(5,9)--(7,9)--(7,5);
\draw [fill=lightgray] (5,9)--(3,12)--(5,12)--(7,9);
\draw [fill=lightgray] (3,12)--(0,14)--(2,14)--(5,12);
\end{scope}
\begin{scope}[shift={(15,0.3)},scale=3/14]
\draw [thick,fill=lightgray] (0,14)--(2,14)--(2+3*1*14/6,14-1*14/3)--((3*1*14/6,14-1*14/3)--cycle;
\end{scope}

\begin{scope}[shift={(9,6.3)},rotate around={180:(0,0)}, scale=3/14]
\draw  [fill=lightgray] (0,0)--(3,2)--(5,2)--(2,0)--(0,0);  
\draw  [fill=lightgray] (3,2)--(5,5)--(7,5)--(5,2);
\draw [fill=lightgray] (5,5)--(5,9)--(7,9)--(7,5);
\draw [fill=lightgray] (5,9)--(3,12)--(5,12)--(7,9);
\draw [fill=lightgray] (3,12)--(0,14)--(2,14)--(5,12);
\end{scope}
\begin{scope}[shift={(8.5,0.3)},scale=3/14]
\draw [thick,fill=lightgray] (0,14)--(2,14)--(2+3*3.3*14/6,14-3.3*14/3)--((3*3.3*14/6,14-3.3*14/3)--cycle;
\end{scope}

\begin{scope}[shift={(-11,0)}]
  \begin{scope}[shift={(6,6.3)},rotate around={180:(0,0)}, scale=3/14]
\draw  [fill=lightgray] (0,0)--(3,2)--(5,2)--(2,0)--(0,0);  
\draw  [fill=lightgray] (3,2)--(5,5)--(7,5)--(5,2);
\draw [fill=lightgray] (5,5)--(5,9)--(7,9)--(7,5);
\draw [fill=lightgray] (5,9)--(3,12)--(5,12)--(7,9);
\draw [fill=lightgray] (3,12)--(0,14)--(2,14)--(5,12);
\end{scope}

\begin{scope}[shift={(5.5,0.3)},scale=3/14]
\draw [thick,fill=lightgray] (0,14)--(2,14)--(2+3*3.3*14/6,14-3.3*14/3)--((3*3.3*14/6,14-3.3*14/3)--cycle;
\end{scope}
\end{scope}

\begin{scope}[shift={(13,1.6)},scale=3/14]
\draw [ultra thick] (22,6.3)--(24,6.3);
\draw [ultra thick] (19,4.3)--(21,4.3);
\draw [ultra thick] (16,2.3)--(18,2.3);
\draw [ultra thick] (13,0.3)--(15,0.3);
\draw [ultra thick] (10,-2)--(12,-2);
\draw [ultra thick] (7,-4.3)--(9,-4.3);
\draw [ultra thick] (5,-7)--(7,-7);
\end{scope}

\begin{scope}
  \node at (8,-1) {$V$};
  \node at (-7,-1.2) {$V_{(-1,0)}$ };
  \node at (23,-1) {$V_{(1,0)}$ is good};
  \node at (0,7.2) {$V_{(-1,1)}$ is good};
  \node at (14,7.2) {$V_{(0,1)}$ is super-good};
\end{scope}

\end{scope}
\end{tikzpicture}
\caption{The top picture shows the local neighbourhood $W_1$ of
  the block $V = V_{(0,0)}$; in this case the fibers are simply the disjoint translates of the
rightmost empty quasi-stable half-ring $\cR$ in the last quarter of $V_{(1,0)}$. The bottom picture shows the local neighbourhood $W_2$; in this case the fibers are not all equal: they grow as they `descend' the steps of the upward $\kappa$-stair $\cM$ (the little horizontal intervals). Each fiber becomes an elongated version of the rightmost empty half-ring $\cR$.}
\label{fig:strategy}
\end{figure}

Let us write $\nu_i$ for the Bernoulli($p$) product measure on $S_i = \{0,1\}^{F_i}$ conditioned on the event $H_i$. The main step in the proof is the following bound on $I_j$ for $j \in \{1,2\}$:
\begin{equation}
  \label{eq:core1}
I_j \, \leq \, \frac{1}{p_1} \cdot \mu\Big(\id_{\{F_{N+1} \text{ is empty}\}} \var_\nu(f) \Big),
\end{equation}
where $\var_{\nu}(\cdot)$ is the variance computed w.r.t.~the product measure $\nu=\otimes_{i=1}^N \nu_i$. Before proving~\eqref{eq:core1} (see Section~\ref{sec:East-like}, below), let us show how to use Proposition~\ref{lem:gen-Poincare} and Lemma \ref{lem:var0} to deduce Proposition~\ref{prop:I1:I2} from it.

\begin{proof}[Proof of Proposition~\ref{prop:I1:I2}, assuming~\eqref{eq:core1}]
Consider the generalized East chain on the space $\otimes_{i=1}^N (S_i,\nu_i)$ with constraining event $S_i^g = \big\{ F_i \text{ is empty} \big\}$ (see Definition~\ref{def:gen:East}). Note that the East constraint for the last fiber $F_N$ is always satisfied because $F_{N+1}$ is empty, and that the parameters $\{q_i\}_{i=1}^N$ of the generalized East process satisfy 
\[
q_i = \nu_i\big( S_i^g \big) \geq \, q^{O(n_2)} = \exp\Big( - O\big( q^{-\a} \log(1/q)^2 \big) \Big).
\]  
Noting that $N \leq |V| = q^{-O(1)}$, it follows from~\eqref{eq:scaling} that 
\begin{equation}\label{eq:losing:log}
T_{\text{\tiny East}}\big( n,\bar \a \big) \leq \exp\Big( O\big( q^{-\a} \log(1/q)^3 \big) \Big).
\end{equation}
Hence, applying Proposition \ref{lem:gen-Poincare} to bound $\var_\nu(f)$ from above, and recalling~\eqref{eq:core1} and that $\L = \bigcup_{i=1}^N F_i$, we obtain
\begin{align}
{I_j} & \, \leq \, \frac{1}{p_1} \cdot \mu\Big(\id_{\{F_{N+1} \text{ is empty}\}}\var_\nu(f) \Big)\nonumber\\
& \, \leq \, {{e^{O( q^{-\a} \log(1/q)^3)}}} \mu\bigg(\id_{\{F_{N+1} \text{ is empty}\}}\sum_{i=1}^N\nu\Big( \id_{\{F_{i+1} \text{ is empty}\}}\var_{\nu_i}(f) \Big) \bigg)\nonumber\\
& \, {{= \, e^{O( q^{-\a} \log(1/q)^3)}}} \mu\bigg(\id_{\{F_{N+1} \text{ is empty}\}}\sum_{i=1}^N \mu_{\L}\Big( \id_{\{F_{i+1} \text{ is empty}\}}\var_{\nu_i}(f) \Big) \bigg),
\label{eq:core1ter}
\end{align}
{{where the final inequality follows from the definition of $\nu_i$, and property~(b) of the fibers, which implies that the event $H_1 \cap \cdots \cap H_N$ has probability at least $p_1^3 = 1 - o(1)$ (since it is implied by the goodness of three blocks).}} 

{{Recall that, by property~(c) of the fibers, $F_i$ is contained in the closure, under the $\cU$-bootstrap process restricted to {{a $O(1)$-neighbourhood $U_i$ of the}} set $F_i\cup F_{i+1}$, of any set of empty sites containing $F_{i+1}$ and for which the event $H_i$ holds. We may therefore apply Lemma~\ref{lem:var0}, with $A := F_{i+1}$, $B := F_i$, $\cE := H_i$ and $U := U_i$, to obtain}}
\begin{gather}
\label{eq:core1bis}
\mu_{{{\L}}}\Big(\id_{\{F_{i+1} \text{ is empty}\}}\var_{\nu_i}(f) \Big)
\leq O(n_2) q^{-O(n_2)}\sum_{x \in {{U_i}}} \mu_{{{\L}}} \big( c_x\var_x(f) \big),
\end{gather}
since $|F_i| = O(n_2)$. 
Inserting~\eqref{eq:core1bis} into~\eqref{eq:core1ter} we obtain
\[
I_j \, \leq \, e^{O( q^{-\a} \log(1/q)^3)} \sum_{x \in {{\hat W_j}}} \mu\big( c_x \var_x(f) \big)
\]
for each $j \in \{1,2\}$, {{and some $O(1)$-neighbourhood $\hat W_j$ of $W_j$,}} as required.
\end{proof}

\begin{remark}\label{rem:whyEast}
{{We remark that our use of the generalized East chain (rather than the
generalised FA-1f chain) in the proof above was necessary (since for $\a$-rooted models Proposition~\ref{prop:critic:spread} can only be used to move infection in the $u$-direction), and also harmless (since in either case the bound we obtain is of the form $\exp\big( \widetilde O(q^{-\a}) \big)$, which is much smaller than $\exp\big( q^{-2\a} \big)$). In the proof for $\b$-unrooted models we will also use the generalized East chain, however, even though in that case we can move infection in both the $u$- and $-u$-directions, and doing so costs us a factor of $\log(1/q)$ in the exponent for models with $\b = \a$. This is because the method we use in this paper does not appear to easily allow us to use the generalised FA-1f chain in this setting.}}
\end{remark}

In order to conclude the proof of the proposition, it remains to
construct in detail the fibers for each case and to prove the basic
inequality~\eqref{eq:core1}.

\subsubsection{Construction of the fibers and the proof of~\eqref{eq:core1}.}
\label{sec:East-like} 

We {{will first define the helping events and prove~\eqref{eq:core1} in the (easier) case $j = 1$. Recall that}}  
\[
I_1 = \mu\left(\id_{\{\o_{\vec e_1} \in G_2\}}\id_{\{\o_{-\vec e_1}\in G_1\}} \var_{V}\big( f \tc G_1 \big) \right),
\] 
{{where $V = V_{(0,0)}$, and that $\o_{\vec e_1} \in G_2$ implies that there exists an empty quasi-stable half-ring $\cR$ of width $\kappa$ that fits in $V_{(1,0)}$ and is entirely contained in the rightmost quarter of $R_{(1,0)}$, and recall that this determines the length $\ell$ of $\cR$, and that $\ell \geq n_2/ m$. By translating $\cR$ slightly (without changing the set $\cR \cap \bbZ^2$) if necessary, we may also assume  that there are no sites of $\bbZ^2$ on the boundary of $\cR$ and in the interior of $R_{(1,0)}$. Let us also choose $\kappa$ so that the vector $\kappa u$ has integer {{coordinates}}. Now, for each such quasi-stable half-ring $\cR$, set 
$$N = N(\cR) := \min\big\{ j : \cR + j \kappa u \subset V_{(-1,0)} \big\}$$
and define $F_i = F_i(\cR) := \hat F_i \cap \bbZ^2$, where 
\[
\hat F_i = \hat F_i(\cR) := \cR + (N + 1 - i) \kappa u,
\]
for each $1 \leq i \leq N + 1$. Note that $V_{(0,0)} \subset \bigcup_{i=1}^N F_i$, and that (by our choice of $\kappa$) there are no sites of $\bbZ^2$ on the boundary of $\hat F_i$ in the interior of $R_{(-1,0)} \cup R_{(0,0)} \cup R_{(1,0)}$.

\begin{definition}\label{def:helping}
For each $\cR$ and $i \in [N]$, let $H_i$ denote the event that for each quasi-stable direction $v \in C$ and every $v$-strip $S$ such that the segment $\partial^{\rm ext}(S) \cap \hat F_i$ has length at least $n_2/(2m)$, there exists an empty helping set $Z \subset F_i$ for~$S$. 
\end{definition}
}}

Notice that in the above definition we do not require {{the $v$-strip $S$}} to be contained in $\hat F_i$. Observe that if the blocks $V_{(-1,0)}$, $V_{(0,0)}$ and $V_{(1,0)}$ are all good, then the event $H_i$ occurs for every $i \in [N]$. Now define $H_{\cR}$ to be the event that $\cR$ is (up to translates preserving the set $\cR \cap \bbZ^2$) the rightmost empty quasi-stable half-ring in $R_{(1,0)}$, and observe that, conditional on $H_{\cR}$, the events $\{H_i\}_{i=1}^N$ are
independent. Moreover, by Corollary~\ref{cor:critic:spread}, {{and since $\kappa$ is sufficiently large,}} if $F_{i+1}$ is empty and $H_{i}$ occurs, then the $\cU$-bootstrap process restricted to {{a $O(1)$-neighbourhood of}} the set $F_i \cup F_{i+1}$ is able to infect $F_i$. The fibers $\{F_i\}_{i=1}^{N+1}$ therefore satisfy conditions (a), (b) and (c) of Section~\ref{sec:core}. Recall that we write $\L = \bigcup_{i=1}^N F_i$. We make the following claim, which implies~\eqref{eq:core1} in the case $j = 1$:
  
\begin{claim}\label{claim:core1}
\begin{equation}
\label{eq:claim1}
I_1 \leq \frac{1}{p_1}  \sum_{\cR} \mu\Big( \id_{H_\cR} \var_{\L}\big( f \;\big|\; H_1 \cap \cdots \cap H_N \big) \Big).
  \end{equation}
\end{claim}

{{Note that the sum in the claim is over equivalence classes of quasi-stable half-rings $\cR$, where two half-rings are equivalent if they have the same intersection with $\bbZ^2$.}}

\begin{proof}[Proof of Claim~\ref{claim:core1}]
{{We first claim that 
\begin{equation}\label{claim:core1:step1}
I_1 \leq \frac{1}{p_1} \sum_{\cR} \mu\Big( \id_{H_\cR} \id_{\{\o_{\pm \vec e_1}\in G_1\}}\mu_V\Big( \id_{\{\o_0 \in G_1\}} \big( f - a \big)^2 \Big) \Big),
\end{equation}
where $\o_0\equiv \o_{V_{(0,0)}}$ and, for any $\o \in H_{\cR}$, we set 
\[
a = a\big( \o_{\bbZ^2 \setminus \L} \big) := \mu_\L\big( f \;\big|\; H_1 \cap \cdots \cap H_N \big),
\] 
noting that, on the event $H_{\cR}$, the set $\L$ and the fibers become deterministic. To prove~\eqref{claim:core1:step1} we use Definition~\ref{def:goodsets}, which implies that if $V_{(1,0)}$ is super-good then it is also good, and also that the event $H_{\cR}$ holds for some $\cR$, and the standard inequality $\var(X) \leq \bbE\big[ (X-a)^2 \big]$, which holds for any $a\in \bbR$ and any random variable $X$.

Recalling that if the blocks $V_{(-1,0)}$, $V_{(0,0)}$ and $V_{(1,0)}$ are all good, then the event $H_i$ occurs for every $i \in [N]$, it follows from~\eqref{claim:core1:step1} that}}
\begin{align*}
 I_1 & \, \leq \, \frac{1}{p_1} \sum_{\cR} \mu\Big( \id_{H_\cR} \, \mu_{{{\L}}} \Big( \id_{H_1 \cap \cdots  \cap H_N} \big( f - a \big)^2 \Big) \Big)\\
& \, \leq \, \frac{1}{p_1} \sum_{\cR} \mu\Big( \id_{H_\cR} \var_\L\big( f \;\big|\; H_1 \cap \cdots  \cap H_N \big) \Big),
\end{align*}
where the last inequality follows from our choice of $a$ and the trivial inequality
\[
  \mu_\L\Big( \id_{H_1 \cap \cdots  \cap H_N} \big( f - a \big)^2 \Big) \leq  \mu_\L\Big( \big( f - a \big)^2  \;\big|\; H_1 \cap \cdots  \cap H_N \Big).
\] 
{{This proves the claim, and hence~\eqref{eq:core1} in the case $j = 1$.}}
\end{proof}

We now turn to {{the analysis of}} the term 
$$I_2 = \mu\left(\id_{\{\o_{\vec e_2}\in G_2\}}\id_{\{\o_j\in G_1\, \forall j\in \bbL^+\}}
\var_{V}(f) \right).$$
In this case we need to modify the definition of the fibers $F_i$ in order to take into account
the different local neighbourhood $W_2$ of $V_{(0,0)}$ and the different good and super-good environment in $W_2$ (see Figures~\ref{fig:setting} {{and~\ref{fig:strategy}}}). 

{{First, let us define $H_\cR$ to be the event that $\cR$ is (up to translates preserving the set $\cR \cap \bbZ^2$) the rightmost empty quasi-stable half-ring of width $\kappa$ that fits in $V_{(0,1)}$, and observe that the length $\ell$ of $\cR$ satisfies $\ell \geq n_2/ m$, and that the event $\o_{\vec e_2} \in G_2$ implies that $H_\cR$ holds for some $\cR$ in the rightmost quarter of $R_{(0,1)}$. As before, we may choose $\cR$ so that there are no sites of $\bbZ^2$ on its boundary in the interior of $R_{(0,1)}$.

The fibers $\{F_i\}_{i=1}^{N+1}$ will be similar to those used above to bound $I_1$, but some of the $v_1$-strips (which form the bottom portion of each fiber) will be elongated as the fibers ``descend" the upward $\kappa$-stair in $V_{(1,0)}$, see Figure~\ref{fig:strategy}. (Recall that we call these objects \emph{elongated quasi-stable half-rings}.) To be precise, let us write $L(\cR)$ for the two-way infinite $v_1$-strip of width $\kappa$ that contains the $v_1$-strip of $\cR$, and define
$$N = N(\cR) := \min\bigg\{ j : V_{(0,0)} \subset \bigcup_{i = 1}^j \big( L(\cR) + {{i}} \kappa u \big) \bigg\}.$$
Now,}} recall that $a = \Theta(n_2)$ is the number of rows {{of $V$}}, and recall Definition~\ref{def:stair}. Let {{$\cM = \big\{ M^{(1)}, \ldots, M^{(a)} \big\}$}} be an upward $\kappa$-stair contained in the {{leftmost}} quarter of $V_{(1,0)}$, {{and define the fibers $\hat F_i = \hat F_i(\cR,\cM)$ recursively as follows:
\begin{itemize}
\item[$(a)$] $\hat F_{N+1} := \cR$;
\item[$(b)$] For each $i \in [N]$ set $\hat F'_i := \hat F_{i+1} + \kappa u$. Now define:
\begin{itemize}
\item[$(i)$] $\hat F_i$ to be an elongated version 
of $\hat F'_i$ such that $\big( \hat F_i \setminus \hat F_i' \big) \cap \bbZ^2$ is a subset of a step of $\cM$, if such an elongated quasi-stable half-ring exists;  
\item[$(ii)$] $\hat F_i := \hat F_i'$ otherwise.
\end{itemize}
\end{itemize}
As before, we set $F_i = F_i(\cR,\cM) := \hat F_i \cap \bbZ^2$ for each $1 \leq i \leq N + 1$. {{Let us write $H_\cM$ for the event}} that $\cM$ is the first (in some arbitrary ordering) empty upward {{$\kappa$}}-stair contained in the leftmost quarter of $V_{(1,0)}$. We can now define the helping events.

\begin{definition}\label{def:helping:2}
For each $\cR$ and $\cM$, and each $i \in [N]$, let $H_i$ denote the event that for each quasi-stable direction $v \in C$ and every $v$-strip $S$ such that the segment 
$$\partial^{\rm ext}(S) \cap \hat F_i \cap \big( R_{(-1,1)} \cup R_{(0,1)} \big)$$ 
has length at least $n_2/(2m)$, there exists an empty helping set $Z \subset F_i$ for~$S$.
\end{definition}

Observe that if the blocks $V_{(-1,1)}$ and $V_{(0,1)}$ are both good, then the event $H_i$ occurs for every $i \in [N]$. Moreover, conditional on the event $H_\cR {{ \, \cap \, H_\cM}}$, the events $\{H_i\}_{i=1}^N$ are independent and, by Corollary~\ref{cor:critic:spread} {{(see Remark~\ref{rem:helping-sets})}}, if $F_{i+1}$ is empty and {{the events $H_\cM$ and}} $H_i$ occur, then the $\cU$-bootstrap process restricted to {{a $O(1)$-neighbourhood of}} the set $F_i \cup F_{i+1}$ is able to infect $F_i$. {{It follows that if the event $H_\cR \, \cap \, H_\cM$ occurs, then}} the fibers $\{F_i\}_{i=1}^{N+1}$ satisfy conditions (a), (b) and (c) of Section~\ref{sec:core}. 


We make the following claim, which implies~\eqref{eq:core1} in the case $j = 2$:
  
\begin{claim}\label{claim:core2}
\begin{equation}
\label{eq:claim2}
I_2 \leq \frac{1}{p_1}  \sum_{\cR, \cM} \mu\Big( \id_{H_\cR} {{\id_{H_\cM}}} \var_{\L}\big( f \;\big|\; H_1 \cap \cdots \cap H_N \big) \Big).
  \end{equation}
\end{claim}

The proof of Claim~\ref{claim:core2} is identical to that of Claim~\ref{claim:core1}. As discussed above, this completes the proof of the Proposition~\ref{prop:I1:I2}, and hence of Theorem~\ref{mainthm:2} in the case where $\cU$ is $\a$-rooted and Assumption~\ref{easy} holds.
}}

\subsection{The $\b$-unrooted case}
\label{sec:beta-unrooted}

In this section we assume that the bilateral difficulty $\b$ of the updating rule $\cU$ is smaller than $2\a$. We will prove that, if Assumption~\ref{easy} holds, then there exists a constant $\lambda = \lambda(\cU)$ such that 
$$\bbE_\mu(\t_0) \leq \frac{\trel(q,\cU)}{q} \leq {{\exp\Big( \lambda \cdot q^{-\b} \big( \log(1/q) \big)^3 \Big)}}$$
for all sufficiently small $q > 0$. Note that the first inequality follows from~\eqref{eq:mean-infection}, and so, by Definition~\ref{def:PC}, it will suffice to prove that
\begin{equation}\label{eq:unrooted:aim}
\var(f) \leq {{\exp\Big( \lambda \cdot q^{-\b} \big( \log(1/q) \big)^3 \Big)}} \sum_{x \in \bbZ^2} \mu\big( c_x \var_x(f) \big)
\end{equation}
for some $\lambda = \lambda(\cU)$ and all local functions $f$. We will deduce a bound of the form~\eqref{eq:unrooted:aim} from the \emph{unoriented} constrained Poincar\'e inequality~\eqref{eq:9FA} and Corollary~\ref{cor:critic:spread}. 


{{Recall from Section~\ref{sec:critic:spread} that $C \subset S^1$ is an open semicircle such that $\a(v) \le \b$ for each $v \in C \cup - C$, and that we let $u$ be the mid-point of $C$. Similarly to Section~\ref{sec:critical:rooted}, we set
\begin{equation}\label{def:n1n2:unrooted}
n_1 = \big\lfloor q^{-{{3}}\kappa} \big\rfloor \qquad \text{and} \qquad n_2 = \big\lfloor \kappa^4 q^{-\b} \log (1/q) \big\rfloor,
\end{equation}
where $\kappa = \kappa(\cU)$ is a sufficiently large constant.}}

We need to slightly modify the definition of the good and super-good events $G_1$ and $G_2$ as follows. {{Let $(v_1,\dots, v_m)$ be the quasi-stable directions in $C$, and let $(v_1,\dots, v_{m'})$ be the quasi-stable directions in $-C$ (see Definition~\ref{def:quasi-stable}). As in Section~\ref{sec:critical:rooted}, it follows by Assumption~\ref{easy} that we may choose the voracious sets so that the helping sets for $S_v$ are subsets of $\partial^{\rm ext}(S_v)$ for each quasi-stable direction $v \in C \cup - C$.}}


{{
\begin{definition}[Good and super-good events]
\label{def:new goodsets} \ 
\begin{enumerate}
  \item The block $V_i = R_i \cap \bbZ^2$ satisfies the \emph{good} event $G_1$ iff:  
  \begin{enumerate}[(a)]
       \item for each quasi-stable direction $v \in C$ and every $v$-strip $S$
    such that the length of the segment $\partial^{\rm ext}(S) \cap R_i$ is at least $n_2/(4m)$, there
    exists an empty helping set $Z \subset {{\partial^{\rm ext}(S) \cap V_i}}$ for~$S$;
       \item for each quasi-stable direction $v \in -C$ and every $v$-strip $S$
    such that the length of the segment $\partial^{\rm ext}(S) \cap R_i$ is at least $n_2/(4m')$, there
    exists an empty helping set $Z \subset {{\partial^{\rm ext}(S) \cap V_i}}$ for~$S$; 
    \item there exist two empty upward $\kappa$-stairs, one within the leftmost quarter of $V_i$, and one within the rightmost quarter of $V_i$. 
  \end{enumerate}
  \item The block $V_i$ satisfies the \emph{super-good} event $G_2$ iff it satisfies the good event $G_1$, and moreover there exist two empty quasi-stable half-rings $\cR^+$ and $\cR^-$, of width $\kappa$, that both fit in $V_i$, with $\cR^+$ relative to $C$ and entirely contained in the rightmost quarter of $R_i$, and with $\cR^-$ relative to $-C$ and entirely contained in the leftmost quarter of $R_i$.
\end{enumerate}
\end{definition}
}}

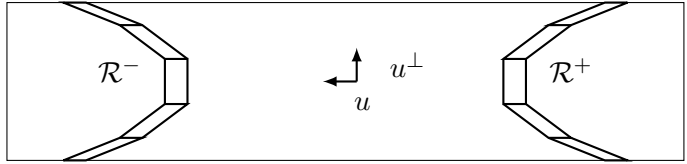
\begin{figure}[ht]
  \centering
  \begin{tikzpicture}[>=latex, scale=0.15]
\begin{scope}[shift={(5,0)}, rotate around={180:(10,10)}]
\draw [thick] (0,0)--(5,2)--(7,2)--(2,0)--(0,0);  
\draw [thick] (5,2)--(9,5)--(11,5)--(7,2);
\draw [thick] (9,5)--(9,9)--(11,9)--(11,5);
\draw [thick] (9,9)--(5,12)--(7,12)--(11,9);
\draw [thick] (5,12)--(0,14)--(2,14)--(7,12);
\end{scope}
\begin{scope}[shift={(-5,0)}]
\node at (6.5,11) {$u$};
\node at (10.5,14.5) {$u^\perp$};
 \draw [->,thick, shift={(-2,0)}] (8,13) -- (5,13); 
 \draw [->,thick, shift={(-2,0)}] (8,13)--(8,16);
\node at (25,14) {$\cR^+$};
\node at (-15,14) {$\cR^-$};
\end{scope}
\begin{scope}[shift={(-25,+6)}]
\draw [thick] (0,0)--(5,2)--(7,2)--(2,0)--(0,0);  
\draw [thick] (5,2)--(9,5)--(11,5)--(7,2);
\draw [thick] (9,5)--(9,9)--(11,9)--(11,5);
\draw [thick] (9,9)--(5,12)--(7,12)--(11,9);
\draw [thick] (5,12)--(0,14)--(2,14)--(7,12);
\end{scope}
\draw (-30,6) rectangle (30,20);
  \end{tikzpicture}
   \caption{The two quasi-stable half-rings $\cR^\pm$. For simplicity they have been
     drawn as {{mirror images of one another}}, although in general {{the quasi-stable directions do not necessarily have this property}}.}\label{fig:+/- half-ring}
\end{figure}

It is easy to check that, with the new definition of the good and super-good events, Lemma \ref{lem:p1p2} still holds. {{It follows, by Theorem~\ref{thm:CPI}, that the unconstrained Poincar\'e inequality~\eqref{eq:9FA} holds for any local function $f$, i.e.,
\begin{align}\label{eq:7:1}
& \var(f) \leq T(p_2) \bigg( \sum_{\varepsilon=\pm 1}\sum_{i\in \bbZ^2}\mu\left(\id_{\{\o_{i+\varepsilon\vec e_2} \in G_2\}}\id_{\{\o_{j}\in G_1\, \forall j\in \bbL^{\varepsilon}(i)\}} \var_i(f)\right) \nonumber\\
& \hspace{3.5cm} + \sum_{\varepsilon=\pm 1}\sum_{i\in \bbZ^2}\mu\left(\id_{\{\o_{i+\varepsilon\vec e_1} \in G_2\}}\id_{\{\o_{i-\varepsilon \vec e_1}\in G_1\}} \var_i\big( f \tc G_1 \big) \right) \bigg). 
\end{align}
with 
\[
T(p_2) \, = \, p_2^{-O(1)} \, = \, \exp\Big( O\big( q^{-\b} \log(1/q)^2 \big) \Big). 
\]
As in Section~\ref{sec:critical:rooted}, using translation invariance it will suffice to bound from above, for a fixed (and arbitrary) local function~$f$, the following four generic terms:
\[
I^\pm _1(i) := \mu\left(\id_{\{\o_{i\pm\vec e_1} \in G_2\}} \id_{\{\o_{i \mp \vec e_1} \in G_1\}} \var_{V_i}(f \tc G_1)\right),
\] 
and
\[
I^\pm_2(i) := \mu\left(\id_{\{\o_{i \pm \vec e_2} \in G_2\}} \id_{\{\o_j \in G_1 \, \forall j \in \bbL^\pm(i)\}} \var_{V_i}(f) \right).
\]
}}

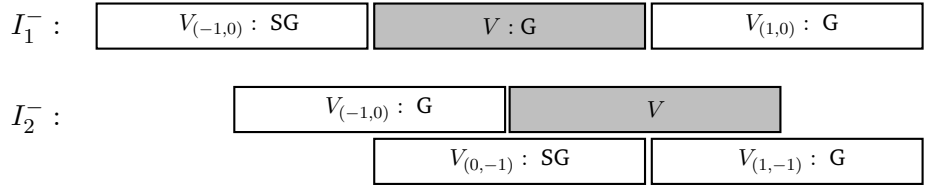
\begin{figure}[ht]
\centering
\begin{tikzpicture}[scale=0.3]
\node at (-15,1) {$I^-_1:$};
\node at (-15,-3) {$I^-_2:$};
\begin{scope}
\draw [thick, fill=lightgray] (0,0) rectangle (12,2);
\draw [thick,  xshift=12.3cm] (0,0) rectangle (12,2);
\draw [thick,  xshift=-12.3cm] (0,0) rectangle (12,2);
\node [scale=0.8] at (6,1) {$V:\text{G}$};
\node [scale=0.8] at (18.5,1) {$V_{(1,0)}:\text{ G}$};
\node [scale=0.8] at (-6,1) {$V_{(-1,0)}:\text{ SG}$};
\end{scope}
\begin{scope}[shift={(0,-6)}]
\draw [thick] (0,0) rectangle (12,2);
\draw [thick, xshift=12.3cm] (0,0) rectangle (12,2);
\draw [thick,xshift=-6.2cm,yshift=2.3cm] (0,0) rectangle
(12,2);
\node [scale=0.8] at (12.5-12.3,3.3) {$V_{(-1,0)}:\text{ G}$};
\draw [thick, fill=lightgray, xshift=6cm,yshift=2.3cm] (0,0) rectangle
(12,2);
\node [scale=0.8] at (6,1) {$V_{(0,-1)}:\text{ SG}$};
\node [scale=0.8] at (18.5,1) {$V_{(1,-1)}:\text{ G}$};
\node [scale=0.8] at (12.5,3.3) {$V$};
\end{scope}
\end{tikzpicture}
\caption{{{In $I_1^-$ the block $V\equiv V_{(0,0)}$ is conditioned to be good (G), while the blocks $V_{(1,0)}$ and $V_{(-1,0)}$ are good and super-good (SG) respectively. In $I_2^-$ the blocks $V_{(-1,0)}$ and $V_{(1,-1)}$ are good, the block $V_{(0,-1)}$ is super-good, and $V$ is unconditioned.}}}\label{fig:setting2}
\end{figure}

{{Define $W^+_1 = W^-_1 = V_{(0,0)} \cup V_{(-1,0)} \cup V_{(1,0)}$, and $W_2^+ = V_{(0,0)} \cup V_{(-1,0)} \cup V_{(1,0)} \cup V_{(-1,1)} \cup V_{(0,1)}$ and $W_2^- = V_{(0,0)} \cup V_{(1,0)} \cup V_{(-1,0)} \cup V_{(1,-1)} \cup V_{(0,-1)}$. The following upper bounds on $I_1^\pm$ and $I_2^\pm$ (cf. Proposition~\ref{prop:I1:I2}) follow exactly as in Section~\ref{sec:critical:rooted}.

\begin{proposition}
\label{prop:I:plusminus}
For each $j \in \{1,2\}$, there exist $O(1)$-neighbourhoods $\hat W^\pm_j$ of $W^\pm_j$ such that
\begin{align*}
  I^\pm_j \, \leq \exp\Big( O\big( q^{-\b} \log(1/q)^{{3}} \big) \Big) \sum_{x \in \hat W^\pm_j} \mu\big(c_x \var_x(f)\big).
\end{align*}
\end{proposition}

}}



\begin{proof}[{{Sketch proof of Proposition~\ref{prop:I:plusminus}}}]
The terms $I_1^+$ and $I_2^+$ can be treated exactly as the terms $I_1$ and $I_2$ analysed in
the previous section, because the new good and super-good events
imply the good and super-good events for the $\a$-rooted case. {{We may therefore repeat the proof of Proposition~\ref{prop:I1:I2}, with the only difference being that $n_2$ is now as defined in~\eqref{def:n1n2:unrooted}, to obtain the claimed bounds on $I_1^+$ and $I_2^+$.}} 

{{For the new terms, $I_1^-$ and $I_2^-$ (which are illustrated in Figure~\ref{fig:setting2}), the argument is exactly the same}} after a rotation of $\pi$ of the coordinate axes. {{Indeed, a good block now contains suitable empty helping sets for the quasi-stable directions in $-C$ (as well as $C$), and an empty upward $\kappa$-stair in the rightmost quarter (as well as the leftmost), and a super-good block contains an empty quasi-stable half-ring relative to $-C$ in the leftmost quarter (as well as one relative to $C$ in the rightmost quarter). Such a rotation therefore transforms $I_1^-$ and $I_2^-$ into $I_1^+$ and $I_2^+$, and so the proof of the claimed bounds is once again identical to that of Proposition~\ref{prop:I1:I2}.}}
\end{proof}

{{
\begin{remark}\label{rmk:losing:log}
As noted in Remark~\ref{rem:whyEast}, our application of the generalized East chain in the proof above cost us a factor of $\log(1/q)$ in the exponent. More precisely, this log-factor was lost in step~\eqref{eq:losing:log} of the proof of Proposition~\ref{prop:I:plusminus}, when (roughly speaking) we passed through an energy barrier corresponding to the simultaneous existence of about $\log(1/q)$ empty quasi-stable half-rings in a single block. As stated precisely in Conjecture~\ref{unbalanced:conj}, we expect that (at least for models with $\b = \a$) the true relaxation time does not contain this additional factor of $\log(1/q)$. 
\end{remark}
}}

{{Combining Proposition~\ref{prop:I:plusminus} with~\eqref{eq:7:1}, and noting that $|\hat W^\pm_j| = q^{-O(1)}$, we obtain a final Poincar\'e inequality of the form~\eqref{eq:unrooted:aim}, i.e.,
\[
\var(f) \leq \exp\Big( O\big( q^{-\b} \log(1/q)^3 \big) \Big) \sum_{x \in \bbZ^2} \mu\big( c_x \var_x(f) \big),
\]
as required. This completes the proof of Theorem~\ref{mainthm:2} for update families $\cU$ such that Assumption~\ref{easy} holds.}} \qed

\section{Critical KCM: removing the simplifying assumption}
\label{sec:fullgen}

In this section we {{explain how to modify the proof given in Section~\ref{sec:critic-ub} in order to avoid using Assumption~\ref{easy}. Since the argument is essentially identical for $\a$-rooted and $\b$-unrooted families, for simplicity we will restrict ourselves to}} the $\a$-rooted case. 

Our solution requires a slight change in the geometry of the quasi-stable half-ring. In what follows we will always work in the frame $(-u,u^\perp)$, where $u$ is the midpoint of the semicircle $C$ {{given by Lemma~\ref{cor:Ccritic} (cf.~Sections~\ref{sec:critic:spread} and~\ref{sec:critical:rooted}). 

Recall from Definition~\ref{def:strips} the definitions of the $+$- and $-$-boundaries of a $v$-strip $S$. The following key definition is illustrated in Figure~\ref{fig:general half-ring}.}} 

 \begin{definition}[Generalised quasi-stable half-rings]
\label{def:gen-qshr}
{{Let $(v_1,\dots, v_m)$ be the quasi-stable directions in $C$, ordered as in Definition \ref{def:half-ring}, and let $\cR$ be a quasi-stable half-ring of width $w$ and length $\ell$ relative to $C$. For each quasi-stable direction $v \in C$, let $S_v$ be the $v$-strip in $\cR$, and let $\hat S_v^{l}$ and $\hat S_{v}^{r}$ be the (unique) $v$-strips of width $w/3$ and length $\ell/3$ satisfying the following properties:
\begin{itemize}
\item[(i)] $\hat S_{v}^{l}$ and $\hat S_{v}^{r}$ each share exactly one corner with $S_{v}$; moreover each of these corners lies at the ``top" of $S_v$ when working in the frame $(-u,u^\perp)$.
\item[(ii)] $\partial_- (\hat S_{v}^{l}) \subset \partial_+(S_{v})$ and $\partial_- (\hat S_{v}^{r})\subset \partial_- (S_{v})$.
\end{itemize}
Set  
\[
S_{v}^g := \big( S_{v} \setminus \hat S_{v}^{r} \big) \cup \hat S_{v}^{l} ,
\]
and set 
\[
  \cR^g := \bigcup_{i=1}^m S_{v_i}^g.
\]
We call $\cR^g$ the \emph{generalised version of $\cR$}, 
and define the ``core'' of $\cR^g$ to be the set $\cR^g \cap \cR$.}}
\end{definition}

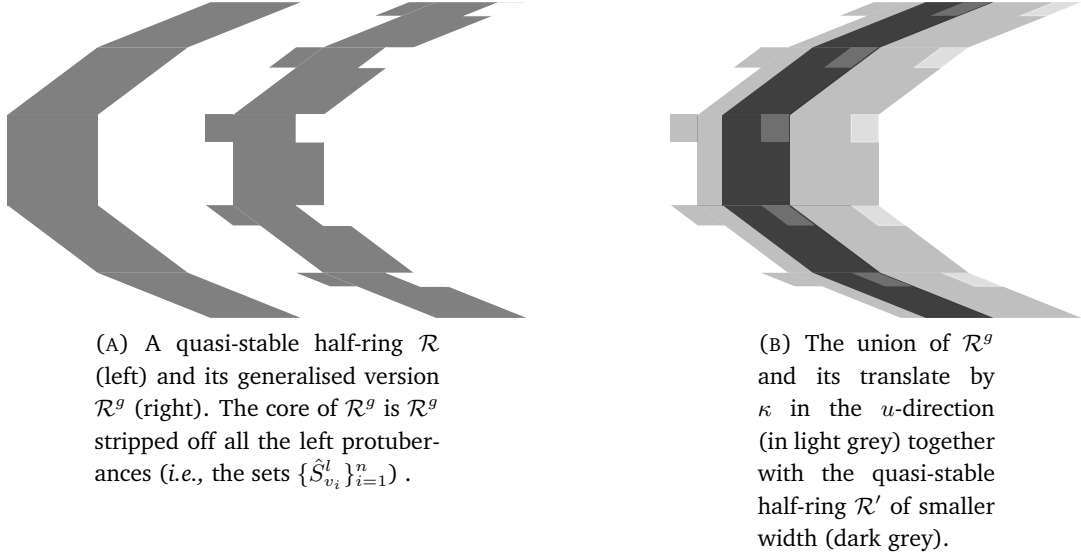
\begin{figure}[ht]
        \centering
  \subfloat[A quasi-stable half-ring $\cR$ (left) and its generalised
     version $\cR^g$ (right). The core of $\cR^g$ is $\cR^g$
      stripped off all the left protuberances (\ie the sets $\{\hat
      S_{v_i}^{l}\}_{i=1}^n$)  \label{fig:general half-ring}
.]{ 
  \begin{tikzpicture}[scale=0.3]
\begin{scope}[rotate around={180:(10,10)}]

 \begin{scope}
\fill [gray] (0,0)--(5,2)--(9,2)--(4,0)--cycle;  
\fill [gray] (5,2)--(9,5)--(13,5)--(9,2);
\fill [gray] (9,5)--(9,9)--(13,9)--(13,5);
\fill [gray] (9,9)--(5,12)--(9,12)--(13,9);
\fill [gray] (5,12)--(0,14)--(4,14)--(9,12);
\end{scope}

\begin{scope}[shift={(10,0)}]
\fill [gray] (0,0)--(5,2)--(9,2)--(4,0)--cycle;  
\fill [gray] (5,2)--(9,5)--(13,5)--(9,2);
\fill [gray] (9,5)--(9,9)--(13,9)--(13,5);
\fill [gray] (9,9)--(5,12)--(9,12)--(13,9);
\fill [gray] (5,12)--(0,14)--(4,14)--(9,12);
\end{scope}

\begin{scope}
\foreach \i in {0.3}{   
\fill [gray, shift={(4,0)}] (0,0)--(5*\i,2*\i)--(9*\i,2*\i)--(4*\i,0); 
\fill [white, draw=white] (0,0)--(5*\i,2*\i)--(9*\i,2*\i)--(4*\i,0); 
\fill [white, draw=white, shift={(5,2)}] (0,0)--(4*\i,3*\i)--(8*\i,3*\i)--(4*\i,0);
\fill [gray, shift={(9,2)}] (0,0)--(4*\i,3*\i)--(8*\i,3*\i)--(4*\i,0);
\fill [white, draw=white, shift={(9,5)}] (0,0)--(0,4*\i)--(4*\i,4*\i)--(4*\i,0); 
\fill [gray, shift={(13,5)}] (0,0)--(0,4*\i)--(4*\i,4*\i)--(4*\i,0); 
\fill [ white, draw=white, shift={(9,9)}] (0,0)--(-4*\i,3*\i)--(0,3*\i)--(4*\i,0); 
\fill [ gray, shift={(13,9)}] (0,0)--(-4*\i,3*\i)--(0,3*\i)--(4*\i,0); 
\fill [ white, draw=white, shift={(5,12)}] (0,0)--(-5*\i,2*\i)--(-\i,2*\i)--(4*\i,0); 
\fill [ gray, shift={(9,12)}] (0,0)--(-5*\i,2*\i)--(-\i,2*\i)--(4*\i,0); 
}
\end{scope}
\end{scope}
  \end{tikzpicture}
}
\hfill
\subfloat[The union of $\cR^g$ and its translate by $\kappa$ in the $u$-direction (in light grey) together with the {{quasi-stable half-ring}} 
$\cR'$ {{of}} smaller width (dark grey). \label{fig:general half-ring2}]{                
\begin{tikzpicture}[scale=0.3]
\begin{scope}[rotate around={180:(10,10)}]
\begin{scope}[opacity=1,shift={(3.9,0)}]
\fill [black] (0,0)--(5,2)--(8,2)--(3,0)--cycle;  
\fill [black] (5,2)--(9,5)--(12,5)--(8,2)--cycle;
\fill [black] (9,5)--(9,9)--(12,9)--(12,5)--cycle;
\fill [black] (9,9)--(5,12)--(8,12)--(12,9)--cycle;
\fill [black] (5,12)--(0,14)--(3,14)--(8,12)--cycle;
\end{scope}

\begin{scope}[opacity=0.5]

 \begin{scope}

\fill [gray] (0,0)--(5,2)--(9,2)--(4,0)--cycle;  
\fill [gray] (5,2)--(9,5)--(13,5)--(9,2);
\fill [gray] (9,5)--(9,9)--(13,9)--(13,5);
\fill [gray] (9,9)--(5,12)--(9,12)--(13,9);
\fill [gray] (5,12)--(0,14)--(4,14)--(9,12);

\foreach \i in {0.3}{   
\fill [gray, shift={(4,0)}] (0,0)--(5*\i,2*\i)--(9*\i,2*\i)--(4*\i,0); 
\fill [white, draw=white] (0,0)--(5*\i,2*\i)--(9*\i,2*\i)--(4*\i,0); 
\fill [white, draw=white, shift={(5,2)}] (0,0)--(4*\i,3*\i)--(8*\i,3*\i)--(4*\i,0);
\fill [gray, shift={(9,2)}] (0,0)--(4*\i,3*\i)--(8*\i,3*\i)--(4*\i,0);
\fill [white, draw=white, shift={(9,5)}] (0,0)--(0,4*\i)--(4*\i,4*\i)--(4*\i,0); 
\fill [gray, shift={(13,5)}] (0,0)--(0,4*\i)--(4*\i,4*\i)--(4*\i,0); 
\fill [ white, draw=white, shift={(9,9)}] (0,0)--(-4*\i,3*\i)--(0,3*\i)--(4*\i,0); 
\fill [ gray, shift={(13,9)}] (0,0)--(-4*\i,3*\i)--(0,3*\i)--(4*\i,0); 
\fill [ white, draw=white, shift={(5,12)}] (0,0)--(-5*\i,2*\i)--(-\i,2*\i)--(4*\i,0); 
\fill [ gray, shift={(9,12)}] (0,0)--(-5*\i,2*\i)--(-\i,2*\i)--(4*\i,0); 
}
\end{scope}

 \begin{scope}[shift={(4,0)}]
\fill [gray] (0,0)--(5,2)--(9,2)--(4,0)--cycle;  
\fill [gray] (5,2)--(9,5)--(13,5)--(9,2);
\fill [gray] (9,5)--(9,9)--(13,9)--(13,5);
\fill [gray] (9,9)--(5,12)--(9,12)--(13,9);
\fill [gray] (5,12)--(0,14)--(4,14)--(9,12);

\foreach \i in {0.3}{   
\fill [gray,shift={(4,0)}] (0,0)--(5*\i,2*\i)--(9*\i,2*\i)--(4*\i,0); 
\fill [gray] (0,0)--(5*\i,2*\i)--(9*\i,2*\i)--(4*\i,0); 
\fill [gray, shift={(5,2)}] (0,0)--(4*\i,3*\i)--(8*\i,3*\i)--(4*\i,0);
\fill [gray, shift={(9,2)}] (0,0)--(4*\i,3*\i)--(8*\i,3*\i)--(4*\i,0);
\fill [gray, shift={(9,5)}] (0,0)--(0,4*\i)--(4*\i,4*\i)--(4*\i,0); 
\fill [gray, shift={(13,5)}] (0,0)--(0,4*\i)--(4*\i,4*\i)--(4*\i,0); 
\fill [gray, shift={(9,9)}] (0,0)--(-4*\i,3*\i)--(0,3*\i)--(4*\i,0); 
\fill [gray, shift={(13,9)}] (0,0)--(-4*\i,3*\i)--(0,3*\i)--(4*\i,0); 
\fill [gray, shift={(5,12)}] (0,0)--(-5*\i,2*\i)--(-\i,2*\i)--(4*\i,0); 
\fill [gray, shift={(9,12)}] (0,0)--(-5*\i,2*\i)--(-\i,2*\i)--(4*\i,0); 
}
\end{scope}
\end{scope}
  
\end{scope}
  \end{tikzpicture}
}
\caption{A generalised quasi-stable half-ring.}\label{fig:generalised:ring}
 \end{figure}

Recall now {{the following two key ingredients of}} the proof given in the previous section {{(see Section \ref{sec:core})}} under the simplifying Assumption~\ref{easy}:
\begin{enumerate}[(i)]
\item a {{sufficiently large}} empty quasi-stable half-ring $\cR$ is able to completely infect its translate $\cR+\kappa u$, provided that {{a certain ``helping" event occurs;}}
\item {{the helping event depends only on the configuration inside $\cR$.}}    
\end{enumerate}
Here we prove a similar result for the generalised quasi-stable half-rings \emph{without} the simplifying assumption. {{We first define the helping event, cf. Definition~\ref{def:helping}.}}

\begin{definition}
Given a {{quasi-stable half-ring $\cR$ of length $\ell$ and width $\kappa$, we define $H(\cR)$ to be the event that for each quasi-stable direction $v \in C$ and every $v$-strip $S$ of length $\ell$ with $\partial_+(S) \subset \cR$, there exists an empty helping set for~$S$ in~$\cR^g$.}}
\end{definition}

{{If $H(\cR)$ holds, then we will say that $\cR^g$ is \emph{helping}. We will modify (see below) the good and super-good events $G_1$ and $G_2$ (see Definition~\ref{def:goodsets}) so that they guarantee that this helping event occurs, and choose the constant $\kappa = \kappa(\cU) > 0$ (as in Section~\ref{sec:critical:rooted}) so that the conclusion of Lemma~\ref{lem:p1p2} holds, and so that $\kappa u$ has integer coordinates. We will also choose our (generalised) quasi-stable half-rings so that there are no sites of $\bbZ^2$ on their boundary, except on the top and bottom boundaries of the rectangles $R_i$.}}

\begin{lemma}
Let {{$\cR$ be a quasi-stable half-ring of length $\ell$ and width $\kappa$, and let $\cR^g$ be the generalised version of $\cR$.}} Assume that the core of $\cR^g$ is empty and that $\cR^g$ and its translate $\cR^g + \kappa u$ are {{both}} helping. Then there exists a $O(1)$-neighbourhood $U$ of $\cR^g \cup \big( \cR^g + \kappa u \big)$ such that {{the $\cU$-bootstrap process}} restricted to $U$ is able to infect the core of $\cR^g + \kappa u$.
\end{lemma}

\begin{proof}
The {{lemma is a straightforward consequence of Proposition~\ref{prop:critic:spread}, using}} the geometry of the generalised quasi-stable half-rings. {{To spell out the details (cf. the proof of Corollary~\ref{cor:critic:spread}), fix $\cR$}} as in the lemma, and let $\cR'$ be {{any}} quasi-stable half-ring {{of length $\ell$ and width $\kappa/3$}} such that: 
\begin{itemize}
\item[(a)] {{$\cR' = \cR + \l u$}} for some $\l \ge 0$, and 
\item[(b)] ${{\cR'}} \subset \cR^g \cup \big( \cR^g + \kappa u)$,
\end{itemize}
see Figure~\ref{fig:general half-ring2}. {{We claim that, for every quasi-stable direction $v \in C$, there exists an empty helping set in $\cR^g \cup \big( \cR^g + \kappa u \big)$ for the $v$-strip $S'_{v}$ of $\cR'$. Indeed, this follows from the fact that $\cR^g$ and $\cR^g + \kappa u$ are both helping, since (by construction) either $\partial_+(S'_v) \subset \cR$ or $\partial_+(S'_v) \subset \cR + \kappa u$. 

Now, since the core of $\cR^g$ is empty, it follows, by Proposition \ref{prop:critic:spread}, that there exists a $O(1)$-neighbourhood $U$ of $\cR^g \cup \big( \cR^g + \kappa u \big)$ such that the $\cU$-bootstrap process restricted to $U$ can advance in the $u$-direction, and infect the core of $\cR^g + \kappa u$, as claimed.}}
\end{proof}

Given the above lemma, the proof of Theorem~\ref{mainthm:2} proceeds
exactly as the one given in Section~\ref{sec:critic-ub}, with only two
main changes:
\begin{enumerate}[(a)]
\item the fibers $\{F_i\}_{i=1}^N$ are no longer the
  quasi-stable half-rings (or their elongated version) but rather the
  generalised quasi-stable half-rings (or their elongated version);
\item when defining the generalised East process for the fibers, the
  constraining event $S_i^g$ (see Definition \ref{def:gen:East}), which in
  Section \ref{sec:critic-ub} was simply $S_i^g=\{F_{i}\text{ is empty}\},$ now
  becomes $S_i^g=\{\text{the core of $F_{i}$ is empty}\}.$    
\end{enumerate}
We leave the (straightforward) {{task of verifying the details}} to the reader.

\newpage
\appendix
\section{}

\subsection{Proof of Proposition~\ref{lem:gen-Poincare}}\label{sec:pf-gen-Poincare}

We will follow closely the proof of a very similar result proved
in~\cite{CFM2}*{Proposition 3.4}. Let $\{P_t\}_{t\ge 0}$ be the Markov
semigroup associated to either the generalised East chain or the
generalised FA-1f chain. Using reversibility, it follows (see, e.g.,~\cite{Saloff}*{Theorem~2.1.7}) that 
\begin{equation}
  \label{eq:A1bis}
\lim_{t\to \infty} - \frac 1t \log\Big( \max_\o \big\| P_t(\o,\cdot)
-\nu(\cdot) \big\|_{\rm \tiny TV} \Big)= \frac{1}{\trel},
\end{equation}
where $\|\cdot\|_{\rm \tiny TV}$ denotes the {{total}} variation distance.
 We now claim that {{for every function $f \colon \Omega \to \bbR$}} with $\|f\|_\infty\le 1$,
\begin{equation}\label{eq:A0}
\big\| P_t f - \nu(f) \big\|_\infty \leq C(n,q) e^{-t/t^*}
\end{equation}
for some $0 < C(n,q) < {{\infty}}$ and {{either}} $t^* \le T_{\text{East}}(n,\bar\a)/q$
or $t^* \le T_{\text{FA}}(n,\bar\a)/q$, depending on which of the two
models we are considering. Clearly \eqref{eq:A1bis} and \eqref{eq:A0} imply that $\trel \le t^*$
and {{(recalling Definition~\ref{def:PC})}} the proposition follows.

To prove~\eqref{eq:A0}, let $\t_x(\o)$ be the {{time of the}} first legal ring at
$x$ {{(that is, the first time that the state of $x$ is
    resampled)}} when the starting configuration is $\o$. Then, for any function
$f \colon \otimes_{x\in [n]} S_x \mapsto \bbR$ with $\nu(f)=0,$ we write
\begin{align}\label{eq:A1}
& \|P_tf\|_\infty \leq \max_\o {{\Big\{}} \Big| \bbE_\o\Big( f\big( \o(t) \big) \cdot \mathds
1_{\{\t_x(\o) \,<\, t\ \forall x\}} \Big) \Big| \nonumber\\
& \hspace{4.5cm} + \|f\|_\infty \cdot n \cdot \max_{x\in[n]} \, \bbP_\o\Big( \t_x(\o) > t \Big) {{\Big\}}}, 
\end{align}
where $\bbP_\o(\cdot)$ and $\bbE_\o(\cdot)$ denote the law and
associated expectation of the chain $\{\o(t)\}_{t\ge 0}$ with
$\o(0)=\o$. 

If $\eta(\o)=\{\eta_x(\o)\}_{x\in [n]}$ denotes the collection of the 0-1
variables $\eta_x=\mathds 1_{\{\o_x\in S^g_x\}}$ and $\hat \t_x(\eta)$ is the hitting time of the set {{$\big\{ \eta' : \ \eta'_x \ne \eta_x \big\}$,}} then $\big\{ \t_x(\o) > t \big\} \subset \big\{ \hat \t_x(\eta(\o)) > t \big\}$, and hence
$\bbP_\o\big( \t_x(\o) > t \big) \, \le \, \bbP_\o\big( \hat\t_x(\eta(\o)) > t \big)$.
Notice that $\eta(t)\equiv \eta(\o(t))$ is itself a Markov chain whose law
$\tilde\bbP_{{\eta}}(\cdot)$ coincides with that of either the
non-homogeneous East chain or the non-homogeneous FA-1f chain,
depending on the chain described by
$P_t$. Therefore, $\bbP_\o\big(\hat\t_x(\eta) > t \big)=\tilde \bbP_\eta\big(\hat\t_x(\eta) > t\big),$
where $\eta\equiv \eta(\o)$. 
Letting $\tilde \nu=\text{Ber}(\a_1)\otimes\dots\otimes
\text{Ber}(\a_n),$ we have that $\tilde \nu$ is the reversible measure
for the $\eta$-chain and that
\begin{align*}
\tilde \bbP_\eta\big( \hat\t_x(\eta) > t \big) & \, \leq \, \frac{1}{\min_\eta \tilde\nu(\eta)}\sum_{\eta'} \tilde\nu(\eta'){{\tilde \bbP}}_{\eta'}\big(\hat\t_x{{(\eta')}} > t \big)\\
& \, \leq \,
\begin{cases}
2q^{-n} \exp\big( -tq / T_{\text{East}}(n,\bar \a) \big) & \text{for the East process,}\\
2q^{-n} \exp\big( - tq / T_{\text{FA}}(n,\bar \a) \big) & \text{for the FA-1f process},      
  \end{cases}
\end{align*}
where the factor $q^{-n}$ comes from $\tilde\nu(\eta) \geq q^n$
and the exponential bounds follow from~\cite{Praga}*{Theorem 4.4}. In particular, the inverse of the exponential rate of decay (in $t$) of the second term in the r.h.s.~of \eqref{eq:A1} is smaller than $T_{\text{East}}(n,\bar\a)/q$ or $T_{\text{FA}}(n,\bar\a)/q,$ depending on which of the two models we are considering.   

We now analyse the first term in the r.h.s.~of \eqref{eq:A1}. Conditionally on {{the event}} $\bigcap_x \big\{ \t_x{{(\o)}} < t \big\}$ and on {{the vector}} $\eta(t) \, {{\in \{0,1\}^n}}$, the
variables $\big\{ \o_x(t) : x\in [n] \big\}$ are independent with law
$\otimes_{x\in [n]}\nu_x\big( \cdot\tc \eta_x(t) \big)$. Thus, if
$g(\eta'):=\nu\big( f \tc \eta' \big)$, then 
\begin{align*}
\bbE_\o\Big( f\big( \o{{(t)}} \big) \cdot \mathds 1_{\{\t_x(\o) \, < \, t\ \forall x\}} \Big) 
& \, = \, \bbE_\o\Big( g\big( \eta(t) \big) \cdot \mathds 1_{\{\t_x(\o) \, < \, t\ \forall x\}} \Big)\\ 
& \, = \, \tilde P_t\,g(\eta) -\, \bbE_\o\Big( g\big(\eta(t) \big) \cdot \mathds 1_{\{\max_x \t_x(\o) \, > \, t\}} \Big)
\end{align*}
where {{$\tilde P_t\,g(\eta) \equiv \tilde \bbE_\eta\big( g \big(\eta(t)\big) \big) = \bbE_\o\big(g\big(\eta(t)\big)\big)$}}. The last term in the r.h.s.~above can be analysed exactly as the
second term in the r.h.s.~of \eqref{eq:A1}. {{Moreover, by the}} Cauchy--{{Schwarz}}
inequality and \eqref{eq:relax:exp:decay}, the first term satisfies
\[
\|\tilde P_t\,g\|_\infty \leq
\frac{1}{\min_\eta \tilde\nu(\eta)} \var_{\tilde \nu}\big(\tilde
P_t\,g\big)^{1/2} \leq \frac{1}{q^n}e^{-\l t}\var_{\tilde \nu}(g)^{1/2},
\]
where $\l$ is either
$T_{\text{East}}(n,\bar\a)^{-1}$ or $T_{\text{FA}}(n,\bar\a)^{-1}$ depending on the
chosen model. This proves~\eqref{eq:A0}, and hence the proposition. \qed

\subsection{Proof of the scaling \eqref{eq:scaling}}

Recall that {{$q := \min\big\{ 1 - \a_x : x \in [n] \big\}$,}} and let $T_{\text{\tiny East}}\big( n,q\big)$ {{and}} $T_{\text{\tiny FA}}\big(n,q\big)$ be the relaxation times {{of}} the {{homogenous East and FA-1f chains}} on $[n]$ with parameters {{$\a_x = 1-q$}} for each ${{x}} \in [n]$. It {{was proved}} in~\cites{CMRT,CFM} that 
\[
T_{\text{\tiny East}}\big( n,q\big)= q^{-O(\min\{ \log n, \,
                                            \log(1/q) \})} \quad
                                            \text{and}\quad T_{\text{\tiny FA}}\big( n,q\big) = q^{-O(1)}. 
\]
Thus, it will {{suffice}} to prove that 
\[
T_{\text{\tiny East}}(n,\bar\a)= {{\frac{1}{q} \, \cdot \,}} T_{\text{\tiny East}}(n,q)\quad\text{ and }\quad
T_{\text{\tiny FA}}(n,\bar\a)= {{\frac{1}{q} \, \cdot \,}} T_{\text{\tiny FA}}(n,q).
\] 
For simplicity we only treat the East case, since the FA-1f case
follows by exactly the same arguments.

{{Consider the generalized East chain on $\O=[0,1]^n$}} in which each
vertex ${{x}} \in [n],$ with rate one and independently across $[n]$, is
resampled from the uniform measure on $[0,1]$ if either {{$x \le n-1$ and $\o_{x+1} \ge 1 - q$, or $x = n$.}} The chain is reversible w.r.t.~the uniform measure $\nu$
on $\O$ and, {{by}} Proposition \ref{lem:gen-Poincare}, {{we have}}
\begin{equation}\label{eq:prop:app:in:app:B}
\var_\nu(f) \leq \frac{1}{q} \cdot T_{\text{\tiny East}}(n,q) \cdot \sum_{{{x}} = 1}^n \nu\big( \vec c_{{x}}\var_{{x}}(f) \big)
\end{equation}
for {{every}} function $f \colon \O\mapsto \bbR$,  
{{since $\nu\big( \o_x \ge 1 - q \big) = q$ for each $x \in [n]$. (Recall that $\vec c_x(\o) = \id_{\{\o_{x+1} \ge 1-q\}}$}} if ${{x}} \le n-1$, and that ${{\vec c}}_n(\o) \equiv 1$.{{)}} 

{{Now, let}} $\eta = \{\eta_{{x}}\}_{{{x \in [n]}}}$ with $\eta_{{x}} := \id_{\{\o_{{x}} < \a_{{x}} \}}$, and, for an arbitrary {{function}} $g \colon \{0,1\}^n \mapsto \bbR$, 
{{set}} $f(\o) := g\big( \eta(\o) \big)$. {{Note that $\eta_x \le \id_{\{\o_x \le 1-q \}}$ (by the definition of $q$), and that}} the law of the variables $\eta$
w.r.t.~$\nu$ is the product Bernoulli measure
$\pi=\text{Ber}(\a_1)\otimes\dots\otimes \text{Ber}(\a_n)$. {{Therefore, applying~\eqref{eq:prop:app:in:app:B}}} to $f$, we {{obtain}}
\[
\var_\pi(g)= 
\var_\nu(f) \leq \frac{1}{q} \cdot T_{\text{\tiny East}}(n,q) \cdot \bigg( \sum_{x=1}^{n-1} \pi\Big( \id_{\{\h_{x+1}=0\}} \var_x(g) \Big) + \pi\big( \var_n(g) \big) \bigg).
\] 
{{The right-hand side of this inequality is exactly $C \cdot \cD(g)$, where $C = 1/q \cdot T_{\text{\tiny East}}(n,q)$ and $\cD(g)$ is the Dirichlet form of $g$ associated to the generator of the non-homogenous East model. Since $g$ was an arbitrary function, it follows by Definition~\ref{def:PC} that $T_{\text{\tiny East}}(n,\bar\a) = 1/q \cdot T_{\text{\tiny East}}(n,q)$, as required. As noted above, the proof that $T_{\text{\tiny FA}}(n,\bar\a) = 1/q \cdot T_{\text{\tiny FA}}(n,q)$ is identical.}}
\qed

\subsection{Proof of Proposition~\ref{lem:MT:prop34}}

We will deduce the proposition from~\cite{MT}*{Theorem 1}. The
deduction is almost exactly the same as that
of~\cite{MT}*{Proposition~3.4}, but for completeness we give the
details. Set $\ell= \big\lceil \log(1/p_2) \big\rceil$, $L = \big\lfloor 1/p_2^2 \big\rfloor$, and for each $i \in \bbZ^2$, define
\[
C_i(\ell) = \bigcup_{k=0}^\ell \big\{ i + \vec e_2 + k\vec e_1 \big\}.
\] 
Let also $\cP_i(\ell,L)$ be the family of oriented paths starting in $C_i(\ell)$ and of length $L$.
We define two families of events $\big\{ A_i^{(1)}, A_i^{(2)}\big\}_{i\in \bbZ^d}$ as follows: 
\begin{align*}
  A_i^{(1)} & = \big\{ \text{$\o_j\in G_1$ for all $j\in C_i(\ell)\cup
             \{i+\vec e_1\}\cup \{i+\vec e_2-\vec e_1\}$} \big\},\\
A_i^{(2)} & = \big\{ \text{there exists a good path in $\cP_i(\ell,L)$ and the smallest 
good one is super-good} \big\},
\end{align*}
{{where, if there is more than one smallest good path, then we choose the leftmost one.}}

Observe that $A_i^{(1)} \cap A_i^{(2)} \subset \G_i$, {{since $A_i^{(1)}$ implies that the smallest good path in $\cP_i(\ell,L)$ starts at $i + \vec e_2$,\footnote{This follows from the observation that the word (of length $L$) obtained from $W \in \{\vec e_1,\vec e_2\}^L$ by adding $\vec e_1$ at the start and removing the final letter is at most $W$ in alphabetical order.} 
and hence is equal to the smallest  path in the definition of $\G_i$.}} We now want to apply~\cite{MT}*{Theorem 1} to the two families of
constraints $\big\{ c_i^{(k)} \big\}_{i\in \bbZ^2}$, where $c_i^{(k)}
:= \id_{\{A_i^{(k)}\}}$ for each $k \in \{1,2\}$. To do so, we need to check the following two conditions:
\begin{itemize}
\item[$(a)$] there exists a two-way infinite sequence of sets $(\ldots,V_{-2},V_{-1},V_0, V_1, V_2, \ldots)$, with $V_n \subset V_{n+1}$ for every $n \in \bbZ$ and $\bigcup_n V_n = \bbZ^2$, such that if $i \not\in V_n$, then the event $A_i^{(k)}$ is independent of the collection of variables $\big\{ \o_i : i \in V_{n+1} \big\}$;
\item[$(b)$] there exists a family $\big\{ \l_I : \emptyset \ne I \subset \{1,2\} \big\}$ of positive constants such that the key condition~\cite{MT}*{equation (2.1)} holds.
\end{itemize}
To see $(a)$, let the sets $V_n$ be all translations of the closed half-space 
$$\ol{\H}_{(1,2)} = \big\{ x \in \Z^2 : \< x,(1,2) \> \le 0 \big\}$$ 
by elements of $\bbZ^2$ (ordered in the obvious way). Now, observe that if $i \not\in V_n$ then $V_{n+1} \subset \ol{\H}_{(1,2)} + i$, and the event $A_i^{(k)}$ is indeed independent of the variables in $\ol{\H}_{(1,2)} + i$.

To prove $(b)$, set $\l_I = 1$ for every non-empty set $I \subset \{1,2\}$, and note that the event $A_i^{(1)}$ depends on $\ell + 3$ variables, and that $A_i^{(2)}$ depends on at most $(L + \ell)^2$ variables. It follows that there exists a constant $\hat \d > 0$ such that~\cite{MT}*{equation (2.1)} holds if 
\begin{equation}
  \label{eq:A10}
\ell \Big( 1 - \mu\big( A_i^{(1)} \big) \Big) + (L + \ell)^2 \Big( 1 - \mu\big( A_i^{(2)} \big) \Big) \leq \hat \d.
\end{equation}
We now claim that if the constant $\d$ of Proposition~\ref{lem:MT:prop34} is chosen to be sufficiently small, then~\eqref{eq:A10} holds. In order to prove this, it is enough to observe that, by the union bound, 
\[
1 - \mu\big( A_i^{(1)} \big) \leq (\ell + 3)(1-p_1),
\] 
and that
\begin{align*}
1 - \mu\big( A_i^{(2)} \big) & \leq \, \mu\big(\text{there is no good path in } \cP_i(\ell,L) \big)\\
& \hspace{1.5cm} + \max_{\g \in \cP_i(\ell,L)} \mu\big( \g \text{ is not super-good} \, \big| \, \g \text{ is good} \big)\\
& \leq \, e^{-m(p_1)\ell} + \big( 1 - p_2 \big)^L,  
\end{align*}
with $\lim_{p_1 \to 1} m(p_1) = \infty,$ by a standard Peierls bound and by the FKG inequality. 
In conclusion, if $\d > 0$ is sufficiently small then we may apply~\cite{MT}*{Theorem 1}, which gives
\[
\var(f)\leq 4 \sum_{i\in \bbZ^2}\mu\Big(\id_{\{A_i^{(1)}\cap A_i^{(2)}\}}\var_i(f)\Big) \leq 4 \sum_{i\in \bbZ^2}\mu\Big(\id_{\G_i}\var_i(f)\Big),
\]
where the final inequality holds because $A_i^{(1)} \cap A_i^{(2)}\subset \G_i$.

 \qed


\begin{bibdiv}
 \begin{biblist}



\bib{Aldous}{article}{
      author={Aldous, D.},
      author={Diaconis, P.},
       title={The asymmetric one-dimensional constrained {I}sing model:
  rigorous results},
        date={2002},
     journal={J. Stat. Phys.},
      volume={107},
      number={5-6},
       pages={945\ndash 975},
}

\bib{FH}{article}{
  title = {Kinetic Ising Model of the Glass Transition},
  author = {Andersen, Hans C.},
  author = {Fredrickson, Glenn H.},
  journal = {Phys. Rev. Lett.},
  volume = {53},
  number = {13},
  pages = {1244--1247},
  date = {1984},
}

\bib{DaiPra}{article}{
author = {Asselah, A.},
author={Dai Pra, P.},
title = {{Quasi-stationary measures for conservative dynamics in the infinite lattice}},
journal = {Ann. Prob.},
volume={29},
number={4},
pages={1733--1754},
year = {2001}
}




\bib{BBPS}{article}{
author={Balister, P.},
author={Bollob\'as, B.},
author={Przykucki, M.J.},
author={Smith, P.},
journal={Trans. Amer. Math. Soc.},
title={Subcritical $\cU$-bootstrap percolation models have non-trivial phase transitions},
pages={7385--7411},
volume={368},
year={2016},
}



\bib{BDMS}{article}{
author= {B.~Bollob\'as}, 
author= {H.~Duminil-Copin}, 
author= {R.~Morris}, 
author= {P.~Smith}, 
title={Universality of two-dimensional critical cellular automata}, 
journal={to appear in \emph{Proc. London Math. Soc.}},
year={2016},
eprint = {arXiv.org:1406.6680},
}

\bib{BCMS-Duarte}{article}{
	Author = {Bollob\'as, Bela},
author= {Duminil-Copin, Hugo},
author= {Morris, Robert},
author= {Smith, Paul},
		Title = {{The sharp threshold for the Duarte model}},
journal = {Ann. Prob.},
year = {2017},
volume = {45},
pages = {4222--4272}
}














\bib{BCMRT}{article}{
   author={Blondel, O.},
   author={Cancrini, N.},
   author={Martinelli, F.},
   author={Roberto, C.},
   author={Toninelli, C.},
   title={Fredrickson-Andersen one spin facilitated model out of
   equilibrium},
   journal={Markov Proc. Rel. Fields},
   volume={19},
   date={2013},
   pages={383--406},
}


\bib{BSU}{article}{
title={Monotone cellular automata in a random environment},
author={Bollob\'as, B.},
author={Smith, P.},
author = {Uzzell, A.},
journal={Combin. Probab. Comput.},
volume={24},
year={2015},
number={4},
pages={687--722},
}




	
















\bib{Praga}{article}{
   author={Cancrini, N.},
   author={Martinelli, F.},
   author={Roberto, C.},
   author={Toninelli, C.},
   title={Facilitated spin models: recent and new results},
   conference={
      title={Methods of contemporary mathematical statistical physics},
   },
   book={
      series={Lecture Notes in Math.},
      volume={1970},
      publisher={Springer},
      place={Berlin},
   },
   date={2009},
   pages={307--340},
}







\bib{CMRT}{article}{
      author={Cancrini, N.},
      author={Martinelli, F.},
      author={Roberto, C.},
      author={Toninelli, C.},
       title={Kinetically constrained spin models},
        date={2008},
     journal={Prob. Theory Rel. Fields},
      volume={140},
      number={3-4},
       pages={459\ndash 504},
  url={http://www.ams.org/mathscinet/search/publications.html?pg1=MR&s1=MR2365481},
}

\bib{CMST}{article}{
      author={Cancrini, N.},
      author={Martinelli, F.},
      author={Schonmann, R.},
      author={Toninelli, C.},
       title={Facilitated oriented spin models: some non equilibrium results},
        date={2010},
        ISSN={0022-4715},
     journal={J. Stat. Phys.},
      volume={138},
      number={6},
       pages={1109\ndash 1123},
         url={http://dx.doi.org/10.1007/s10955-010-9923-x},
}




\bib{CLR}{article}{
author={Chalupa, J.},
author={Leath, P.L.},
author={Reich, G.R.}, 
title={Bootstrap percolation on a Bethe latice}, 
journal={J. Physics C}, 
volume={12},
pages={L31--L35},
year={1979},
}


\bib{CFM3}{article}{
      author={Chleboun, Paul},
      author={Faggionato, Alessandra},
      author={Martinelli, Fabio},
       title={Mixing time and local exponential ergodicity of the
         East-like process in {$\bbZ^d$}},
       year ={2015},
     journal={Annales de la Facult{\'e} des Sciences de Toulouse:
       Math\'ematiques, S\'erie 6,},
     volume={24},
number={4},
       pages={717--743},
}


\bib{CFM}{article}{
      author={Chleboun, Paul},
      author={Faggionato, Alessandra},
      author={Martinelli, Fabio},
       title={{Time scale separation and dynamic heterogeneity in the low
  temperature East model}},
       year ={2014},
     journal={Commun. Math. Phys. },
     volume={328},
       pages={955-993},
}




\bib{CFM2}{article}{
  author={Chleboun, Paul},
author ={Faggionato, Alessandra},
author={Martinelli, Fabio},
  title={Relaxation to equilibrium of generalized East processes on $\bbZ^d$: Renormalisation group analysis and energy-entropy competition},
journal = {Ann. Prob.},
volume={44},
number={3},
pages={1817--1863},
year = {2016}
}



 \bib{CDG}{article}{
       author={Chung, F.},
       author={Diaconis, P.},
       author={Graham, R.},
        title={Combinatorics for the East model},
         date={2001},
      journal={Adv. Appl. Math.},
       volume={27},
       number={1},
        pages={192\ndash 206},
   url={http://www.ams.org/mathscinet/search/publications.html?pg1=MR&s1=MR1835679},
 }






\bib{Duarte}{article}{
	Author = {J.A.M.S.~Duarte},
	Journal = {Phys. A.},
	Number = {3},
	Pages = {1075--1079},
	Title = {{Simulation of a cellular automaton with an oriented bootstrap rule}},
	Volume = {157},
	Year = {1989}}


\bib{DC-Enter}{article}{
author = {Duminil-Copin, Hugo},
author={A.C.D.~van Enter},
title = {{Sharp metastability threshold for an anisotropic bootstrap percolation model}},
journal = {Ann. Prob.},
year = {2013},
volume = {41},
pages = {1218--1242},
}

\bib{DPEH}{article}{
author = {Duminil-Copin, Hugo},
author={A.C.D~van Enter},
author = {Hulshof, Tim},
title={{Higher order corrections for anisotropic bootstrap percolation}},
year = {2016},
eprint = {arXiv.org:1611.03294},
}

\bib{vanEnter}{article}{
author={A.C.D.~van Enter},
title={Proof of Straley's argument for bootstrap percolation},
date={1987},
journal={
J. Stat. Phys.},
volume={48},
 pages={943\ndash 945}}




	
\bib{East-review}{article}{
      author={Faggionato, Alessandra},
      author={Martinelli, Fabio},
      author={Roberto, Cyril},
      author={Toninelli, Cristina},
       title={{The {E}ast model: recent results and new progresses}},
        date={2013},
     journal={Markov Proc. Rel. Fields},
volume={19},
pages={407--458},
}




\bib{FMRT-cmp}{article}{
      author={Faggionato, A.},
      author={Martinelli, F.},
      author={Roberto, C.},
      author={Toninelli, C.},
       title={Aging through hierarchical coalescence in the East model},
        date={2012},
        ISSN={0010-3616},
     journal={Commun. Math. Phys.},
      volume={309},
       pages={459\ndash 495},
         url={http://dx.doi.org/10.1007/s00220-011-1376-9},
}











\bib{GarrahanSollichToninelli}{article}{
      author={J.~P.~Garrahan},
      author={Sollich, P.},
      author={Toninelli, C.},
       title={Kinetically constrained models},
        date={2011},
     journal={in ``Dynamical heterogeneities in glasses, colloids, and granular
  media" (Eds.: L. Berthier, G. Biroli, J.-P. Bouchaud, L. Cipelletti and W. van Saarloos), Oxford Univ. Press},
}
















\bib{JACKLE}{article}{
      author={J\"{a}ckle, J.},
      author={Eisinger, S.},
       title={A hierarchically constrained kinetic {I}sing model},
        date={1991},
     journal={Z. Phys. B: Condensed Matter},
      volume={84},
      number={1},
       pages={115\ndash 124},
}







 \bib{Levin-2008}{book}{
      author={D.~A.~Levin},
      author={Peres, Y.},
      author={E.~L.~Wilmer},
       title={{M}arkov chains and mixing times},
   publisher={American Mathematical Society},
        date={2008},
}

\bib{Liggett}{book}{
      author={Liggett, T.M.},
       title={Interacting particle systems},
   publisher={Springer-Verlag},
     address={New York},
        date={1985},
}




\bib{MT}{article}{
author={Martinelli,Fabio},
author={Toninelli, Cristina},
title={Towards a universality picture for the relaxation to equilibrium of kinetically constrained models},
journal={to appear in Ann. Prob.},
date={2018},
 eprint={arXiv.org:1701.00107},
}

\bib{MMT}{article}{
author = {{Mar{\^e}ch\'e}, Laure},
author={Martinelli,Fabio},
author={Toninelli, Cristina},
title={Energy barriers and the infection time for the kinetically
  constrained Duarte model},
year={in preparation}, 
}

\bib{Mountford}{article}{
	Author = {Mountford, T.S.},
	Journal = {Stoch. Proc. Appl.},
	Number = {2},
	Pages = {185--205},
	Title = {{Critical length for semi-oriented bootstrap percolation}},
	Volume = {56},
	Year = {1995}}



\bib{Robsurvey}{article}{
author = {Morris, Robert},
title = {{Bootstrap percolation, and other automata 
}},
journal={Europ. J. Combin.},
Pages = {250--263},
	Volume = {66},
	Year = {2017}
}







\bib{Ritort}{article}{
      author={Ritort, F.},
      author={Sollich, P.},
       title={Glassy dynamics of kinetically constrained models},
        date={2003},
     journal={Adv. Physics},
      volume={52},
      number={4},
       pages={219\ndash 342},
}

\bib{Saloff}{book}{
      author={Saloff-Coste, Laurent},
      editor={Bernard, Pierre},
       title={Lectures on finite {M}arkov chains},
      series={Lecture Notes in Mathematics},
   publisher={Springer Berlin Heidelberg},
        date={1997},
      volume={1665},
        ISBN={978-3-540-63190-3},
         url={http://dx.doi.org/10.1007/BFb0092621},
}

\bib{BPd}{article}{
author={Schonmann, R.},
title={On the behaviour of some cellular automata related
to bootstrap percolation},
date={1992},
journal={ Ann. Prob.},
volume={20}, 
pages={174\ndash 193}}
\bib{SE2}{article}{
      author={Sollich, P.},
      author={Evans, M.R.},
       title={Glassy time-scale divergence and anomalous coarsening in a
  kinetically constrained spin chain},
        date={1999},
     journal={Phys. Rev. Lett},
      volume={83},
       pages={3238\ndash 3241},
}



\bib{AS}{article}{
author={Pillai, N.S.},
author={Smith, A.},
date= {2017},
title = {{Mixing times for a constrained Ising process on the torus at low density}},
journal={Ann. Prob.},
volume={45},
number={2},
pages={1003--1070},
}


 \end{biblist}
 \end{bibdiv}
\end{document}